\documentclass[10pt]{stacks-project2}
\usepackage{latexsym, amscd, amsfonts, eucal, mathrsfs, amsmath, amssymb, amsthm, tikz-cd, xr, stmaryrd, amsxtra, MnSymbol}
\usepackage[alphabetic]{amsrefs}
\usepackage{enumitem}
\usepackage{adjustbox}
\usepackage[hidelinks]{hyperref}
\usepackage{chngcntr}

\counterwithin{equation}{subsection}
\newtheorem{theorem}[equation]{Theorem}
\newtheorem{lemma}[equation]{Lemma}

\newtheorem{proposition}[equation]{Proposition}
\newtheorem{corollary}[equation]{Corollary}
\newtheorem{conjecture}[equation]{Conjecture}
\newtheorem*{theoremA}{Theorem}
\newtheorem*{claim}{Claim}

\theoremstyle{definition}
\newtheorem{sub}[equation]{}

\newtheorem{remark}[equation]{Remark}
\newtheorem{hypothesis}[equation]{Hypothesis}

\newtheorem{definition}[equation]{Definition}
\newtheorem{example}[equation]{Example}

\title[Irreducible components of some crystalline deformation rings]
{On the irreducible components of some crystalline deformation rings}
\author{Robin Bartlett} 
\address{Max Planck Institute for Mathematics, Bonn}
\email{robinbartlett18@mpim-bonn.mpg.de}
\begin{document}
\maketitle
\begin{abstract}
	We adapt a technique of Kisin to construct and study crystalline deformation rings of $G_K$ for a finite extension $K/\mathbb{Q}_p$. This is done by considering a moduli space of Breuil--Kisin modules, satisfying an additional Galois condition, over the unrestricted deformation ring. For $K$ unramified over $\mathbb{Q}_p$ and Hodge--Tate weights in $[0,p]$, we study the geometry of this space. As a consequence we prove that, under a mild cyclotomic-freeness assumption, all crystalline representations of an unramified extension of $\mathbb{Q}_p$, with Hodge--Tate weights in $[0,p]$, are potentially diagonalisable.
\end{abstract}
\tableofcontents
\section{Introduction}

Let $K/\mathbb{Q}_p$ be a finite extension, let $\mathbb{F}$ denote a finite field of characteristic $p$ and let $V_{\mathbb{F}}$ denote a continuous representation of $G_K = \operatorname{Gal}(\overline{K}/K)$ on a finite dimensional $\mathbb{F}$-vector space. In \cite{Kis08} Kisin constructs a quotient of the universal framed deformation ring of $V_{\mathbb{F}}$, parametrising deformations which are crystalline (even potentially semistable) with fixed Hodge--Tate weights.

The motivation for studying such deformation rings comes from the conjectures of Fontaine--Mazur \cite{FM}, and the desire to prove that many Galois representations arise from modular forms. The first considerable progress towards these questions came with the modularity lifting theorem of Wiles \cite{Wiles}, where crucial use was made of the fact that the deformation rings parametrising crystalline representations of $G_{\mathbb{Q}_p}$ with Hodge--Tate weights either $0$ or $1$ are power series rings, and so as simple as one could hope for. It was later shown by Kisin \cite{KisFF} that in fact one could proceed in more general situations where the deformation rings were less well-behaved, provided one could maintain some control on their irreducible components. More recently still these ideas have coalesced into the notion of potential diagonalisability. Roughly speaking a crystalline representation is potentially diagonalisable if it lies on the same irreducible component as a particularly simple representation (e.g. a direct sum of characters). This condition was introduced in \cite{BLGGT} where a general modularity lifting theorem was proved under the local assumption of potential diagonalisability at primes above $p$.

The aim of this paper is to extend the methods introduced by Kisin in \cite{KisFF} to treat Hodge--Tate weights beyond the range $[0,1]$. In particular we are able to prove the following.

\begin{theoremA}
	Suppose $K$ is unramified over $\mathbb{Q}_p$ and that $V$ is a crystalline representation of $G_K$ on a finite dimensional $\overline{\mathbb{Q}}_p$-vector space, with Hodge--Tate weights contained in $[0,p]$. Assume the mod~$p$ semi-simplification $\overline{V}$ of $V$ is strongly cyclotomic-free\footnote{See Definition~\ref{cyclofree}.}. Then $V$ is potentially diagonalisable.
\end{theoremA}

There are two new aspects of this result. The first is the range $[0,p]$ which we allow our Hodge--Tate weights to vary over; for unramified extensions of $\mathbb{Q}_p$ potential diagonalisability was previously only known for Hodge--Tate weights in the range $[0,p-1]$. See the work of Gao--Liu \cite{GL14}. The second is with our method, which opens the possibility of proving potential diagonalisability when $K$ is a possibly ramified extension of $\mathbb{Q}_p$ (where potential diagonalisability is only known for two dimensional representations with Hodge--Tate weights between $0$ and $1$, cf. \cite[3.4.1]{GK15}). While we are unable to treat such $K$ in this paper we hope our methods will be useful in the more general situation.

As already mentioned, our starting point is with ideas originally employed by Kisin. In \cite{Kis06} Kisin identifies a collection of $G_K$-representations on finite free $\mathbb{Z}_p$-modules, those of finite $E$-height. This condition depends only upon the restriction of the representation to $G_{K_\infty}$ where $K_\infty = K(\pi^{1/p^\infty})$ for some choice of uniformiser $\pi \in K$: a representation is of finite $E$-height if the etale $\varphi$-module associated to its restriction to $G_{K_\infty}$ admits a particular kind of lattice, what is now known as a Breuil--Kisin module. Kisin proves that any $\mathbb{Z}_p$-lattice inside a crystalline (even semi-stable) representation with Hodge--Tate weights in $[0,h]$ is of $E$-height $\leq h$. If $R^\square_{V_{\mathbb{F}}}$ denotes the universal framed deformation ring of $V_{\mathbb{F}}$ then the main construction of \cite{Kis08} uses these Breuil--Kisin modules to build a projective $R^\square_{V_{\mathbb{F}}}$-scheme $\mathcal{L}^{\leq h}$. The scheme-theoretic image of the morphism to $\operatorname{Spec}R^\square_{V_{\mathbb{F}}}$ corresponds to a quotient of $R^\square_{V_{\mathbb{F}}}$ parametrising deformations of $E$-height $\leq h$. As not every finite $E$-height representation is crystalline this quotient is, in general, too large. Instead it provides an approximation to the crystalline quotient from which the desired crystalline deformation ring is obtained as a further quotient. Unfortunately forming this second quotient requires inverting $p$, obscuring the integral structure of the deformation ring.

One exception is when $h =1$. In this case any representation of $E$-height $\leq 1$ is crystalline. The $E$-height $\leq 1$ quotient of $R^\square_{V_{\mathbb{F}}}$ is therefore precisely the crystalline quotient, and so the geometry of $\mathcal{L}^{\leq 1}$, which may be understood through its definition in terms of semi-linear algebra, can then be used to study the geometry of the crystalline deformation rings. This is the technique employed in \cite{KisFF}.

Our aim is to refine the construction of $\mathcal{L}^{\leq h}$ so that something similar happens for $h >1$. For this we use results of Gee--Liu--Savitt and Ozeki (see Theorem~\ref{GLS}) which provide necessary and sufficient conditions for any representation of finite $E$-height to be crystalline. In Section~\ref{crysdefrings} we show this condition cuts out a closed subscheme $\mathcal{L}^{\leq h}_{\operatorname{crys}}$ of  $\mathcal{L}^{\leq h}$. The scheme-theoretic image of the morphism $\mathcal{L}^{\leq h}_{\operatorname{crys}} \rightarrow \operatorname{Spec} R^\square_{V_{\mathbb{F}}}$ then corresponds to the quotient of $R^\square_{V_{\mathbb{F}}}$ parametrising crystalline deformations with Hodge--Tate weights $\leq h$ (at least up to $p$-power torsion).

In general we do not understand the geometry of $\mathcal{L}^{\leq h}_{\operatorname{crys}}$. It is a closed subset of an affine Grassmannian cut out by a seemingly complicated Galois condition. However, if we restrict to the case in which $K$ is unramified over $\mathbb{Q}_p$ and $h = p$, then we show that around the closed points of $\mathcal{L}^{\leq p}_{\operatorname{crys}}$ this Galois condition is in fact equivalent to a condition phrased solely in terms of semilinear algebra. This is the main result of Section~\ref{strongdivisibility}. In Section~\ref{localstructure} we use this to describe the local geometry of $\mathcal{L}^{\leq p}_{\operatorname{crys}}$. Provided $V_{\mathbb{F}}$ is cyclotomic-free (a slightly weaker condition than that of strongly cyclotomic-freeness appearing in the theorem above) we show that the $\mathbb{Z}_p$-flat locus $\mathcal{L}^\circ \subset \mathcal{L}^{\leq p}_{\operatorname{crys}}$ is open and the completed local rings at closed points of $\mathcal{L}^\circ$ are power series rings.

In the last section we use all this to deduce consequences for crystalline deformation rings. We prove the theorem above. We also make a conjecture (which we prove in the two-dimensional case) on connectedness of the fibre of $\mathcal{L}^\circ$ over the closed point of $\operatorname{Spec}R^\square_{V_{\mathbb{F}}}$ for irreducible $V_{\mathbb{F}}$. We then explain, assuming this conjecture, how the cyclotomic-freeness assumption from the above theorem can be weakened (but not removed). Finally we identify a good situation (when the fibre of $\mathcal{L}^\circ$ over the closed point of $\operatorname{Spec}R^{\square}_{V_{\mathbb{F}}}$ is zero-dimensional and reduced) in which $\mathcal{L}^\circ$ can be used to show that each irreducible component of crystalline deformation rings is formally smooth. We conclude by illustrating this with some concrete examples, recovering previous computations of \cite{KisFM} and \cite{Sand} in the case of two-dimensional representation of $G_{\mathbb{Q}_p}$. 
\subsection*{Acknowledgements} I would like to thank Frank Calegari, Mark Kisin, Tong Liu and James Newton for helpful conversations and correspondence. This work was done at the Max Planck Institute for Mathematics, in Bonn, and it is a pleasure to thank this institution for their support.
\section{Crystalline deformation rings}\label{crysdefrings}
\subsection{Integral $p$-adic Hodge theory}

\begin{sub}\label{notation}
Let $k$ be a finite field of characteristic $p$ and let $K_0 = W(k)[\frac{1}{p}]$. Fix $K$ a totally ramified extension of $K_0$ of degree $e$. Also fix a uniformiser $\pi$ of $K$ and a compatible system $\pi^{1/p^\infty}$ of $p$-th power roots of $\pi$ in an algebraic closure $\overline{K}$ of $K$. Let $E(u) \in W(k)[u]$ denote the minimal polynomial of $\pi$ over $K_0$.	
	
The ring $\mathfrak{S} = W(k)[[u]]$ is equipped with a Frobenius $\varphi$ which acts on $W(k)$ as the usual Witt vector Frobenius, and which sends $u \mapsto u^p$. This Frobenius extends uniquely to a Frobenius on $\mathcal{O}_{\mathcal{E}}$, the $p$-adic completion of $\mathfrak{S}[\frac{1}{u}]$, which we again denote by $\varphi$.

Let $C$ be the completion of $\overline{K}$ with integers $\mathcal{O}_C$. The inverse limit of the system
	$$
	\mathcal{O}_C/p \leftarrow \mathcal{O}_C/p \leftarrow \mathcal{O}_C/p \leftarrow \ldots
	$$
	(with transition maps given by $x \mapsto x^p$) is denoted $\mathcal{O}_{C^\flat}$. By construction the $p$-th power map on $\mathcal{O}_{C^\flat}$ is an automorphism. The obvious map $\varprojlim_{x \mapsto x^p} \mathcal{O}_C \rightarrow \mathcal{O}_{C^\flat}$ is a multiplicative bijection which allows us to equip $\mathcal{O}_{C^\flat}$ with a valuation $v^\flat$ as follows. If $v$ denotes the valuation on $\mathcal{O}_C$ normalised so that $v(p)= 1$ then $v^\flat(x) := v(x^\sharp)$ where $x^\sharp \in \mathcal{O}_C$ is the image of $x$ under the projection $\mathcal{O}_{C^\flat} = \varprojlim \mathcal{O}_C \rightarrow \mathcal{O}_C$ onto the first coordinate. This makes $\mathcal{O}_{C^\flat}$ into a complete valuation ring with field of fractions $C^\flat$. The continuous $G_K$-action on $\mathcal{O}_C$ induces continuous $G_K$-actions on $\mathcal{O}_{C^\flat}$ and $C^\flat$.

Let $A_{\operatorname{inf}}= W(\mathcal{O}_{C^\flat})$. By functoriality of the Witt vector construction the $G_K$-action on $\mathcal{O}_{C^\flat}$ transfers to a $G_K$-action on $A_{\operatorname{inf}}$. Likewise we obtain a $G_K$-action on $W(C^\flat)$. We also obtain Frobenius endomorphisms on $A_{\operatorname{inf}}$ and $W(C^\flat)$ lifting the $p$-th power maps on $\mathcal{O}_{C^\flat}$ and $C^\flat$. These endomorphisms commute with the $G_K$-actions.

The compatible system of $p$-th power roots of $\pi \in K$ gives rise to an element $\pi^\flat \in \mathcal{O}_{C^\flat}$ with $v^\flat(\pi^\flat) = 1/e$. The map of $W(k)$-algebras $\mathfrak{S} \rightarrow A_{\operatorname{inf}}$ sending $u \mapsto [\pi^\flat]$ (where $[\cdot]$ denotes the Teichmuller map) is an embedding compatible with the Frobenius on either ring. This map extends to a Frobenius compatible embedding $\mathcal{O}_{\mathcal{E}} \rightarrow W(C^\flat)$ where $\mathcal{O}_{\mathcal{E}}$ denotes the $p$-adic completion of $\mathfrak{S}[\frac{1}{u}]$.
\end{sub}
\begin{sub}\label{etale}
	Let $V$ be a finitely generated $\mathbb{Z}_p$-module equipped with a continuous $\mathbb{Z}_p$-linear action of $G_{K_\infty}$. The results of \cite{Fon00} assert that there exists a unique finitely generated $\mathcal{O}_{\mathcal{E}}$-submodule $M \subset V \otimes_{\mathbb{Z}_p} W(C^\flat)$ such that
	\begin{equation}\label{Acomparison}
	M \otimes_{\mathcal{O}_{\mathcal{E}}} W(C^\flat) = V \otimes_{\mathbb{Z}_p} W(C^\flat)
	\end{equation}
	and such that the $W(C^\flat)$-semilinear extension of the $G_{K_\infty}$-action on $V$ fixes $M$ and such that the restriction of the trivial $W(C^\flat)$-semilinear Frobenius on $V$ induces an isomorphism $M \otimes_{\mathcal{O}_{\mathcal{E}},\varphi} \mathcal{O}_{\mathcal{E}} =: \varphi^*M \xrightarrow{\sim} M$. The construction of $M$ is functorial in $V$. In particular if $V$ admits a $G_{K_\infty}$-equivariant $\mathbb{Z}_p$-linear action of a $\mathbb{Z}_p$-algebra $A$ then $M$ can be viewed as a module over $\mathcal{O}_{\mathcal{E},A} = \mathcal{O}_{\mathcal{E}} \otimes_{\mathbb{Z}_p} A$.
\end{sub}
\begin{definition}\label{Eheight}
	If $V$ is a finite free $\mathbb{Z}_p$-module equipped with a continuous $\mathbb{Z}_p$-linear action of $G_{K}$, and if $M$ is associated to $V|_{G_{K_\infty}}$ as in \ref{etale}, then $V$ has $E$-height $\leq h$ if there exists a $\varphi$-stable finite free $\mathfrak{S}$-submodule $\mathfrak{M} \subset M$ such that (i) the induced map $\varphi^*\mathfrak{M}= \mathfrak{M} \otimes_{\mathfrak{S},\varphi} \mathfrak{S} \rightarrow \mathfrak{M}$ has cokernel killed by $E(u)^h$ and (ii) there is an equality $\mathfrak{M} \otimes_{\mathfrak{S}} \mathcal{O}_{\mathcal{E}} = M$. 
\end{definition}
	The association of $\mathfrak{M}$ to a representation of $E$-height $\leq h$ is a fully faithful functor, cf. \cite[2.1.12]{Kis06}. In particular there exists at most one $\mathfrak{M} \subset M$ as above; we call this the \emph{Breuil--Kisin module} associated to $V$.

\begin{sub}
	As usual let $\mathbb{Z}_p(1)$ denote the free rank one $\mathbb{Z}_p$-module consisting of compatible systems of $p$-th power roots of unity in $\overline{K}$. Consider the ring of $p$-adic periods $B_{\operatorname{dR}}^+$ defined in \cite{Fon94}. There is a homomorphism $\mathbb{Z}_p(1) \rightarrow B_{\operatorname{dR}}^+$ sending $\xi \mapsto \operatorname{log}([\xi]) := \sum_{n \geq 1} (-1)^{n+1} \frac{([\xi] -1)^n}{n}$. Fix a $\mathbb{Z}_p$-generator $\epsilon$ of $\mathbb{Z}_p(1)$ and set $t = \operatorname{log}([\epsilon])$. We also write $\mu = [\epsilon] -1 \in A_{\operatorname{inf}}$.
	
	As in \cite[III.1]{Col98} let $A_{\operatorname{max}} \subset B_{\operatorname{dR}}^+$ be the subring of elements which can be written as $\sum_{n \geq 0} x_n (\frac{[\pi^\flat]^{en}}{p^n}) \in B_{\operatorname{dR}}^+$ with $(x_n)_{n \geq 0}$ a sequence in $A_{\operatorname{inf}}$ converging $p$-adically to zero. Then $B_{\operatorname{max}}^+ = A_{\operatorname{max}}[\frac{1}{p}]$ and $B_{\operatorname{max}} = B_{\operatorname{max}}^+[\frac{1}{t}]$. The Frobenius on $A_{\operatorname{inf}}$ extends to each of these rings. 
	
	If $V$ is a finite free $\mathbb{Z}_p$-module equipped with a continuous $\mathbb{Z}_p$-linear action of $G_K$ then $V \otimes_{\mathbb{Z}_p} \mathbb{Q}_p$ is crystalline if and only if the $K_0$-vector space $D_{\operatorname{crys}}(V) := (V \otimes_{\mathbb{Z}_p} B_{\operatorname{max}})^{G_K}$ has dimension equal to $ \operatorname{rank}_{\mathbb{Z}_p} V$.\footnote{The usual definition of a crystalline representation is made using the period ring $B_{\operatorname{crys}}$. If $A_{\operatorname{crys}} \subset B_{\operatorname{dR}}$ consists of elements of the form $\sum_{n \geq 0} x_n \frac{[\pi^\flat]^{en}}{n!}$ with $x_n \in A_{\operatorname{inf}}$ converging $p$-adically to zero, then $B_{\operatorname{crys}}^+ = A_{\operatorname{crys}}[\frac{1}{p}]$ and $B_{\operatorname{crys}} = B_{\operatorname{crys}}^+[\frac{1}{t}]$. Since $v_p(n!) \leq n$ and $n \leq v_p((pn)!)$ we see that $\varphi(A_{\operatorname{max}}) \subset A_{\operatorname{crys}} \subset A_{\operatorname{max}}$. Thus $\varphi(B_{\operatorname{max}}) \subset B_{\operatorname{crys}} \subset B_{\operatorname{max}}$. Using this one see that $(V \otimes_{\mathbb{Z}_p} B_{\operatorname{max}})^{G_K} = (V \otimes_{\mathbb{Z}_p} B_{\operatorname{crys}})^{G_K}$.}
\end{sub}
\begin{sub}\label{identifications}
	Suppose $V$ is a representation of $E$-height $\leq h$ with corresponding Breuil--Kisin module $\mathfrak{M}$. Set $D = (\mathfrak{M}/u\mathfrak{M}) \otimes_{W(k)} K_0$. This is a $K_0$-vector space equipped with an bijective Frobenius $\varphi^*D \xrightarrow{\sim} D$. We claim there exists $\varphi,G_{K_\infty}$-equivariant identifications
	\begin{equation}\label{biggest}
	D \otimes_{K_0} B_{\operatorname{max}} \cong \mathfrak{M} \otimes_{\mathfrak{S}} B_{\operatorname{max}} \cong V \otimes_{\mathbb{Z}_p} B_{\operatorname{max}}
	\end{equation}
	where $G_{K_\infty}$ is made to act trivially on $D$. The right-hand identification follows from the next lemma since $A_{\operatorname{inf}}[\frac{1}{\mu}]$ is a subring of $B_{\operatorname{max}}$.
	
	\begin{lemma}\label{1/mu}
		Let $V$ be a representation of $E$-height $\leq h$ and $\mathfrak{M}$ the corresponding Breuil--Kisin module. Write $\mathfrak{M}^\varphi$ for the image of $\varphi^*\mathfrak{M} \rightarrow \mathfrak{M}$. Then there exists a $\varphi,G_{K_\infty}$-equivariant identification
		$$
		\mathfrak{M}^{\varphi} \otimes_{\mathfrak{S}} A_{\operatorname{inf}}[\tfrac{1}{\mu}] \cong V \otimes_{\mathfrak{S}} A_{\operatorname{inf}}[\tfrac{1}{\mu}]
		$$
		which recovers \eqref{Acomparison} after tensoring with  $W(C^\flat)$.
	\end{lemma} 
	\begin{proof}
		This follows by applying \cite[Lemma 4.26]{BMS} to the Breuil--Kisin--Fargues module $\mathfrak{M}^\varphi \otimes_{\mathfrak{S}} A_{\operatorname{inf}} = \varphi(\mathfrak{M}) \otimes_{\varphi(\mathfrak{S})} A_{\operatorname{inf}}$.\footnote{Note that in \emph{loc. cit.} $\mathfrak{S}$ is viewed as a subring of $A_{\operatorname{inf}}$ via $u \mapsto [\pi^\flat]^p$, which is different to our embedding. This is the reason why $\mathfrak{M}^\varphi$ appears rather than $\mathfrak{M}$.}
	\end{proof}
	
	For the left-hand side of \eqref{biggest}, let $\mathcal{O}^{\operatorname{rig}} \subset K_0[[u]]$ denote the subring of power series converging on the open unit disk, and consider $\lambda = \prod_{n=0}^\infty \varphi^n(\frac{E(u)}{E(0)}) \in \mathcal{O}^{\operatorname{rig}}$. In \cite[1.2.6]{Kis06} a $\varphi$-equivariant inclusion
	\begin{equation}\label{ident1}
	D \otimes_{K_0} \mathcal{O}^{\operatorname{rig}} \hookrightarrow \mathfrak{M}^{\varphi} \otimes_{\mathfrak{S}} \mathcal{O}^{\operatorname{rig}}
	\end{equation}
	is constructed which is an isomorphism modulo~$u$ and which becomes an isomorphism after inverting $\varphi(\lambda)$. It is also $G_{K_\infty}$-equivariant, for the trivial $G_{K_\infty}$-action on both sides. Since the inclusion $\mathfrak{S} \rightarrow A_{\operatorname{inf}}$ extends to an embedding $\mathcal{O}^{\operatorname{rig}} \rightarrow B^+_{\operatorname{max}}$, which maps $\varphi(\lambda)$ onto a unit in $B_{\operatorname{max}}^+$, we obtain the left-hand side of \eqref{biggest}. 
\end{sub}
We can now formulate the main result of \cite{Kis06}. See also \cite[1.2.1]{Kis10}.
\begin{proposition}[Kisin]\label{kis}
	If $V$ is a $G_K$-stable $\mathbb{Z}_p$-lattice inside a crystalline representation with Hodge--Tate weights\footnote{Our Hodge--Tate weights are normalised so that the cyclotomic character has weight $-1$.} in $[0,h]$ then $V$ is of $E$-height $\leq h$. Furthermore:
	\begin{enumerate}
		\item $D_{\operatorname{crys}}(V) \subset V \otimes_{\mathbb{Z}_p} B_{\operatorname{max}}$ is identified with $D$ under \eqref{biggest}.
		\item Tensoring \eqref{ident1} with the map $\mathcal{O}^{\operatorname{rig}} \rightarrow K$ given by $u \mapsto \pi$ identifies $\mathfrak{M}^\varphi /E(u)\mathfrak{M}^\varphi$ with an $\mathcal{O}_K$-lattice inside $D_{\operatorname{crys}}(V)_K = D_{\operatorname{crys}}(V) \otimes_{K_0} K$. Via the inclusion $B_{\operatorname{max}} \otimes_{K_0} K \rightarrow B_{\operatorname{dR}}$ we identify $D_{\operatorname{crys}}(V)_K=(V \otimes_{\mathbb{Q}_p} B_{\operatorname{dR}})^{G_K}$ and, under this identification, the surjection $\mathfrak{M}^\varphi \rightarrow D_{\operatorname{crys}}(V)_K$ induces a map of $\mathfrak{M}^\varphi \cap E(u)^i\mathfrak{M}$ into 
		$$
		F^i D_{\operatorname{crys}}(V)_K := (V \otimes_{\mathbb{Q}_p} t^iB_{\operatorname{dR}}^+)^{G_K}
		$$ 
		which becomes surjective after inverting $p$.
	\end{enumerate} 
\end{proposition}

	Not every finite $E$-height representation is crystalline; indeed in \cite[1.1.13]{Gao19} it is shown that $V$ has finite $E$-height if and only if $V|_{G_{K_m}}$ is semi-stable where $K_m = K(\pi^{1/p^m})$ for a suitably large $m$. The starting point of this article is a description identifying which finite $E$-height representations are crystalline. To explain this fix a representation $V$ of finite $E$-height with associated Breuil--Kisin module $\mathfrak{M}$. Using Lemma~\ref{1/mu}, or simply \eqref{Acomparison}, we obtain a $\varphi,G_{K_\infty}$-equivariant identification
	\begin{equation}\label{galeq}
	\mathfrak{M} \otimes_{\mathfrak{S}} W(C^\flat) \cong V \otimes_{\mathbb{Z}_p} W(C^\flat)
	\end{equation}
	The $G_K$-actions on $V$ and $W(C^\flat)$ therefore transfer to a $\varphi$-equivariant $G_K$-action on $\mathfrak{M} \otimes_{\mathfrak{S}} W(C^\flat)$.
	\begin{theorem}[Gee--Liu--Savitt, Ozeki]\label{GLS}
		Let $V$ be a finite free $\mathbb{Z}_p$-module with a continuous $\mathbb{Z}_p$-linear action of $G_K$. Then the following are equivalent:
		\begin{enumerate}
			\item $V \otimes_{\mathbb{Z}_p} \mathbb{Q}_p$ is crystalline with Hodge--Tate weights in $[0,h]$.
			\item $V$ is of $E$-height $\leq h$ and the $G_K$-action on $\mathfrak{M} \otimes_{\mathfrak{S}} W(C^\flat)$ induced from \eqref{galeq} is such that $(\sigma-1)(m) \in \mathfrak{M} \otimes_{\mathfrak{S}} [\pi^\flat]\varphi^{-1}(\mu)A_{\operatorname{inf}}$ for every $m \in \mathfrak{M}$ and $\sigma \in G_K$.
		\end{enumerate}
	\end{theorem}

	That (1) implies (2) is essentially \cite[4.10]{GLS}, while the converse is proven in \cite[Theorem 21]{Ozeki}. As both these results are not formulated as we need (and also because they assume that $p>2$) we devote the rest of this section to a proof of the theorem. Our argument that (1) implies (2) is essentially the same as that in \cite{GLS}, but our proof of the converse differs from Ozeki's.

\begin{proof}[Proof that (2) implies (1) in Theorem~\ref{GLS}]
	In fact we prove something stronger. Namely consider $V$ and $\mathfrak{M}$ as in \ref{identifications} and suppose the $G_K$-action on $V$ is such that, when transferred to $\mathfrak{M} \otimes_{\mathfrak{S}} B_{\operatorname{max}}$ via \eqref{biggest},
	\begin{equation}\label{thirdcondition}
	(\sigma-1)(m) \in \mathfrak{M} \otimes_{\mathfrak{S}} [\pi^\flat]A_{\operatorname{inf}}
	\end{equation}
	for every $m \in \mathfrak{M}$ and $\sigma \in G_K$. Then we show $V \otimes_{\mathbb{Z}_p} \mathbb{Q}_p$ is crystalline. For this it suffices to show the $G_K$-action is trivial on $D$. To this end let $S_{\operatorname{max}} \subset \mathcal{O}^{\operatorname{rig}}$ denote the subring $W(k)[[u,\frac{u^{e}}{p}]] \cap \mathcal{O}^{\operatorname{rig}}$. Clearly the inclusion $\mathcal{O}^{\operatorname{rig}} \rightarrow B_{\operatorname{max}}^+$ maps $S_{\operatorname{max}}$ into $A_{\operatorname{max}}$. Recall that a power series $\sum a_i u^i$ with $a_i \in K_0$ lies in $\mathcal{O}^{\operatorname{rig}}$ if and only if $v_p(a_i) + ir \rightarrow 0$ for any $r > 0$. This series is contained in $S_{\operatorname{max}}$ if furthermore $v(a_i) +  i/e \geq 0$. By taking $r = \frac{1}{e}$ we see $S_{\operatorname{max}}[\frac{1}{p}] = \mathcal{O}^{\operatorname{rig}}$, and so we can choose a $W(k)$-lattice $D^\circ \subset D$ so that every $d \in D^\circ$ can be written as $\sum s_i m_i$ with $s_i \in S_{\operatorname{max}}$ and $m_i \in \mathfrak{M}$. If $s \in S_{\operatorname{max}}$ and $\sigma \in G_K$ then $(\sigma-1)(s) \in \frac{[\pi^\flat]}{p} A_{\operatorname{max}}$ since $G_K$ acts trivially on the constant term. From this and \eqref{thirdcondition} we deduce that 
	$$
	(\sigma -1)(d) = \sum (\sigma(s_i) - s_i)\sigma(m_i) + \sum s_i(\sigma(m_i) - m_i)  \in \mathfrak{M} \otimes_{\mathfrak{S}} \tfrac{[\pi^\flat]}{p}A_{\operatorname{max}}
	$$
	for any $d = \sum s_i m_i \in D^\circ$. There exists an $m \in \mathbb{Z}$ such that $\varphi^{-1}(D^\circ) \subset \frac{1}{p^m}D^\circ$. Thus $\varphi^{-n}(D^\circ) \subset \frac{1}{p^{nm}} D^\circ$ for $n \geq 1$. Since the $G_K$-action is $\varphi$-equivariant we have
	$$
	(\sigma-1)(d) = \varphi^n\left( (\sigma-1)(\varphi^{-n}(d)) \right) \in \varphi^n ( \mathfrak{M} \otimes_{\mathfrak{S}} \tfrac{[\pi^\flat]}{p^{nm+1}} A_{\operatorname{max}} ) \subset \mathfrak{M} \otimes_{\mathfrak{S}} \tfrac{[\pi^\flat]^{p^n}}{p^{nm+1}} A_{\operatorname{max}}
	$$
	whenever $d \in D^\circ$. However $\frac{[\pi^\flat]^{p^n}}{p^{nm+1}} \in p^{p^n-nm-1} A_{\operatorname{max}}$ and so, since $A_{\operatorname{max}}$ is $p$-adically complete, it must be that $(\sigma-1)(d) = 0$.
\end{proof}

Now we show (1) implies (2). One of the advantages of using $B_{\operatorname{max}}^+$ is that its topology is better behaved than that of $B_{\operatorname{crys}}^+$. In particular we have:

\begin{lemma}\label{topology}
	Equip $B^+_{\operatorname{max}}$ with the topology making $(p^nA_{\operatorname{max}})_{n \geq 0}$ a basis of open neighbourhoods of $0$. Then $B_{\operatorname{max}}^+$ is complete and any principal ideal $aB_{\operatorname{max}}^+ \subset B_{\operatorname{max}}^+$ is closed.
\end{lemma}	
\begin{proof}
	Completeness is immediate since $A_{\operatorname{max}}$ is $p$-adically complete. To check $aB_{\operatorname{max}}^+$ is closed consider a sequence $b_i \in aB_{\operatorname{max}}^+$ converging to $b \in B^+_{\operatorname{max}}$. We must show $b \in aB_{\operatorname{max}}^+$. Since $B_{\operatorname{max}}^+$ is a domain it suffices to show $\frac{b_i}{a}$ converges in $B_{\operatorname{max}}^+$. This follows from \cite[III.2.1]{Col98} which asserts that if $|| x|| = \operatorname{inf}_{n \mid p^nx \in A_{\operatorname{max}}} p^n$ then $p^{-1}||x||||y|| \leq ||xy||$. Hence $|| \frac{b_i}{a} - \frac{b_j}{a}|| \leq \frac{p}{||a||} ||b_i - b_j||$, and so as $B^+_{\operatorname{max}}$ is complete $\frac{b_i}{a}$ converges.
\end{proof}

For $\sigma \in G_K$ consider $\epsilon(\sigma) \in \mathbb{Z}_p(1)$ defined by $\epsilon(\sigma)_n = \sigma(\pi^{1/p^n})/\pi^{1/p^n}$.
\begin{lemma}\label{monoGal}
	Suppose $V \otimes_{\mathbb{Z}_p} \mathbb{Q}_p$ is crystalline and that $\mathfrak{M}$ is the Breuil--Kisin module associated to $V$. Define a  differential operator $\mathcal{N}$ over $\partial = u\frac{d}{du}$ on $\mathfrak{M} \otimes_{\mathfrak{S}} \mathcal{O}^{\operatorname{rig}}[\frac{1}{\lambda}] = D \otimes_{K_0} \mathcal{O}^{\operatorname{rig}}[\frac{1}{\lambda}]$ by asserting $\mathcal{N}(d) = 0$ for all $d \in D$. Then 
	$$
	\sigma(m) = \sum_{n \geq 0} \mathcal{N}^n(m) \otimes \frac{(- \operatorname{log}([\epsilon(\sigma)]))^n}{n!}
	$$
	for $m \in \mathfrak{M}^\varphi \otimes_{\mathfrak{S}} \mathcal{O}^{\operatorname{rig}}[\frac{1}{\varphi(\lambda)}]$ and $\sigma \in G_K$.
\end{lemma}
\begin{proof}
	Since $\mathfrak{M}^\varphi \otimes_{\mathfrak{S}} \mathcal{O}^{\operatorname{rig}}[\frac{1}{\varphi(\lambda)}] = D \otimes_{K_0} \mathcal{O}^{\operatorname{rig}}[\frac{1}{\varphi(\lambda)}]$ it is enough to consider $m = fd$ with $d \in D$ and $f \in \mathcal{O}^{\operatorname{rig}}[\frac{1}{\varphi(\lambda)}]$. By definition $\mathcal{N}^n(fd) = \partial^n(f)d$. By (1) of Proposition~\ref{kis} we identify $D = D_{\operatorname{crys}}(V)$ and the $G_K$-action on $V$ fixes $D$; hence $\sigma(fd) = \sigma(f)d$. The lemma therefore reduces to checking that 
	\begin{equation}\label{logexp}
	\sum_{n \geq 0} \frac{(-\operatorname{log}([\epsilon(\sigma)]))^n}{n!} \partial^n(f)
	\end{equation}
	converges in $B_{\operatorname{max}}^+$ to $\sigma(f)$. It suffices to consider $f = u^i$. Then $\sigma(f) = [\epsilon(\sigma)^i] u^i$. On the other hand, using that $\partial^n(u^i) = (-i)^n u^i$, we see \eqref{logexp} equals $\operatorname{exp}(\operatorname{log}([\epsilon(\sigma)]^i)) u^i$. If this sum converges it will do so to $[\epsilon(\sigma)]^iu^i$ which proves the lemma.
	
	To show convergence it is enough to show $\frac{\operatorname{log}([\epsilon(\sigma)])^n}{n!}$ lies in $A_{\operatorname{max}}$ and in this ring converges $p$-adically to zero. Note that $\operatorname{log}([\epsilon(\sigma)]) = \alpha t$  for some $\alpha \in \mathbb{Z}_p$. The proof of \cite[III.3.9]{Col98} shows that $t \in pA_{\operatorname{max}}$ if $p>2$ and $t \in p^2A_{\operatorname{max}}$ if $p =2$. Convergence of $\frac{\operatorname{log}([\epsilon(\sigma)])^n}{n!}$ then follows because $\frac{p^n}{n!} \in \mathbb{Z}_p$ converges $p$-adically to zero when $p >2$, and $\frac{p^{2n}}{n!}$ converges $p$-adically to zero when $p =2$.
\end{proof}

\begin{proof}[Proof that (1) implies (2) in Theorem~\ref{GLS}]
	It suffices to prove, for $m \in \varphi(\mathfrak{M})$ and $\sigma \in G_K$, that $(\sigma-1)(m) \in \mathfrak{M}^{\varphi} \otimes_{\mathfrak{S}} [\pi^\flat]^p\mu A_{\operatorname{inf}}$. Since $A_{\operatorname{inf}}[\frac{1}{\mu}]$ is $G_K$-stable, Lemma~\ref{1/mu} ensures that $(\sigma -1)(m) \in \mathfrak{M}^\varphi \otimes_{\mathfrak{S}} A_{\operatorname{inf}}[\frac{1}{\mu}]$.
	
	On the other hand we know from the previous lemma that
	\begin{equation}\label{anotherequation}
	(\sigma -1)(m) = \sum_{n \geq 1} \mathcal{N}^n(m) \otimes \frac{(- \operatorname{log}([\epsilon(\sigma)]))^n}{n!}
	\end{equation}
	Since $\partial \circ \varphi = p \varphi \circ \partial$ the operator $\mathcal{N}$ satisfies $\mathcal{N} \varphi = p\varphi \mathcal{N}$, and so
	$$
	\mathcal{N}^n(m) \in \varphi(\mathcal{N}(\mathfrak{M} \otimes_{\mathfrak{S}} \mathcal{O}^{\operatorname{rig}}[\tfrac{1}{\lambda}]))
	$$
	for $n \geq 1$. By the definition of $\mathcal{N}$ we have $\mathcal{N}(\mathfrak{M} \otimes_{\mathfrak{S}} \mathcal{O}^{\operatorname{rig}}[\frac{1}{\varphi(\lambda)}]) \subset \mathfrak{M} \otimes_{\mathfrak{S}} u\mathcal{O}^{\operatorname{rig}}[\frac{1}{\lambda}]$. Therefore $\mathcal{N}^n(m) \in \mathfrak{M}^\varphi \otimes_{\mathfrak{S}} [\pi^\flat]^p B^{+}_{\operatorname{max}}$. Since $\operatorname{log}([\epsilon(\sigma)]) \in tA_{\operatorname{max}} = \mu A_{\operatorname{max}}$ (the equality follows from \cite[III.3.9]{Col98}) each term of \eqref{anotherequation} is contained in $\mathfrak{M}^\varphi \otimes_{\mathfrak{S}} [\pi^\flat]^p \mu B_{\operatorname{max}}^+$. Lemma~\ref{topology} implies the entire sum \eqref{anotherequation} is contained in $\mathfrak{M}^\varphi \otimes_{\mathfrak{S}} [\pi^\flat]^p \mu B_{\operatorname{max}}^+$.
	
	To complete the proof it suffices to show that $A_{\operatorname{inf}}[\frac{1}{\mu}] \cap [\pi^\flat]^p \mu B_{\operatorname{max}}^+ = [\pi^\flat]^p \mu A_{\operatorname{inf}}$. This follows from the next two facts. The first is that if $a \in A_{\operatorname{inf}} \cap [\pi^\flat]^p B_{\operatorname{max}}^+$ then $a \in [\pi^\flat]^pA_{\operatorname{inf}}$. This is proven with $B_{\operatorname{max}}^+$ replaced by $B_{\operatorname{crys}}^+$ in \cite[Lemma 3.2.2]{Liu10b}. Using that $\varphi(B_{\operatorname{max}}^+) \subset B_{\operatorname{crys}}^+$ we deduce the same applies for $B_{\operatorname{max}}^+$. The second fact is that if $a \in A_{\operatorname{inf}} \cap \mu^n B_{\operatorname{max}}^+$ then $a \in \mu^n A_{\operatorname{inf}}$. It suffices to prove this when $n=1$. The homomorphism $\theta: A_{\operatorname{inf}} \rightarrow \mathcal{O}_C$ given by $\sum [x_i]p^i \mapsto \sum x_i^\sharp p^i$ extends to $\theta: B_{\operatorname{max}}^+ \rightarrow C$ and, since $\theta(\varphi^n(\mu))= 0$, we must have $\theta(\varphi^n(a)) = 0$ for all $n \geq 0$. The claim then follows from \cite[Proposition 5.1.3]{Fon94} which states that $\lbrace a \in A_{\operatorname{inf}} \mid \varphi^n(a) \in \operatorname{ker} \theta\text{ for all $n \geq 0$} \rbrace = \mu A_{\operatorname{inf}}$.
\end{proof}
\subsection{The locus of crystalline Breuil--Kisin modules}
\begin{sub}
	Let $A$ be an Artin local ring with finite residue field $\mathbb{F}$ of characteristic $p$. Suppose $V_A$ is a finite free $A$-module equipped with a continuous $A$-linear $G_{K_\infty}$-action.
	
	 Since $A$ is a finite $\mathbb{Z}_p$-module, as in \ref{etale} we obtain an $\mathcal{O}_{\mathcal{E},A} = \mathcal{O}_{\mathcal{E}} \otimes_{\mathbb{Z}_p} A$-module $M_A$ equipped with an isomorphism $\varphi^* M_A \xrightarrow{\sim} M_A$ such that there exists a $\varphi,G_{K_\infty}$-equivariant identification \eqref{Acomparison}. Since $V_A$ is $A$-free $M_A$ is $\mathcal{O}_{\mathcal{E},A}$-free, cf. \cite[1.2.7]{KisFF}. For any $A$-algebra $B$ set $M_B = M_A \otimes_A B$ and $V_B = V_A \otimes_A B$.
\end{sub}
\begin{definition}\label{Ldefn}
For any $A$-algebra $B$ define $\mathcal{L}^{\leq h}(V_B)$ to be the set of finite projective $\mathfrak{S}_B = \mathfrak{S} \otimes_{A} B$-submodules $\mathfrak{M}_B \subset M_B$ satisfying $\mathfrak{M}_B \otimes_{\mathfrak{S}} \mathcal{O}_{\mathcal{E}} = M_B$ and, if $\mathfrak{M}_B^\varphi$ denotes the image of $\varphi^*\mathfrak{M}_B$ under $\varphi^*M_B \rightarrow M_B$, satisfying
$$
E(u)^h \mathfrak{M}_B \subset \mathfrak{M}_B^{\varphi} \subset \mathfrak{M}_B
$$
If $B \rightarrow B'$ is a map of $A$-algebras and $\mathfrak{M}_B \in \mathcal{L}^{\leq h}(V_B)$ then $\mathfrak{M}_B \otimes_B B'$ is a finite projective $\mathfrak{S}_{B'}$-submodule\footnote{Note $\mathfrak{M}_B \otimes_{\mathfrak{S}_B} \mathfrak{S}_{B'} = \mathfrak{M}_B \otimes_{\mathfrak{S}_B} (\mathfrak{S}_B \otimes_{B'} B) = \mathfrak{M}_B \otimes_{B'} B$.} of $M_{B'}$ and $\varphi^* (\mathfrak{M}_{B} \otimes_B B') \rightarrow (\mathfrak{M}_B \otimes_B B')$ has cokernel killed by $E^h$. Since $\mathcal{O}_{\mathcal{E}} \otimes _{\mathfrak{S}} (\mathfrak{M}_B \otimes_B B')\cong M_B \otimes_B B'$ we have that $\mathfrak{M}_B \otimes_B B' \in \mathcal{L}^{\leq h}(V_{B'})$. Thus $B \mapsto \mathcal{L}^{\leq h}(V_B)$ is a functor on $A$-algebras.
\end{definition}
The functor $\mathcal{L}^{\leq h}$ was introduced by Kisin. In \cite[1.3]{Kis08} he proves:
\begin{proposition}[Kisin]\label{kisprop}
	The functor $B \mapsto \mathcal{L}^{\leq h}(V_B)$ is represented by a projective $A$-scheme $\mathcal{L}^{\leq h}_A$. If $A \rightarrow A'$ is a map of Artin local rings with finite residue field then there are functorial isomorphisms $\mathcal{L}^{\leq h}_A \otimes_A A' \cong \mathcal{L}^{\leq h}_{A'}$. Furthermore, $\mathcal{L}^{\leq h}_A$ is equipped with a very ample line bundle which is similarly functorial in $A$.
\end{proposition}

\begin{sub}
Now suppose $V_A$ is a finite free $A$-module equipped with a continuous $A$-linear action of $G_K$. Apply the previous discussion to $V_A|_{G_{K_\infty}}$. If $B$ is an $A$-algebra and $\mathfrak{M}_B \in \mathcal{L}^{\leq h}(V_B)$ then \eqref{Acomparison} induces a $\varphi,G_{K_\infty}$-equivariant isomorphism
\begin{equation}\label{comparison}
\mathfrak{M}_B \otimes_{\mathfrak{S}} W(C^\flat) \cong M_B \otimes_{\mathcal{O}_{\mathcal{E}}} W(C^\flat) \cong V_B \otimes_{\mathbb{Z}_p} W(C^\flat)
\end{equation}
The $G_K$-action on $V_B$ and $W(C^\flat)$ provides an action of $G_K$ on $\mathfrak{M}_B \otimes_{\mathfrak{S}} W(C^\flat)$.	
\end{sub}

\begin{definition}\label{Lcrys}
	For any $A$-algebra $B$ let $\mathcal{L}^{\leq h}_{\operatorname{crys}}(V_B)$ denote the set of $ \mathfrak{M}_B \in \mathcal{L}^{\leq h}(V_B)$ such that the $G_K$-action on $\mathfrak{M}_B \otimes_{\mathfrak{S}} W(C^\flat)$ given by \eqref{comparison} satisfies
	$$
	(\sigma-1)(m) \in \mathfrak{M}_B \otimes_{\mathfrak{S}} [\pi^\flat]\varphi^{-1}(\mu)A_{\operatorname{inf}}
	$$
	for all $m \in \mathfrak{M}_B$ and all $\sigma \in G_K$. Again $B \mapsto \mathcal{L}^{\leq h}_{\operatorname{crys}}(V_B)$ is a functor on $A$-algebras.
\end{definition}

We shall prove that $B \mapsto \mathcal{L}^{\leq h}_{\operatorname{crys}}(V_B)$ is represented by a closed subscheme of $\mathcal{L}^{\leq h}$. First we need some lemmas.
\begin{lemma}\label{smallest}
	Let $Q$ be a flat $\mathbb{Z}_p$-module and $A$ a $\mathbb{Z}_p$-algebra with $p^nA =0$ for some $n \geq 0$. For any $x \in A \otimes_{\mathbb{Z}_p} Q$ there exists a smallest ideal $I(x) \subset A$ such that $x \in I(x) \otimes_{\mathbb{Z}_p} Q$.
\end{lemma}
\begin{proof}
	We shall show there exists a smallest $\mathbb{Z}_p$-submodule $M(x) \subset A$ such that $M(x) \otimes_{\mathbb{Z}_p} Q$ contains $x$. Then $I(x)$ will be equal to the ideal generated by $M(x)$ over $A$; if $J \subset A$ is an ideal such that $x \in J \otimes_{\mathbb{Z}_p} Q$ then $M(x) \subset J$ and so $I(x) \subset J$.
	
	We use that $\otimes_{\mathbb{Z}_p} Q$ commutes with finite intersections, since $Q$ is $\mathbb{Z}_p$-flat. Choose a finitely generated $\mathbb{Z}_p$-submodule $M \subset A$ with $x \in M \otimes_{\mathbb{Z}_p} Q$. Since $p^nA = 0$, $M$ has finite length and so contains only finitely many $\mathbb{Z}_p$-submodules. Thus, if $M(x)$ is the intersection of all $M' \subset M$ with $x \in M' \otimes_{\mathbb{Z}_p} Q$ then $x \in M(x) \otimes_{\mathbb{Z}_p} Q$. If $M'' \subset A$ is any other $\mathbb{Z}_p$-submodule with $x \in M'' \otimes_{\mathbb{Z}_p} Q$ then $x \in (M'' \cap M) \otimes_{\mathbb{Z}_p} Q$ and so $M(x) \subset (M'' \cap M) \subset M''$. Therefore $M(x)$ is as desired.
\end{proof}
In the proof of the following lemma we use that both $[\pi^\flat]$ and $\varphi^{-1}(\mu)$ are units in $W(C^\flat)$. This can be seen by observing that modulo~$p$ both are non-zero in $C^\flat$.
\begin{lemma}\label{ideal}
	Let $B$ be an $A$-algebra and $\mathfrak{M}_B \in \mathcal{L}^{\leq h}(V_B)$. There exists a unique ideal $I \subset B$ such that, for any $A$-algebra homomorphism $B \rightarrow B'$, $\mathfrak{M}_B \otimes_B B' \in \mathcal{L}^{\leq h}_{\operatorname{crys}}(V_{B'})$ if and only if $B \rightarrow B'$ factors through $B \rightarrow B/I$.
\end{lemma}
\begin{proof}
	Consideration of Teichmuller expansions shows that if $x \in W(C^\flat)$ and $p x \in A_{\operatorname{inf}}$ then $x \in A_{\operatorname{inf}}$. Therefore the $\mathfrak{S}$-module $Q := W(C^\flat)/A_{\operatorname{inf}}$ is $\mathbb{Z}_p$-flat, and so
	$$
	0 \rightarrow B' \otimes_{\mathbb{Z}_p} A_{\operatorname{inf}} \rightarrow B' \otimes_{\mathbb{Z}_p} W(C^\flat) \rightarrow B' \otimes_{\mathbb{Z}_p} Q \rightarrow 0
	$$
	is exact. Since $\mathfrak{M}_{B'} := \mathfrak{M}_B \otimes_B B'$ is finite projective over $\mathfrak{S}_{B'}$, applying $\mathfrak{M}_{B'} \otimes_{\mathfrak{S}_{B'}}$ to the above exact sequence yields a sequence
	$$
	0 \rightarrow \mathfrak{M}_{B'} \otimes_{\mathfrak{S}} A_{\operatorname{inf}} \rightarrow \mathfrak{M}_{B'} \otimes_{\mathfrak{S}} W(C^\flat) \rightarrow \mathfrak{M}_{B'} \otimes_{\mathfrak{S}} Q \rightarrow 0
	$$
	which is again exact. Thus $\mathfrak{M}_{B'} \in \mathcal{L}^{\leq h}_{\operatorname{crys}}(V_{B'})$ if and only if, for every $m \in \mathfrak{M}_{B'}$ and every $\sigma \in G_K$, the image of
	$$
	\tfrac{(\sigma-1)(m)}{[\pi^\flat]\varphi^{-1}(\mu)} \in \mathfrak{M}_{B'} \otimes_{\mathfrak{S}} W(C^\flat)
	$$
	in $\mathfrak{M}_{B'} \otimes_{\mathfrak{S}} Q$ is zero. In fact, since $\mathfrak{M}_{B'}$ is generated over $B'$ by the image of $\mathfrak{M}_B$, we need only consider $m$ contained in the image of $\mathfrak{M}_B \rightarrow \mathfrak{M}_{B'}$.
	
	As $\mathfrak{M}_B$ is finite projective over $\mathfrak{S}_B$ there is an isomorphism $\mathfrak{M}_B \oplus Z \cong (B \otimes_{\mathbb{Z}_p} \mathfrak{S})^r$ for some $\mathfrak{S}_B$-module $Z$. Thus we obtain an inclusion $\mathfrak{M}_B \otimes_{\mathfrak{S}} Q \hookrightarrow (B \otimes_{\mathbb{Z}_p} Q)^r$. If $e_i$ denotes the standard basis of $(B \otimes_{\mathbb{Z}_p} Q)^r$ then, for every $m \in \mathfrak{M}_B$ and $\sigma \in G_K$, the image of $\frac{(\sigma-1)(m)}{[\pi^\flat]\varphi^{-1}(\mu)}$ under $\mathfrak{M}_B \otimes_{\mathfrak{S}} W(C^\flat) \rightarrow \mathfrak{M}_B \otimes_{\mathfrak{S}} Q \hookrightarrow (B \otimes_{\mathbb{Z}_p} Q)^r$ can be written as
	$$
	\sum \alpha(m,\sigma,i) e_i
	$$
	for some $\alpha(m,\sigma,i) \in B \otimes_{\mathbb{Z}_p} Q$. Let $I(m,\sigma,i) \subset B$ be the smallest ideal such that $\alpha(m,\sigma,i) \in I(m,\sigma,i) \otimes_{\mathbb{Z}_p} Q$ (which exists by Lemma~\ref{smallest}) and let $I = \sum_{m,\sigma,i} I(m,\sigma,i)$. The discussion from the previous paragraph shows that $I$ is as required by the lemma.
\end{proof}

\begin{proposition}
	There exists a closed $A$-subscheme $\mathcal{L}^{\leq h}_{A,\operatorname{crys}}$ of $\mathcal{L}^{\leq h}_A$ which represents the functor $B \mapsto \mathcal{L}^{ \leq h}_{\operatorname{crys}}(V_B)$.
\end{proposition}
\begin{proof}
	To any morphism $\operatorname{Spec}B \rightarrow \mathcal{L}^{\leq h}$ of $A$-schemes we associate $\mathfrak{M}_B \in \mathcal{L}^{\leq h}(V_B)$ and so an ideal $I_B \subset B$ as in Lemma~\ref{ideal}. The uniqueness in Lemma~\ref{ideal} implies that if $B \rightarrow B'$ is an $A$-algebra homomorphism then $I_{B'}$ is the ideal of $B'$ generated by the image of $I_B$. Thus the association $B \mapsto I_B$ defines a coherent sheaf of ideals on $\mathcal{L}^{\leq h}$. Let $\mathcal{L}^{\leq h}_{A,\operatorname{crys}}$ be the corresponding closed $A$-subscheme of $\mathcal{L}^{\leq h}$. Since a morphism $\operatorname{Spec}B \rightarrow \mathcal{L}^{\leq h}$ of $A$-schemes factors through $\mathcal{L}^{\leq h}_{A,\operatorname{crys}}$ if and only if $I_B = 0$, and this occurs if and only if $\mathfrak{M}_B \in \mathcal{L}^{\leq h}_{\operatorname{crys}}(V_B)$ it follows that $\mathcal{L}^{\leq h}_{A,\operatorname{crys}}$ represents $B \mapsto \mathcal{L}^{\leq h}_{\operatorname{crys}}(V_B)$.
\end{proof}
\begin{sub}\label{circ}
Now let $A$ be a complete local Noetherian ring with residue field $\mathbb{F}$ and maximal ideal $\mathfrak{m}_{A}$. Let $V_{A}$ be a finite free $A$-module equipped with a continuous action of $G_K$.	
\end{sub}
\begin{corollary}\label{formal}
	There exists a projective $A$-scheme $\mathcal{L}^{\leq h}_{A,\operatorname{crys}}$ which, for each $i \geq 1$ represents the functor $B \mapsto \mathcal{L}^{\leq h}_{\operatorname{crys}}(V_{A} \otimes_A B)$ on $A$-algebras $B$ with $\mathfrak{m}_{A}^iB = 0$.
\end{corollary}
\begin{proof}
	Set $A_i = A/ \mathfrak{m}_{A}^i$. The projective schemes $\mathcal{L}^{\leq h}_{A_i,\operatorname{crys}}$ form an inverse system of schemes over $A_i$, and so a formal scheme over $A$. The very ample line bundles on each $\mathcal{L}^{\leq h}_{A_i}$ restrict to an inverse system of very ample line bundles on the $\mathcal{L}^{\leq h}_{A_i,\operatorname{crys}}$. As a consequence of \cite[Th\'eor\`eme 5.4.5]{EGAIII} this formal scheme arises from a projective $A$ scheme as required.
\end{proof}
\begin{sub}
Suppose $C$ is a local finite flat $\mathbb{Z}_p$-algebra and $V_C$ is a finite free $C$-module equipped with a continuous $C$-linear action of $G_K$. Then there is an $\mathcal{O}_{\mathcal{E},C} =\mathcal{O}_{\mathcal{E}} \otimes_{\mathbb{Z}_p} C$-module $M_C$ equipped with an isomorphism $\varphi^* M_C \rightarrow M_C$ and a $\varphi,G_{K_\infty}$-equivariant identification $M_C \otimes_{\mathcal{O}_{\mathcal{E}}} W(C^\flat) \cong V_C \otimes_{\mathbb{Z}_p} W(C^\flat)$. In the obvious way we make sense of the sets $\mathcal{L}^{\leq h}(V_C)$ and $\mathcal{L}^{\leq h}_{\operatorname{crys}}(V_C)$. Thus $\mathfrak{M}_C \in \mathcal{L}^{\leq h}(V_C)$ if $\mathfrak{M}_C \subset M_C$ is a $\varphi$-stable projective $\mathfrak{S}_C = \mathfrak{S} \otimes_{\mathbb{Z}_p} C$-module so that $\mathfrak{M}_C \otimes_{\mathfrak{S}} \mathcal{O}_{\mathcal{E}} = M_C$ and so that $\varphi^*\mathfrak{M}_C \rightarrow \mathfrak{M}_C$ has cokernel is killed by $E(u)^h$. Further $\mathfrak{M}_C \in \mathcal{L}^{\leq h}_{\operatorname{crys}}(V_C)$ if the $G_K$-action on $\mathfrak{M}_C \otimes_{\mathfrak{S}} W(C^\flat) \cong V_C \otimes_{\mathbb{Z}_p} W(C^\flat)$ is such that
$$
(\sigma-1)(m)\in \mathfrak{M}_C \otimes_{\mathfrak{S}} [\pi^\flat]\varphi^{-1}(\mu)A_{\operatorname{inf}}
$$
for all $\sigma \in G_K$ and $m \in \mathfrak{M}_C$.
\end{sub}

\begin{sub}\label{Cpoints}
Let $C$ be an $A$-algebra which is finite flat over $\mathbb{Z}_p$. A morphism $\operatorname{Spec}C \rightarrow \mathcal{L}^{\leq h}_{A,\operatorname{crys}}$ gives morphisms $\mathcal{L}^{\leq h }_{A,\operatorname{crys}} \rightarrow \operatorname{Spec} C_i$, where $C_i = C/p^iC$. For any $i \geq 1$ there is a $j$ such that $\mathfrak{m}_A^j C \subset p^iC$, and so, by Corollary~\ref{formal}, such a system of morphisms gives rise to $\mathfrak{M}_{C_i} \in \mathcal{L}^{\leq h}_{\operatorname{crys}}(V_{C_i})$ with $\mathfrak{M}_{C_{i}} = \mathfrak{M}_{C_{i+1}} \otimes_{C_{i+1}} C_i$. The limit $\mathfrak{M}_C = \varprojlim \mathfrak{M}_{C_i}$ is a projective $\mathfrak{S}_C$-submodule of $M_C = \varprojlim M_{C_i}$ defining an element of $\mathcal{L}^{\leq h}(V_C)$. Under the identification
$$
\mathfrak{M}_C \otimes_{\mathfrak{S}} W(C^\flat) = V_C \otimes_{\mathbb{Z}_p} W(C^\flat)
$$
the $G_K$ action on $\mathfrak{M}_C \otimes_{\mathfrak{S}} W(C^\flat)$ is such that, for each $i \geq 1$ and each $m \in \mathfrak{M}_C, \sigma \in G_K$, the images of the elements
$$
\tfrac{(\sigma-1)(m)}{[\pi^\flat]\varphi^{-1}(\mu)}
$$
in $\mathfrak{M}_{C_i} \otimes_{\mathfrak{S}} W(C^\flat)$ are contained in $\mathfrak{M}_{C_i} \otimes_{\mathfrak{S}} A_{\operatorname{inf}}$. Since $C$ is finite free as a $\mathbb{Z}_p$-module $\mathfrak{M}_C$ is projective, and hence free, over $\mathfrak{S}$. This implies these elements are contained in $\mathfrak{M}_C \otimes_{\mathfrak{S}} A_{\operatorname{inf}}$: since $\mathfrak{M}_C$ is free over $\mathfrak{S}$ it suffices to show that if $x \in W(C^\flat)$ is congruent to an element of $A_{\operatorname{inf}}$ modulo $p^i$ for every $i \geq 1$ then $x \in A_{\operatorname{inf}}$. Considering the Teichmuller expansion of $x$ shows this statement holds. 

Conversely any $\mathfrak{M}_C \in \mathcal{L}^{\leq h}_{\operatorname{crys}}(V_C)$ gives rise to a unique $C$-point of $\mathcal{L}^{\leq h}_{A,\operatorname{crys}}$.
\end{sub} 

\begin{lemma}\label{closedimmersion}
	The morphism $\mathcal{L}^{\leq h}_{A,\operatorname{crys}} \rightarrow \operatorname{Spec} A$ becomes a closed immersion after inverting $p$.
\end{lemma}
\begin{proof}
One argues exactly as in \cite[1.6.4]{Kis08}. As explained in \emph{loc. cit.}, any point of $\mathcal{L} = \mathcal{L}^{\leq h}_{A,\operatorname{crys}}$ valued in a finite local $\mathbb{Q}_p$-algebra $B$ is induced from a $C$-valued point for a finite flat $\mathbb{Z}_p$-algebra $C \subset B$. We claim this implies $\mathcal{L}(B) \rightarrow (\operatorname{Spec}A)(B)$ is injective. Indeed given two $B$-valued points of $\mathcal{L}$ inducing the same $B$-valued point of $\operatorname{Spec}A$ the above produces a finite flat $\mathbb{Z}_p$-algebra $C \subset B$ so that both $B$-valued points factor through $\operatorname{Spec}B \rightarrow \operatorname{Spec}C$. The last sentence of \ref{Eheight} implies $\mathcal{L}^{\leq h}_{\operatorname{crys}}(V_C)$ consists of at most one element, and so \ref{Cpoints}, implies both $B$-valued points of $\mathcal{L}$ are induced from the same $C$-valued point.

Taking $B = E$ for any finite extension $E/\mathbb{Q}_p$ shows that the proper morphism $\mathcal{L}^{\leq h}_{A,\operatorname{crys}} \otimes_{\mathbb{Z}_p} \mathbb{Q}_p \rightarrow \operatorname{Spec}A[\frac{1}{p}]$ is injective on closed points, and at these closed points induces an isomorphism of residue fields. Taking $B = E[\epsilon]/(\epsilon^2)$ shows that at these closed points this morphism also induces an injection of tangent spaces. We conclude it is a closed immersion.
\end{proof}
\begin{proposition}\label{quotient}
	Let $A^{\leq h}_{\operatorname{crys}}$ denote the quotient of $A$ corresponding to the scheme-theoretic image of $\mathcal{L}^{\leq h}_{A,\operatorname{crys}} \rightarrow \operatorname{Spec} A$. Then
	\begin{enumerate}
		\item The morphism $\mathcal{L}^{\leq h}_{A,\operatorname{crys}}  \rightarrow \operatorname{Spec}A^{ \leq h}_{\operatorname{crys}}$ becomes an isomorphism after inverting $p$.
		\item For any finite $\mathbb{Q}_p$-algebra $B$, a map $A \rightarrow B$ factors through $A^{\leq h}_{\operatorname{crys}}$ if and only if $V_B = V_{A} \otimes_{A} B$ is crystalline with Hodge--Tate weights contained in $[0,h]$.
	\end{enumerate}
\end{proposition}
\begin{proof}
	Part (1) follows from Lemma~\ref{closedimmersion}. As a consequence, is $B$ is a finite $\mathbb{Q}_p$-algebra, a map $A \rightarrow B$ factors through $A^{\leq h}_{\operatorname{crys}}$ if and only if $A \rightarrow B$ is induced from a $B$-valued point of $\mathcal{L}^{\leq h}_{A,\operatorname{crys}}$. 
	
	In proving (2) we may assume $B$ is local. As in the proof of Lemma~\ref{closedimmersion}, any $B$-valued point of $\mathcal{L}^{\leq h}_{A,\operatorname{crys}}$ is induced from a $C$-valued point with $C$ finite flat over $\mathbb{Z}_p$. This in turn gives rise to an $\mathfrak{M}_C \in \mathcal{L}^{\leq h}_{\operatorname{crys}}(V_C)$. The fact that the $G_K$-action on $\mathfrak{M}_C \otimes_{\mathbb{Z}_p} W(C^\flat)$ satisfies our usual condition implies, after Theorem~\ref{GLS}, that $V_C = V_{A} \otimes_{A} C$ is a $\mathbb{Z}_p$-lattice inside the crystalline representation $V_C[\frac{1}{p}]$ whose Hodge--Tate weights are contained in $[0,h]$. Thus the same is true for $V_B = V_C[\frac{1}{p}] \otimes_{C[\frac{1}{p}]} B$.
	
	For the converse, suppose $A \rightarrow B$ is such that $V_B = V_{A} \otimes_{A} B$ is crystalline with Hodge--Tate weights contained in $[0,h]$. Then there is a finite flat $\mathbb{Z}_p$-algebra $C \subset B$ so that $A \rightarrow B$ factors through $C$. As $V_C \otimes_{\mathbb{Z}_p} \mathbb{Q}_p$ is a $G_K$-stable $\mathbb{Q}_p$-subspace of $V_B$, $V_C \otimes_{\mathbb{Z}_p} \mathbb{Q}_p$ is also crystalline with Hodge--Tate weights contained in $[0,h]$. Theorem~\ref{GLS} implies there exists a Breuil--Kisin module $\mathfrak{M}_C$ and a $\varphi,G_{K_\infty}$-equivariant identification
	$$
	\mathfrak{M}_C \otimes_{\mathfrak{S}} W(C^\flat) \cong V_C \otimes_{\mathbb{Z}_p} W(C^\flat)
	$$
	such that $(\sigma-1)(m) \in \mathfrak{M}_C \otimes_{\mathfrak{S}} [\pi^\flat]\varphi^{-1}(\mu)A_{\operatorname{inf}}$ for every $\sigma \in G_K$ and every $m \in \mathfrak{M}_C$. By functoriality $\mathfrak{M}_C$ is an $\mathfrak{S}_C$-module, but it need not be projective. However, in the second to last paragraph of the proof of \cite[1.6.4]{Kis08} it is shown that, at the cost of enlarging $C$, one can arrange that $\mathfrak{M}_C$ is projective over $\mathfrak{S}_C$. Thus $A \rightarrow B$ arises from a $C$-point of $\mathcal{L}$ for some $C \subset B$ finite flat over $\mathbb{Z}_p$, and therefore from a $B$-point of $\mathcal{L}$. We conclude that $A \rightarrow B$ factors through $A^{ \leq h}_{\operatorname{crys}}$.
\end{proof}

\begin{remark}\label{points}
	\begin{enumerate}
		\item The fact that $\mathfrak{M}_C$ need not be $\mathfrak{S}_C$-projective, even though $V_C$ is projective as a $C$-module is related to the fact that the functor from finite $E$-height representations to Breuil--Kisin modules is not exact. 
		\item There is one instance in which $V_C$ being $C$-projective implies $\mathfrak{M}_C$ is $\mathfrak{S}_C$-projective. This is when $C$ is the ring of integers of a finite extension of $\mathbb{Q}_p$. See for example \cite[Proposition 3.4]{GLS} for a proof. In particular, if $E/\mathbb{Q}_p$ is finite and $V$ is a $G_K$-stable $\mathcal{O}_E$-lattice inside a crystalline representation of $G_K$ then the Breuil--Kisin module associated to $V$ is an element of $\mathcal{L}^{\leq p}_{\operatorname{crys}}(V)$.
		\end{enumerate}
\end{remark}
\section{Strong divisibility}\label{strongdivisibility}

For the rest of the paper we assume $K$ is an unramified extension of $\mathbb{Q}_p$.
\subsection{Strong divisibility}\label{SD}

\begin{sub}
	Let $\mathbb{F}$ be a finite field of characteristic $p$ and $V_{\mathbb{F}}$ a finite free $\mathbb{F}$-module equipped with a continuous $\mathbb{F}$-linear action of $G_{K_\infty}$.
\end{sub}
\begin{definition}\label{strong}
	Let $\mathcal{L}^{\leq p}_{\operatorname{SD}}(V_{\mathbb{F}})$ denote the set of $\mathfrak{M} \in \mathcal{L}^{\leq p}(V_{\mathbb{F}})$ for which there exists a $k[[u]]$-basis $(m_i)$ of $\mathfrak{M}$ and integers $r_i$ such that $(u^{r_i}m_i)$ forms a $k[[u^p]]$-basis of $\varphi(\mathfrak{M})$. We call $\mathfrak{M}$ satisfying this condition strongly divisible.
\end{definition}

We are going to relate $\mathcal{L}^{\leq p}_{\operatorname{SD}}(V_{\mathbb{F}})$ with $\mathcal{L}^{\leq p}_{\operatorname{crys}}(V_{\mathbb{F}})$.\footnote{Note that the latter set only makes sense when the $G_{K_\infty}$-action on $V_{\mathbb{F}}$ extends to a continuous $G_K$-action.} Before doing so we record how some basic operations on $V_{\mathbb{F}}$ respect $\mathcal{L}_{\operatorname{SD}}^{\leq p}(V_{\mathbb{F}})$ and $\mathcal{L}_{\operatorname{crys}}^{\leq p}(V_{\mathbb{F}})$.

\begin{lemma}\label{exactsequences}
	Let $V_{\mathbb{F}}$ be as above and suppose $W_{\mathbb{F}}$ is another continuous representation of $G_{K_\infty}$ on an $\mathbb{F}$-vector space. Suppose $\mathfrak{M} \in \mathcal{L}^{\leq p}_{\operatorname{SD}}(V_{\mathbb{F}})$ and $\mathfrak{N} \in \mathcal{L}^{\leq p}(W_{\mathbb{F}})$.
	\begin{enumerate}
		\item Suppose there exists a surjective $\varphi$-equivariant map of $k[[u]]$-modules $f\colon \mathfrak{M} \rightarrow \mathfrak{N}$. Then $\mathfrak{N} \in \mathcal{L}^{\leq p}_{\operatorname{SD}}(W_{\mathbb{F}})$.
		\item Suppose there exists an injective $\varphi$-equivariant map of $k[[u]]$-modules $f\colon \mathfrak{N} \rightarrow \mathfrak{M}$ with $u$-torsionfree cokernel. Then $\mathfrak{N} \in \mathcal{L}^{\leq p}_{\operatorname{SD}}(W_{\mathbb{F}})$.
	\end{enumerate}
\end{lemma}
\begin{proof}
	This follows from part (1) of \cite[5.4.6]{B18}.
\end{proof}
\begin{lemma}\label{extend}
	For any finite extension $\mathbb{F}'$ of $\mathbb{F}$ the rule $\mathfrak{M} \mapsto \mathfrak{M} \otimes_{\mathbb{F}} \mathbb{F}'$ defines a map
	$$
	\mathcal{L}^{\leq h}(V_{\mathbb{F}}) \rightarrow \mathcal{L}^{\leq h}(V_{\mathbb{F}} \otimes_{\mathbb{F}} \mathbb{F}')
	$$
	Further, $\mathfrak{M} \in \mathcal{L}^{\leq p}_{\operatorname{SD}}(V_{\mathbb{F}})$ if and only if its image lies in $\mathcal{L}^{\leq p}_{\operatorname{SD}}(V_{\mathbb{F}} \otimes_{\mathbb{F}} \mathbb{F}')$. If $V_{\mathbb{F}}$ admits a $G_K$-action then likewise $\mathfrak{M} \in \mathcal{L}^{\leq h}_{\operatorname{crys}}(V_{\mathbb{F}})$ if and only if its image lies in $\mathcal{L}^{\leq h}_{\operatorname{crys}}(V_{\mathbb{F}} \otimes_{\mathbb{F}} \mathbb{F})$.
\end{lemma}
\begin{proof}
	The only part which does not follow immediately from the definitions is that $\mathfrak{M} \otimes_{\mathbb{F}} \mathbb{F}' \in \mathcal{L}^{\leq p}_{\operatorname{SD}}(V_{\mathbb{F}} \otimes_{\mathbb{F}} \mathbb{F}')$ implies $\mathfrak{M} \in \mathcal{L}^{\leq p}_{\operatorname{SD}}(V_{\mathbb{F}})$. For this note that the inclusion $\mathfrak{M} \rightarrow \mathfrak{M} \otimes_{\mathbb{F}} \mathbb{F}'$ is $\varphi$-equivariant with $u$-torsionfree cokernel. Thus we can apply (2) of Lemma~\ref{exactsequences}.
\end{proof}

\begin{lemma}\label{unramtwist!}
	For any unramified $\mathbb{F}$-valued character $\psi$ of $G_K$ there is a bijection
	$$
	\mathcal{L}^{\leq h}(V_{\mathbb{F}}) \xrightarrow{\sim} \mathcal{L}^{\leq h}(V_{\mathbb{F}} \otimes_{\mathbb{F}} \mathbb{F}(\psi))
	$$
	which identifies $\mathcal{L}^{\leq p}_{\operatorname{SD}}(V_{\mathbb{F}})$ with $\mathcal{L}^{\leq p}_{\operatorname{SD}}(V_{\mathbb{F}} \otimes_{\mathbb{F}} \mathbb{F}(\psi))$ and, if $V_\mathbb{F}$ admits a $G_K$-action, identifies $\mathcal{L}^{\leq h}_{\operatorname{crys}}(V_{\mathbb{F}})$ with $\mathcal{L}^{\leq h}_{\operatorname{crys}}(V_{\mathbb{F}} \otimes_{\mathbb{F}} \mathbb{F}(\psi))$. 
	\end{lemma}
	\begin{proof}
	First we note there exists a $y \in (\overline{k} \otimes_{\mathbb{F}_p} \mathbb{F})^\times$ such that $\sigma(y) = \psi(\sigma)^{-1}y$ for all $\sigma \in G_K$, due to the assumption $\psi$ is unramified. Using $y$ we can describe the etale $\varphi$-module associated to $V_{\mathbb{F}}(\psi) := V_{\mathbb{F}} \otimes_{\mathbb{F}} \mathbb{F}(\psi)$ in terms of that associated to $V_{\mathbb{F}}$. To do this consider the $\mathbb{F}$-linear map $V_{\mathbb{F}} \rightarrow V_{\mathbb{F}}(\psi)$ given by $v \mapsto v \otimes 1$. Applying $\otimes_{\mathbb{Z}_p} W(C^\flat)$ induces an identification
	\begin{equation}\label{time}
	V_{\mathbb{F}} \otimes_{\mathbb{Z}_p} W(C^\flat) \xrightarrow{\sim} V_{\mathbb{F}}(\psi) \otimes_{\mathbb{Z}_p} W(C^\flat)
	\end{equation}
	which is $\varphi$-equivariant when both sides are equipped with the Frobenius which is trivial on $V_{\mathbb{F}}$ and $V_{\mathbb{F}}(\psi)$. If $M \subset V_{\mathbb{F}} \otimes_{\mathbb{Z}_p} W(C^\flat)$ is the etale $\varphi$-module associated to $V_{\mathbb{F}}$ then its image $M(\psi)$ under this map is a finitely generated $\mathcal{O}_{\mathcal{E}}$-submodule of $(V_{\mathbb{F}} \otimes_{\mathbb{F}} \mathbb{F}(\psi)) \otimes_{\mathbb{Z}_p} W(C^\flat)$ on which Frobenius acts by an isomorphism and on which $G_{K_\infty}$ acts on by the character $\psi$. Via the inclusion $\overline{k} \rightarrow C^\flat$ we can view $y$ as an element of $C^\flat \otimes_{\mathbb{F}_p} \mathbb{F}$ and so have $yM(\psi) \subset V_{\mathbb{F}}(\psi) \otimes_{\mathbb{Z}_p} W(C^\flat)$. The $G_{K_\infty}$-action on $yM(\psi)$ is then trivial and, since $\varphi(y)/y \in (k \otimes_{\mathbb{F}_p} \mathbb{F})^\times \subset \mathfrak{S}_{\mathbb{F}}^\times$, the Frobenius on $yM(\psi)$ is still an isomorphism. Hence $yM(\psi)$ equals the etale $\varphi$-module associated to $V_{\mathbb{F}}(\psi)$. 
	
	Using this we can describe a map $\mathcal{L}^{\leq h}(V_{\mathbb{F}}) \rightarrow \mathcal{L}^{\leq h}(V_{\mathbb{F}}(\psi))$ sending $\mathfrak{M} \subset M$ onto $y \mathfrak{M}(\psi) \subset yM(\psi)$ where $\mathfrak{M}(\psi)$ equals the image of $\mathfrak{M}$ under \eqref{time}. Clearly this is a bijection. When $h=p$ it also identifies $\mathcal{L}^{\leq p}_{\operatorname{SD}}(V_{\mathbb{F}})$ and $\mathcal{L}^{\leq p}_{\operatorname{SD}}(V_{\mathbb{F}}(\psi))$ when when $\mathfrak{M}$ admits a basis as in Definition~\ref{strong} then so does $y\mathfrak{M}(\psi)$, and vice-versa. Finally, if $V_{\mathbb{F}}$ admits a $G_K$-action and $\mathfrak{M} \in \mathcal{L}^{\leq h}(V_{\mathbb{F}})$ then the $G_K$-action on $y\mathfrak{M}(\psi) \otimes_{\mathfrak{S}} W(C^\flat)$ identifies with the $G_K$-action on $\mathfrak{M} \otimes_{\mathfrak{S}} W(C^\flat)$ twisted by $\psi$. It follows that identifies $\mathcal{L}^{\leq h}_{\operatorname{crys}}(V_{\mathbb{F}})$ and $\mathcal{L}^{\leq h}_{\operatorname{crys}}(V_{\mathbb{F}}(\psi))$ are also identified.
	\end{proof}
	\begin{sub}\label{functorialremark}
	Note that for any $\mathbb{F}$-algebra $B$ the argument above shows that, for any finite free $B$-module $V_B$ equipped with a continuous $G_K$-action, there are functorial bijections $\mathcal{L}^{\leq h}_{\operatorname{crys}}(V_B) \cong \mathcal{L}^{\leq h}_{\operatorname{crys}}(V_B \otimes_B B(\psi))$ and $\mathcal{L}^{\leq h}(V_B) \cong \mathcal{L}^{\leq h}(V_B \otimes_B B(\psi))$.
	\end{sub}

	Finally let $L/K$ be an unramified extension corresponding to a finite extension $l/k$ of residue fields. Set $L_\infty = LK_\infty$. If $\mathfrak{S}_L := W(l)[[u]]$ is embedded into $W(C^\flat)$ as with $\mathfrak{S}$, by mapping $u$ onto $[\pi^\flat]$, then we can make sense of $\mathcal{L}^{\leq h}(V_{\mathbb{F}}|_{G_{L_\infty}})$ as in Definition~\ref{Ldefn}, replacing $K$ and $K_\infty$ by $L$ and $L_\infty$. It's elements are modules over  $\mathfrak{S}_{L,\mathbb{F}} = \mathfrak{S}_L \otimes_{\mathbb{Z}_p} \mathbb{F}$. We write $f\colon \mathfrak{S} \rightarrow \mathfrak{S}_L$ for the inclusion induced by the inclusion $k \subset l$.
	\begin{lemma}\label{inducerestrict}
		\begin{enumerate}
			\item There is a map
			$$
			f^*\colon\mathcal{L}^{\leq h}(V_{\mathbb{F}}) \rightarrow \mathcal{L}^{\leq h}(V_{\mathbb{F}}|_{G_{L_\infty}})
			$$
			sending $\mathfrak{M} \in \mathcal{L}^{\leq h}(V_{\mathbb{F}})$ onto $\mathfrak{M} \otimes_{\mathfrak{S},f} \mathfrak{S}_L$. If $\mathfrak{M} \in \mathcal{L}^{\leq p}(V_{\mathbb{F}})$ then $\mathfrak{M} \in \mathcal{L}^{\leq p}_{\operatorname{SD}}(V_{\mathbb{F}})$ if and only if $f^*\mathfrak{M} = \mathfrak{M} \otimes_{\mathfrak{S}} \mathfrak{S}_L \in \mathcal{L}^{\leq p}_{\operatorname{SD}}(V_{\mathbb{F}}|_{G_{L_\infty}})$. If $V_{\mathbb{F}}$ admits a $G_K$-action then this map sends $\mathcal{L}^{\leq h}_{\operatorname{crys}}(V_{\mathbb{F}})$ into $\mathcal{L}^{\leq h}_{\operatorname{crys}}(V_{\mathbb{F}}|_{G_L})$.
			\item If $V_{\mathbb{F}}$ is a $G_{L_\infty}$-representation then restriction of scalars along $f$ describes a map
			$$
			f_*\colon\mathcal{L}^{\leq h}(V_{\mathbb{F}}) \rightarrow \mathcal{L}^{\leq h}(\operatorname{Ind}^{G_{K_\infty}}_{G_{L_\infty}}V_{\mathbb{F}})
			$$
		\end{enumerate}
	\end{lemma}
	\begin{proof}
	The fact that there are maps $f^*$ and $f_*$ is explained in \cite[6.2.1 and 6.2.4]{B18}. The additional statements regarding the image of $f^*$ are all clear (for the observation that $\mathfrak{M} \otimes_{\mathfrak{S}} \mathfrak{S}_L$ strongly divisible implies $\mathfrak{M}$ is strongly divisible argue as in Lemma~\ref{extend} by considering the inclusion $\mathfrak{M} \rightarrow \mathfrak{M} \otimes_{\mathfrak{S}} \mathfrak{S}_L$ whose cokernel is torsionfree). 
	\end{proof}

\subsection{Strong divisibility in the irreducible case}
\begin{sub}\label{unique}
	If $K^{\operatorname{t}}$ denotes the maximal tamely ramified extension of $K$ then, since $K_\infty$ is totally ramified over $K$, $K^{\operatorname{t}} \cap K_\infty = K$. Thus the restriction map from $\operatorname{Gal}(K_\infty K^{\operatorname{t}}/K_\infty)$ to the tame quotient $\operatorname{Gal}(K^{\operatorname{t}}/K)$ of $G_K$ is an isomorphism. As such, any tamely ramified $G_K$-representation is uniquely determined by its restriction to $G_{K_\infty}$ and conversely, any tamely ramified representation of $G_{K_\infty}$ (i.e. one which factors through $\operatorname{Gal}(\overline{K}/K_{\infty}K^{\operatorname{t}})$) extends uniquely to a tame representation of $G_K$. In particular this applies to irreducible representations of $G_K$ and $G_{K_\infty}$ on $\mathbb{F}$-vector spaces, since both are tamely ramified.
\end{sub}
\begin{proposition}\label{irredequal}
	Suppose that $V_{\mathbb{F}}$ is irreducible as a $G_K$-representation. Then $\mathcal{L}^{\leq p}_{\operatorname{SD}}(V_{\mathbb{F}}) = \mathcal{L}^{\leq p}_{\operatorname{crys}}(V_{\mathbb{F}})$.
\end{proposition}
	
	Before giving a proof we make the following observation:
\begin{sub}\label{simply}
	As we are working with $p$-torsion coefficients, the condition for $\mathfrak{M} \in \mathcal{L}^{\leq h}(V_{\mathbb{F}})$ to lie in $\mathcal{L}^{\leq h}_{\operatorname{crys}}(V_{\mathbb{F}})$ can be simplified. The $v^\flat$-valuation of $[\pi^\flat]$ modulo~$p$ is $1/e$ while the $v^\flat$-valuation of $\varphi^{-1}(\mu)$ modulo~$p$ is $1/(p-1)$. Thus $\mathfrak{M} \otimes_{\mathfrak{S}} [\pi^\flat] \varphi^{-1}(\mu) A_{\operatorname{inf}} = \mathfrak{M} \otimes_{k[[u]]} I$
	where $I \subset \mathcal{O}_{C^\flat}$ is the ideal $u^{1/e + 1/(p-1)}\mathcal{O}_{C^\flat}$. As $K/\mathbb{Q}_p$ is assumed unramified $e=1$ and so $\mathfrak{M} \in \mathcal{L}^{\leq h}(V_\mathbb{F})$ is contained in $\mathcal{L}^{\leq h}_{\operatorname{crys}}(V_{\mathbb{F}})$ if and only if the induced $G_K$-action on $\mathfrak{M} \otimes_{\mathfrak{S}} W(C^\flat) = \mathfrak{M} \otimes_{k[[u]]} C^\flat$ is such that
	$$
	(\sigma -1)(m) \in \mathfrak{M} \otimes_{k[[u]]} u^{p/(p-1)} \mathcal{O}_{C^\flat}
	$$
	for every $\sigma \in G_K$ and $m \in \mathfrak{M}$.
\end{sub} 
\begin{proof}[Proof of Proposition~\ref{irredequal}]
	Using Lemma~\ref{extend} we can assume $\mathbb{F}$ is sufficiently large so that \cite[2.1.2]{B18} applies. Thus there is an unramified extension $L/K$ such that $V_{\mathbb{F}} \cong \operatorname{Ind}_{G_L}^{G_K} W_\mathbb{F}$ for a one-dimensional $G_L$-representation $W_{\mathbb{F}}$. In particular $V_{\mathbb{F}}|_{G_{K_\infty}} \cong \operatorname{Ind}_{G_{L_\infty}}^{G_{K_\infty}} W_{\mathbb{F}}|_{G_{L_\infty}}$. 
	
	\begin{sub}\label{contain}
	We claim that if $\mathfrak{M} \in \mathcal{L}^{\leq p}(V_{\mathbb{F}})$ then there exists an $\mathfrak{N} \in \mathcal{L}^{\leq p}(W_{\mathbb{F}})$ so that $\mathfrak{M} \subset f_*\mathfrak{N}$ with $\mathfrak{M}[\frac{1}{u}] = (f_*\mathfrak{N})[\frac{1}{u}]$. This is essentially \cite[6.3.1]{B18} except that in \emph{loc. cit.} $\mathfrak{M}$ is assumed to be strongly divisible, an assumption which turns out to be unnecessary. To prove the claim consider the map $V_{\mathbb{F}}|_{G_{L_\infty}} \rightarrow W_{\mathbb{F}}$ corresponding to $V_{\mathbb{F}} \cong \operatorname{Ind}_{G_L}^{G_K} W_{\mathbb{F}}$ under Frobenius reciprocity. Lemma~\ref{exactO} below produces a $\varphi$-equivariant surjection $f^*\mathfrak{M} \rightarrow \mathfrak{N}$ for some $\mathfrak{N} \in \mathcal{L}^{\leq p}(W_{\mathbb{F}})$. Via the usual adjunction between $f_*$ and $f^*$ we obtain a non-zero map
	$$
	\mathfrak{M} \rightarrow f_*f^*\mathfrak{M} \rightarrow f_*\mathfrak{N}
	$$
	which is easily checked to be $\varphi$-equivariant. This map must be injective since a non-zero kernel would induce a non-zero $G_{K_\infty}$-subspace of $V_{\mathbb{F}}$. It must be an isomorphism after inverting $u$ because both $\mathfrak{M}$ and $f_*\mathfrak{N}$ have the same rank as $k[[u]]$-modules.
	\end{sub}

	\begin{lemma}\label{exactO}
		Let $0 \rightarrow W_{\mathbb{F}} \rightarrow V_{\mathbb{F}} \rightarrow Z_{\mathbb{F}} \rightarrow 0$ be a $G_{K_\infty}$-equivariant exact sequence. If $\mathfrak{M} \in \mathcal{L}^{\leq h}(V_{\mathbb{F}})$ then there exists $\mathfrak{W} \in \mathcal{L}^{\leq h}(W_{\mathbb{F}})$ and $\mathfrak{Z} \in \mathcal{L}^{\leq h}(Z_{\mathbb{F}})$ together with $\varphi$-equivariant exact sequence 
		$$
		0 \rightarrow \mathfrak{W} \rightarrow \mathfrak{M} \rightarrow \mathfrak{Z} \rightarrow 0
		$$
		which identifies with $0 \rightarrow W_{\mathbb{F}} \rightarrow V_{\mathbb{F}} \rightarrow Z_{\mathbb{F}} \rightarrow 0$ after base-changing to $W(C^\flat)$.
	\end{lemma}
	\begin{proof}
	Since the equivalence between $G_{K_\infty}$-representations and etale $\varphi$-modules is exact there is a $\varphi$-equivariant exact sequence $0 \rightarrow N_{\mathbb{F}} \rightarrow M_{\mathbb{F}} \rightarrow P_{\mathbb{F}} \rightarrow 0$ of etale $\varphi$-modules which identifies with $0 \rightarrow W_{\mathbb{F}} \rightarrow V_{\mathbb{F}} \rightarrow Z_{\mathbb{F}} \rightarrow 0$ after base-change to $W(C^\flat)$. Take $\mathfrak{W} = \mathfrak{M} \cap N_{\mathbb{F}}$ and $\mathfrak{Z} = \operatorname{Im}(\mathfrak{M}) \subset P_{\mathbb{F}}$. It clear both are $\varphi$-stable projective $\mathfrak{S}_{\mathbb{F}}$-modules. It is also clear that $u^h \mathfrak{Z} \subset \mathfrak{Z}^\varphi$ since the the same is true of $\mathfrak{M}$, and so $\mathfrak{Z} \in \mathcal{L}^{\leq h}(Z_{\mathbb{F}})$. Since $\mathfrak{Z}$ is $u$-torsionfree, $u^h\mathfrak{W} = u^h\mathfrak{M} \cap N_{\mathbb{F}}$. As $\mathfrak{W}^\varphi = \mathfrak{M}^\varphi \cap N_{\mathbb{F}}$ we conclude $u^h\mathfrak{W} \subset \mathfrak{M}^\varphi$.
	\end{proof}
	\begin{sub}\label{explicitrankone}
	Return to the proof of Proposition~\ref{irredequal} and fix $\mathfrak{N} \in \mathcal{L}^{\leq p}(W_{\mathbb{F}})$ as in \ref{contain}. Since $W_{\mathbb{F}}$ is one-dimensional we can describe $\mathfrak{N}$ explicitly. We may suppose that $l$, the residue field of $L$, admits an embedding into $\mathbb{F}$. In this case $\mathfrak{S}_{L,\mathbb{F}} = l[[u]] \otimes_{\mathbb{F}_p} \mathbb{F} = \prod_{\theta \in \operatorname{Hom}_{\mathbb{F}_p}(l,\mathbb{F})} \mathbb{F}[[u]]$, where the identification is such that $l$ acts on the $\theta$-th component of the product through $\theta\colon l \rightarrow \mathbb{F}$. Viewing $\mathfrak{N}$ as an $\mathbb{F}[[u]]$-module via the diagonal embedding into $\mathfrak{S}_{L,\mathbb{F}}$, it follows from \cite[6.1.1]{B18} that $\mathfrak{N}$ admits an $\mathbb{F}[[u]]$-basis $(e_\theta)_{\theta \in \operatorname{Hom}_{\mathbb{F}_p}(l,\mathbb{F})}$ satisfying 
	\begin{equation}\label{explicitlydone}
	\varphi(e_{\theta \circ \varphi}) = xu^{r_\theta} e_\theta
	\end{equation}
	for some $x \in l \otimes_{\mathbb{F}_p} \mathbb{F}$ and integers $r_\theta \geq 0$. Since $\mathfrak{N} \in \mathcal{L}^{\leq p}(W_{\mathbb{F}})$ we have $r_{\theta} \in [0,p]$. This basis is chosen so that $l$ acts on $e_\theta$ through $\theta$. 
	\end{sub}
	\begin{sub}\label{equivalent}
	By twisting $V_{\mathbb{F}}$, and so $W_{\mathbb{F}}$, by an unramified character, which is harmless by Lemma~\ref{unramtwist!}, we may assume that $x$ in \eqref{explicitlydone} equals $1$. Under this assumption, \cite[6.5.1]{B18} says that a finite free $\mathfrak{S}_{\mathbb{F}}$-submodule $\mathfrak{M} \subset f_*\mathfrak{N}$ satisfying $\mathfrak{M}[\frac{1}{u}] = (f_*\mathfrak{N})[\frac{1}{u}]$ is an element of $\mathcal{L}^{\leq p}_{\operatorname{SD}}(V_{\mathbb{F}})$ if and only if: 
	\begin{enumerate}
		\item If $m \in \mathfrak{M}$ then $\varphi(m) \in \mathfrak{M}$, and if $\varphi(m) \in u^{p+1}\mathfrak{M}$ then $m \in u \mathfrak{M}$.
		\item For every $\mathbb{F}$-linear combination $\sum \alpha_\theta e_\theta$ which is contained in $\mathfrak{M}$, and every $0<r \leq p$, the $\mathbb{F}$-linear combination
		$$
		\sum_{r_\theta \equiv r~\operatorname{mod} p}\alpha_\theta e_\theta
		$$
		is contained in $\mathfrak{M}$ also.
	\end{enumerate}
	Observe that $u^p\mathfrak{M} \subset \mathfrak{M}^\varphi \subset \mathfrak{M}$ implies (1). Indeed, if $\varphi(m) \in u^{p+1}\mathfrak{M}$ then $\varphi(m) \in u\mathfrak{M}^\varphi$. If $e_i$ is a $k[[u]]$-basis of $\mathfrak{M}$ then $m = \sum \alpha_i e_i$ for $\alpha_i \in k[[u]]$ and $\varphi(m) = \sum \varphi(\alpha_i)\varphi(e_i)$; by definition the $\varphi(e_i)$ form a $k[[u]]$-basis of $\mathfrak{M}^\varphi$ so if $\varphi(m) \in u\mathfrak{M}^\varphi$ we must have each $\varphi(\alpha_i)$ divisible by $u$. This implies each $\alpha_i$ is also divisible by $u$. 
	
	Another consequence of (1) is that $ue_\theta \in \mathfrak{M}$ for every $\theta \in \operatorname{Hom}_{\mathbb{F}_p}(l,\mathbb{F})$. This is explained in the second paragraph after \cite[6.5.1]{B18}.
	\end{sub}
	\begin{sub}\label{sdcondition}
	To finish the proof we have to show that, if $\mathfrak{M} \in \mathcal{L}^{\leq p}(V_{\mathbb{F}})$ is contained in $f_*\mathfrak{N}$, then (2) is satisfied if and only if $\mathfrak{M} \in \mathcal{L}^{\leq p}_{\operatorname{crys}}(V_\mathbb{F})$. The $G_K$-action on $V_{\mathbb{F}}$ induces a continuous $C^\flat \otimes_{\mathbb{F}_p} \mathbb{F}$-semilinear $\varphi$-equivariant action of $G_K$ on $V_{\mathbb{F}}\otimes_{\mathbb{F}_p} C^\flat = \mathfrak{M} \otimes_{k[[u]]} C^\flat$. Conversely any such semilinear $G_K$-action induces a $G_K$-action on $V_{\mathbb{F}}$ extending the $G_{K_\infty}$-action. Thus \ref{unique} implies there is at most one such semilinear $G_K$-action. This semilinear action of $G_K$ can be written explicitly as follows:
	\begin{equation}\label{explicitGalois}
	\sigma(e_\theta) = \eta(\sigma)^{\Theta_\theta} e_\theta, \qquad \sigma \in G_K
	\end{equation}
	where $\Theta_\theta = \sum_{i=0}^{[l:\mathbb{F}_p]-1} p^ir_{\theta \circ \varphi^i}$ and $\eta(\sigma) \in \mathcal{O}_{C^\flat}$ is the unique $p^{[l:\mathbb{F}_p]} -1$-th root of $\sigma(u)/u$ whose image in the residue field of $\mathcal{O}_{C^\flat}$ is $1$. To verify this it suffices to check this does indeed define a group action, that this action is continuous, that it induces the trivial action of $G_{K_\infty}$ on $(f_*\mathfrak{N})[\frac{1}{u}] = \mathfrak{M}[\frac{1}{u}]$, and is $\varphi$-equivariant. The first three are straightforward to check, and checking the $\varphi$-equivariance comes down to the identity
	$$
	\sigma(u^{r_\theta}) \eta(\sigma)^{\Theta_\theta} = u^{r_\theta} \eta(\sigma)^{p\Theta_{\theta \circ \varphi}}
	$$ 
	which follows since 
	$$
	p\Theta_{\theta \circ \varphi} = \sum_{i=0}^{[l:\mathbb{F}_p]-1} p^{i+1} r_{\theta\circ \varphi^i} = (p^{[l:\mathbb{F}_p]-1}-1)r_\theta + \sum_{i=0}^{[l:\mathbb{F}_p]-1} p^ir_{\theta \circ \varphi^i}
	$$
	Therefore this must be the $G_K$-action coming from that on $V_{\mathbb{F}}$. To check the condition from \ref{simply} we shall need:
\end{sub}
	\begin{lemma}\label{valu}	
	For $\sigma \in G_K$ let $m = m(\sigma)$ be such that $\sigma(u)/u \in \mathbb{Z}_p(1)$ is a $\mathbb{Z}_p$-generator of $p^m \mathbb{Z}_p(1)$. Then, for $n \geq 0$,
	\begin{equation*}
	v^\flat(\eta(\sigma)^n -1) = \frac{p^{1+m + v_p(n)}}{p-1}
	\end{equation*}
	\end{lemma}
	\begin{proof}
		This easily reduces to the well-known calculation that $v^\flat(\epsilon -1) = p/(p-1)$ for any $\mathbb{Z}_p$-generator $\epsilon \in \mathbb{Z}_p(1)$, cf. for example \cite[\S5.1.2]{Fon94}.
	\end{proof}
	We have to show that (2) is equivalent to asking that $(\sigma-1)(m) \in \mathfrak{M} \otimes_{k[[u]]} u^{p/p-1}\mathcal{O}_{C^\flat}$ for every $m \in \mathfrak{M}$ and $\sigma \in G_K$ (cf. \ref{simply}). When $m = u^ie_\theta$ for $i \geq 1$ this follows easily from Lemma~\ref{valu} since $(\sigma-1)(u^ie_\theta) = ((\frac{\sigma(u)}{u})^i\eta(\sigma)^{\Theta_\theta} - 1)(u^ie_\theta) = (\eta(\sigma)^{\Theta_\theta + (p^{[l:\mathbb{F}_p]}-1)i} - 1)(u^ie_\theta)$. To complete the proof  we consider elements $\sum \alpha_\theta e_\theta \in \mathfrak{M}$ with $\alpha_\theta \in \mathbb{F}$. 
	We compute that 
	\begin{equation*}
	\begin{aligned}
	\frac{(\sigma-1)(\sum \alpha_\theta e_\theta)}{\eta(\sigma)-1} &= \sum \left( \frac{\eta(\sigma)^{\Theta_\theta} -1}{\eta(\sigma) -1} \right) \alpha_\theta e_\theta \\
	&= \sum \left( 1+ \eta(\sigma) + \ldots \eta(\sigma)^{\Theta_\theta -1} \right) \alpha_\theta e_\theta\\
	& = \sum r_\theta \alpha_\theta e_\theta + \sum \beta_\theta ue_\theta \text{ for some $\beta_\theta \in \mathcal{O}_{C^\flat}$}
 	\end{aligned}
	\end{equation*}
	The last equality follows because  $\eta(\sigma) -1 \in u^{p/p-1}\mathcal{O}_{C^\flat}$ by Lemma~\ref{valu}, and so $1 +\eta(\sigma) + \ldots + \eta(\sigma)^{\Theta_\theta -1} \equiv \Theta_\theta \equiv r_\theta$ modulo~$u^{p/p-1}\mathcal{O}_{C^\flat}$. Since $ue_\theta \in \mathfrak{M}$ for every $\theta$, it follows that $\frac{(\sigma-1)(\sum \alpha_\theta e_\theta)}{\eta(\sigma)- 1} \in \mathfrak{M} \otimes_{k[[u]]} \mathcal{O}_{C^\flat}$ if and only if 
	$$
	\sum r_\theta \alpha_\theta e_\theta \in \mathfrak{M}
	$$
	Since $v^\flat(\eta(\sigma) -1) \geq p/(p-1)$, with equality when $\sigma$ is chosen so that $\sigma(u)/u$ is a $\mathbb{Z}_p$-generator of $\mathbb{Z}_p(1)$, we conclude that $(\sigma-1)(\sum \alpha_\theta e_\theta) \in \mathfrak{M} \otimes_{k[[u]]} u^{p/p-1}\mathcal{O}_{C^\flat}$ for every $m \in \mathfrak{M}$ and $\sigma \in G_K$ if and only if for every $\mathbb{F}$-linear combination $\sum \alpha_\theta e_\theta \in \mathfrak{M}$ we have $\sum r_\theta\alpha_\theta e_\theta \in \mathfrak{M}$. It is easy to check the latter condition is equivalent to (2).
\end{proof}

\subsection{Strong divisibility in general}
\begin{proposition}\label{inclusion}
	Suppose $V_{\mathbb{F}}$ admits a continuous $\mathbb{F}$-linear $G_K$-action. Then $\mathcal{L}^{\leq p}_{\operatorname{crys}}(V_{\mathbb{F}}) \subset \mathcal{L}^{\leq p}_{\operatorname{SD}}(V_{\mathbb{F}})$.\footnote{Unlike in the irreducible case this inclusion is not always an equality. The problem arises from the possibility that $V_{\mathbb{F}}$ may admit two different $G_K$-actions extending a given $G_{K_\infty}$-action. Here is an example: suppose $V_{\mathbb{F}}$ admits an $\mathbb{F}$-basis $(f_1,f_2)$ so that 
	$$
	\sigma(f_1,f_2) = (f_1,f_2)\left( \begin{smallmatrix}
	1 & c(\sigma) \\ 0 & \chi_{\operatorname{cyc}}^{-1}(\sigma)
	\end{smallmatrix} \right)
	$$
	for a $1$-cocycle $c(\sigma)$. We compute that
	$$
	\sigma(f_1,u^{1/p-1}f_2) = (f_1,u^{1/p-1}f_2)\left( \begin{smallmatrix}
		1 & \sigma(u^{1/p-1})c(\sigma) \\ 0 & \frac{\sigma(u^{1/p-1})}{u^{1/p-1}}\chi_{\operatorname{cyc}}^{-1}(\sigma)
	\end{smallmatrix} \right)
	$$
	There exists cocycle $c$ such that $c(\sigma) = 0$ for $\sigma \in G_{K_\infty}$; this occurs when $V_{\mathbb{F}}$ is a tres ramifie extension (cf. \cite[5.4.2]{GLS15}).
	In this case the matrix representing $\sigma$ on $(f_1,u^{1/p-1}f_2)$ is the identity when $\sigma \in G_{K_\infty}$ so $\mathfrak{M}$, the $\mathfrak{S}_{\mathbb{F}}$-span of $f_1$ and $u^{1/p-1}f_2$, is contained in the etale $\varphi$-module associated to $V_{\mathbb{F}}$. Since $\varphi(f_1,u^{1/p-1}f_2) = (f_1,u^{1/p-1}f_2)\left( \begin{smallmatrix}
	1 & 0  \\ 0 & u
	\end{smallmatrix} \right)$ it is easy to see that $\mathfrak{M} \in \mathcal{L}^{\leq 1}_{\operatorname{SD}}(V_{\mathbb{F}})$. However $\mathfrak{M} \not\in \mathcal{L}^{\leq 1}_{\operatorname{crys}}(V_{\mathbb{F}})$ since $u^{1/p-1}c(\sigma)$ is not contained in $u^{p/p-1}\mathcal{O}_{C^\flat}$. The point is that $V_{\mathbb{F}}$ does not arise as the reduction modulo~$p$ of a crystalline representation with Hodge--Tate weights in $[0,1]$.
	}
\end{proposition}

	To prove this we will need to understand how $\mathcal{L}^{\leq p}_{\operatorname{SD}}$ and $\mathcal{L}^{\leq p}_{\operatorname{crys}}$ behave in short exact sequences.

\begin{sub}\label{ext}
	Let $0 \rightarrow W_{\mathbb{F}} \rightarrow V_{\mathbb{F}} \rightarrow Z_{\mathbb{F}} \rightarrow 0$ be a $G_{K_\infty}$-equivariant exact sequence and let $\mathfrak{M} \in \mathcal{L}^{\leq p}(V_{\mathbb{F}})$. Lemma~\ref{exactO} provides a $\varphi$-equivariant exact sequence
	\begin{equation}\label{exact}
	0 \rightarrow \mathfrak{W} \rightarrow \mathfrak{M} \rightarrow \mathfrak{Z} \rightarrow 0
	\end{equation}
	with $\mathfrak{W} \in \mathcal{L}^{\leq h}(W_{\mathbb{F}})$ and $\mathfrak{Z} \in \mathcal{L}^{\leq h}(Z_{\mathbb{F}})$. Choosing an $\mathfrak{S}_{\mathbb{F}}$-splitting of \eqref{exact} allows us to identify $\mathfrak{M} = \mathfrak{W} \oplus \mathfrak{Z}$ as $\mathfrak{S}_{\mathbb{F}}$-modules so that 
	$$
	\varphi_{\mathfrak{M}} = (\varphi_{\mathfrak{W}} + f \circ \varphi_{\mathfrak{Z}}, \varphi_{\mathfrak{Z}})
	$$
	for some $f \in \operatorname{Hom}(\mathfrak{Z},\mathfrak{W})[\frac{1}{u}]$. Here $\operatorname{Hom}(\mathfrak{Z},\mathfrak{W})$ denotes the module of $\mathfrak{S}_{\mathbb{F}}$-linear homomorphisms $\mathfrak{Z} \rightarrow \mathfrak{W}$. Since $\mathfrak{M}^\varphi \subset \mathfrak{M}$ we must have $f(\mathfrak{Z}^\varphi) \subset \mathfrak{W}$ and so, as $u^p \mathfrak{Z} \subset \mathfrak{Z}^\varphi$, it follows that $f \in \frac{1}{u^p} \operatorname{Hom}(\mathfrak{Z},\mathfrak{W})$.
	
	We equip $\operatorname{Hom}(\mathfrak{Z},\mathfrak{W})$ with the Frobenius $\varphi$ given by $\varphi(g) = \varphi_{\mathfrak{W}} \circ g \circ \varphi^{-1}_{\mathfrak{Z}}$. 	Since any two splitting of \eqref{exact} differ by an element $g \in \operatorname{Hom}(\mathfrak{Z},\mathfrak{W})$, by choosing a different splitting we replace $f$ by $f + (\varphi - 1)(g)$. If as usual $\operatorname{Hom}(\mathfrak{Z},\mathfrak{W})^\varphi$ denotes the $\mathfrak{S}_{\mathbb{F}}$-submodule of $\operatorname{Hom}(\mathfrak{Z},\mathfrak{W})[\frac{1}{E}]$ generated by $\varphi(\operatorname{Hom}(\mathfrak{Z},\mathfrak{W}))$ then
	\begin{equation}\label{homheight}
	u^p\operatorname{Hom}(\mathfrak{Z},\mathfrak{W}) \subset \operatorname{Hom}(\mathfrak{Z},\mathfrak{W})^\varphi \subset \tfrac{1}{u^p} \operatorname{Hom}(\mathfrak{Z},\mathfrak{W})
	\end{equation} 
	Furthermore, we can $G_K$-equivariantly identify $(\operatorname{Hom}(\mathfrak{Z},\mathfrak{W}) \otimes_{k[[u]]} C^\flat)^{\varphi = 1} = \operatorname{Hom}(Z_{\mathbb{F}},W_{\mathbb{F}})$
	(the $G_K$-action on $\operatorname{Hom}(Z_{\mathbb{F}},W_{\mathbb{F}})$ being given by $f \mapsto \sigma \circ f \circ \sigma^{-1}$) via the identifications $(\mathfrak{Z} \otimes_{k[[u]]} C^\flat)^{\varphi =1} = Z_{\mathbb{F}}$ and $(\mathfrak{W} \otimes_{k[[u]]} C^\flat)^{\varphi= 1} = W_{\mathbb{F}}$. 
\end{sub}

	\begin{proposition}\label{SDexact} In the situation of \ref{ext}:
		\begin{enumerate}
			\item If $\mathfrak{M} \in \mathcal{L}^{\leq p}_{\operatorname{SD}}(V_{\mathbb{F}})$ then $\mathfrak{W} \in \mathcal{L}^{\leq p}_{\operatorname{SD}}(W_{\mathbb{F}})$ and $\mathfrak{Z} \in \mathcal{L}^{\leq p}_{\operatorname{SD}}(Z_{\mathbb{F}})$.
			\item If $\mathfrak{W} \in \mathcal{L}^{\leq p}_{\operatorname{SD}}(W_{\mathbb{F}})$ and $\mathfrak{Z} \in \mathcal{L}^{\leq p}_{\operatorname{SD}}(Z_{\mathbb{F}})$ then $\mathfrak{M} \in \mathcal{L}^{\leq p}_{\operatorname{SD}}(V_{\mathbb{F}})$ if and only if there exists $g \in \operatorname{Hom}(\mathfrak{Z},\mathfrak{W})$ such that
			$$
			f + (\varphi- 1)(g) \in \operatorname{Hom}(\mathfrak{Z},\mathfrak{W})
			$$
		\end{enumerate}  
	\end{proposition}
	\begin{proof}
		Part (1) is a consequence of Lemma~\ref{exactsequences}, while (2) follows from \cite[Lemma 4.1.3]{B18b}.
	\end{proof}

	\begin{sub}\label{gal}
	Now suppose $0 \rightarrow W_{\mathbb{F}} \rightarrow V_{\mathbb{F}} \rightarrow Z_{\mathbb{F}} \rightarrow 0$ is an exact sequence of $G_{K}$-representations. As in \ref{ext}, for any $\mathfrak{M} \in \mathcal{L}^{\leq p}(V_\mathbb{F})$, there is an exact sequence $0 \rightarrow \mathfrak{W} \rightarrow \mathfrak{M} \rightarrow \mathfrak{Z} \rightarrow 0$ so that, after choosing a splitting of this sequence and identifying $\mathfrak{M} = \mathfrak{W} \oplus \mathfrak{Z}$, we have $\varphi_{\mathfrak{M}} =(\varphi_{\mathfrak{W}} + f \circ \varphi_{\mathfrak{Z}},\varphi_{\mathfrak{Z}})$ for some $f \in \frac{1}{u^p}\operatorname{Hom}(\mathfrak{Z},\mathfrak{W})$.
	
	As $0 \rightarrow W_{\mathbb{F}} \rightarrow V_{\mathbb{F}} \rightarrow Z_{\mathbb{F}} \rightarrow 0$ and $0 \rightarrow \mathfrak{W} \rightarrow \mathfrak{M} \rightarrow \mathfrak{Z} \rightarrow 0$ become identified after applying $ \otimes_{k[[u]]} C^\flat$ we obtain compatible $\varphi$-equivariant $G_K$-actions on $\mathfrak{M} \otimes_{k[[u]]} C^\flat, \mathfrak{W}^\flat \otimes_{k[[u]]} C^\flat$ and $\mathfrak{Z}^\flat\otimes_{k[[u]]} C^\flat$. Under the identification $\mathfrak{M}= \mathfrak{W}\oplus \mathfrak{Z}$ the action of $\sigma \in G_K$ can be written as
	$$
	\sigma_{\mathfrak{M}} = (\sigma_{\mathfrak{W}} + f_{\sigma} \circ \sigma_{\mathfrak{Z}}, \sigma_{\mathfrak{Z}})
	$$
	for some $f_\sigma \in \operatorname{Hom}(\mathfrak{Z},\mathfrak{W})\otimes_{k[[u]]} C^\flat$ satisfying the following conditions:
	\begin{enumerate}
		\item Since $\sigma_{\mathfrak{M}}$ is a group action we must have $f_{\sigma\tau} = f_\sigma + \sigma_{\mathfrak{W}} \circ f_\tau \circ \sigma_{\mathfrak{Z}}^{-1}$. If we equip $\operatorname{Hom}(\mathfrak{Z},\mathfrak{W}) \otimes_{k[[u]]} C^\flat$ with the $G_K$-action given by $\sigma(f) = \sigma_{\mathfrak{W}} \circ f \circ \sigma_{\mathfrak{Z}}^{-1}$ then this says that $\sigma \mapsto f_\sigma$ is a $1$-cocycle valued in $\operatorname{Hom}(\mathfrak{Z},\mathfrak{W}) \otimes_{k[[u]]} C^\flat$. Since the $G_K$-action on $V_{\mathbb{F}}$ is continuous $\sigma \mapsto f_\sigma$ must also be a continuous cocycle.
		\item Since the $G_{K_\infty}$-action on $V_{\mathbb{F}}$ is induced by the trivial action on $\mathfrak{M}$, we must have $\sigma_{\mathfrak{M}}(m) = m$ for every $m \in \mathfrak{M}$ and $\sigma \in G_{K_\infty}$. Thus we must have $f_\sigma(m) = 0$ whenever $m \in \mathfrak{Z}$ and $\sigma \in G_{K_\infty}$.
		\item Since $\sigma_{\mathfrak{M}}$ is $\varphi$-equivariant we must have $(\varphi -1)(f_\sigma) = (\sigma-1)(f)$ for any $\sigma \in G_K$.
	\end{enumerate}
\end{sub}
	\begin{proposition}\label{Galexact} In the situation of \ref{gal}:
		\begin{enumerate}
			\item If $\mathfrak{M} \in \mathcal{L}^{\leq p}_{\operatorname{crys}}(V_{\mathbb{F}})$ then $\mathfrak{W} \in \mathcal{L}^{\leq p}_{\operatorname{crys}}(W_{\mathbb{F}})$ and $\mathfrak{Z} \in \mathcal{L}^{\leq p}_{\operatorname{crys}}(Z_{\mathbb{F}})$.
			\item If $\mathfrak{W} \in \mathcal{L}^{\leq p}_{\operatorname{crys}}(W_{\mathbb{F}})$ and $\mathfrak{Z} \in \mathcal{L}^{\leq p}_{\operatorname{crys}}(Z_{\mathbb{F}})$ then $\mathfrak{M} \in \mathcal{L}^{\leq p}_{\operatorname{crys}}(V_{\mathbb{F}})$ if and only if $f_\sigma \in \operatorname{Hom}(\mathfrak{Z},\mathfrak{W}) \otimes_{k[[u]]} u^{p/p-1} \mathcal{O}_{C^\flat}$ for every $\sigma \in G_K$.
		\end{enumerate}
	\end{proposition}
	\begin{proof}
		For the second statement combine \ref{simply} with \ref{gal}. For the first, as $0 \rightarrow \mathfrak{W} \rightarrow \mathfrak{M} \rightarrow \mathfrak{Z} \rightarrow 0$ becomes $G_K$-equivariant after applying $\otimes_{\mathfrak{S}} W(C^\flat)$, it is clear that $(\sigma-1)(z) \in \mathfrak{Z} \otimes_{k[[u]]} u^{p/p-1}\mathcal{O}_{C^\flat}$ for $z \in \mathfrak{Z}$. Thus $\mathfrak{Z} \in \mathcal{L}^{\leq p}_{\operatorname{crys}}(V_{\mathbb{F}})$. If $n \in \mathfrak{W}$ then $(\sigma-1)(n) \in \mathfrak{M} \otimes_{k[[u]]} u^{p/p-1}\mathcal{O}_{C^\flat} \cap \mathfrak{W} \otimes_{k[[u]]} C^\flat$; this intersection equals $\mathfrak{W} \otimes_{k[[u]]} u^{p/p-1}\mathcal{O}_{C^\flat}$ because $\mathfrak{Z}$ is $u$-torsionfree, and so $\mathfrak{W} \in \mathcal{L}^{\leq p}_{\operatorname{crys}}(V_{\mathbb{F}})$ also.
	\end{proof}
	
\begin{proof}[Proof of Proposition~\ref{inclusion}] Using Lemma~\ref{extend} we can replace $\mathbb{F}$ by a finite extension. As explained in the beginning of the proof of Proposition~\ref{irredequal}, this allows us to assume each Jordan--Holder factor of $V_{\mathbb{F}}$ is induced from a one-dimensional representation over an unramified extension of $K$. Using (1) of Lemma~\ref{inducerestrict} we may then replace $K$ by a suitably large (but finite) unramified extension so that  every Jordan--Holder factor of $V_{\mathbb{F}}$ is one-dimensional. Under this assumption we argue by induction on the length (equivalently the dimension) of $V_{\mathbb{F}}$. 
	
	The base case of the induction is handled by Proposition~\ref{irredequal}. Thus we can assume $V_{\mathbb{F}}$ fits into a $G_K$-equivariant exact sequence $0 \rightarrow W_{\mathbb{F}} \rightarrow V_{\mathbb{F}} \rightarrow Z_{\mathbb{F}} \rightarrow 0$ with $Z_{\mathbb{F}}$ one-dimensional over $\mathbb{F}$ and $W_{\mathbb{F}} \neq 0$. 

As in \ref{ext} if $\mathfrak{M} \in \mathcal{L}^{\leq p}(V_{\mathbb{F}})$ we obtain an exact sequence $0 \rightarrow \mathfrak{W} \rightarrow \mathfrak{M} \rightarrow \mathfrak{Z} \rightarrow 0$ with $\mathfrak{W} \in \mathcal{L}^{\leq p}(W_{\mathbb{F}})$ and $\mathfrak{Z} \in \mathcal{L}^{\leq p}(Z_{\mathbb{F}})$. By choosing a splitting of this sequence we identify $\mathfrak{M} = \mathfrak{W} \oplus \mathfrak{Z}$ as $\mathfrak{S}_\mathbb{F}$-modules, with Frobenius given by 
$$\varphi_{\mathfrak{M}} = (\varphi_{\mathfrak{W}} + f \circ \varphi_{\mathfrak{Z}},\varphi_{\mathfrak{Z}})
$$
for some $f \in \frac{1}{u^p} \operatorname{Hom}(\mathfrak{Z},\mathfrak{W})$. As in \ref{gal} the $G_K$-action on $\mathfrak{M} \otimes_{k[[u]]} C^\flat$ induced by the $G_K$-action on $V_{\mathbb{F}}$ may be written as
$$
\sigma_{\mathfrak{M}} = (\sigma_{\mathfrak{W}} + f_{\sigma} \circ \sigma_{\mathfrak{Z}}, \sigma_{\mathfrak{Z}})
$$
for some $f_\sigma \in \operatorname{Hom}(\mathfrak{Z},\mathfrak{W}) \otimes_{k[[u]]} C^\flat$ satisfying $(\varphi-1)(f_\sigma) = (\sigma- 1)(f)$. If $\mathfrak{M} \in \mathcal{L}^{\leq p}_{\operatorname{crys}}(V_{\mathbb{F}})$ then $\mathfrak{W} \in \mathcal{L}^{\leq p}_{\operatorname{crys}}(W_{\mathbb{F}})$, $\mathfrak{Z} \in \mathcal{L}^{\leq p}_{\operatorname{crys}}(Z_{\mathbb{F}})$, and 
\begin{equation}\label{fsigmacond}
f_\sigma \in \operatorname{Hom}(\mathfrak{Z},\mathfrak{W}) \otimes_{k[[u]]} u^{p/p-1}\mathcal{O}_{C^\flat}
\end{equation}
by Proposition~\ref{Galexact}. By induction $\mathfrak{W} \in \mathcal{L}^{\leq p}_{\operatorname{SD}}(W_{\mathbb{F}})$ and $\mathfrak{Z} \in \mathcal{L}^{\leq p}_{\operatorname{SD}}(Z_{\mathbb{F}})$. By Proposition~\ref{SDexact}, $\mathfrak{M} \in \mathcal{L}^{\leq p}_{\operatorname{SD}}(V_{\mathbb{F}})$ if and only if
\begin{equation*}\label{fcond}
f \in \operatorname{Hom}(\mathfrak{Z},\mathfrak{W}) + \varphi(\operatorname{Hom}(\mathfrak{Z},\mathfrak{W}))
\end{equation*}
Using \eqref{homheight} and \eqref{fsigmacond} we see $(\varphi-1)(f_\sigma) = (\sigma-1)(f) \in \operatorname{Hom}(\mathfrak{Z},\mathfrak{W}) \otimes_{k[[u]]} u^{p/p-1}\mathcal{O}_{C^\flat}$. Thus the proposition follows from the following claim.
\end{proof}
\begin{claim}
	Any $f \in \frac{1}{u^p}\operatorname{Hom}(\mathfrak{Z},\mathfrak{W})$ satisfying $(\sigma-1)(f) \in \operatorname{Hom}(\mathfrak{Z},\mathfrak{W}) \otimes_{k[[u]]} u^{p/p-1} \mathcal{O}_{C^\flat}$ must be contained in $\operatorname{Hom}(\mathfrak{Z},\mathfrak{W}) + \varphi(\operatorname{Hom}(\mathfrak{Z},\mathfrak{W}))$. 
\end{claim}
\begin{proof}[Proof of claim]
	We argue by a further induction, this time on the length of $W_{\mathbb{F}}$. Recall that by assumption every Jordan--Holder factor of $W_{\mathbb{F}}$ is one dimensional. Thus the base case is when both $W_{\mathbb{F}}$ and $Z_{\mathbb{F}}$ is one dimensional. In this case, as explained in \ref{explicitrankone}, $\mathfrak{W}$ and $\mathfrak{Z}$ respectively admit $\mathbb{F}[[u]]$-bases $(w_\tau)_{\tau \in \operatorname{Hom}_{\mathbb{F}_p}(k,\mathbb{F})}$ and $(z_\tau)_{\tau \in \operatorname{Hom}_{\mathbb{F}_p}(k,\mathbb{F})}$ so that
	$$
	\varphi(w_{\tau \circ \varphi}) = x u^{r_\tau} w_\tau, \qquad \varphi(z_{\tau \circ \varphi}) = y u^{s_\tau} z_\tau
	$$
	for $x,y \in (k \otimes_{\mathbb{F}_p} \mathbb{F})^\times$ and $r_\tau, s_\tau \in [0,p]$. The $\mathbb{F}[[u]]$-linear homomorphism $F_\tau: \mathfrak{Z} \rightarrow \mathfrak{W}$ sending $z_{\tau'} \mapsto 0$ for $\tau'\neq \tau$ and $z_\tau \mapsto w_\tau$ is $\mathfrak{S}_{\mathbb{F}}$-linear since it is compatible with the $k$-action on $\mathfrak{Z}$ and $\mathfrak{W}$ (by construction $k$-acts on $z_\tau$ and $w_\tau$ by $\tau$). Thus $F_\tau \in \operatorname{Hom}(\mathfrak{Z},\mathfrak{W})$ and together the $F_\tau$ form an $\mathbb{F}[[u]]$-basis of $\operatorname{Hom}(\mathfrak{Z},\mathfrak{W})$ satisfying $\varphi(F_{\tau \circ \varphi}) = xy^{-1} u^{t_\tau}F_\tau$ for $t_\tau = r_\tau - s_\tau \in [-p,p]$. Since the $G_K$-actions on $\mathfrak{Z} \otimes_{k[[u]]} C^\flat$ and $\mathfrak{W} \otimes_{k[[u]]} C^\flat$ are as in \eqref{explicitGalois} we also have that
	$$
	\sigma(F_\tau) = \eta(\sigma)^{\Theta_\tau}F_\tau, \qquad \Theta_\tau = \sum_{i=0}^{[k:\mathbb{F}_p]-1} t_{\tau \circ \varphi^i}p^i
	$$
	To prove the claim it suffices to consider $f = (\sum_{i \geq -p} a_i u^i) F_\tau$. Then $(\sigma-1)(f)$ equals 
	$$
	\sum_{i \geq -p} a_iu^i \Big( \underbrace{\eta(\sigma)^{\Theta_\tau + (p^{[k:\mathbb{F}_p]}-1)i} -1  }_{(1)}\Big) F_\tau
	$$
	(recall that $\eta(\sigma)$ is a $p^{[k:\mathbb{F}_p]} -1$-th root of $\sigma(u)/u$). Choose $\sigma$ so that $\sigma(u)/u$ is a $\mathbb{Z}_p$-generator of $\mathbb{Z}_p(1)$. Since
	$$
	\Theta_\tau + (p^{[k:\mathbb{F}_p]} -1)i \equiv t_\tau - i \text{ modulo~$p$} 
	$$ 
	Lemma~\ref{valu} implies the $v^\flat$-valuation of (1) is $p/(p-1)$ if $t_\tau -i$ is not divisible by $p$, and is $\geq p^2/(p-1)$ otherwise. Hence 
	$$
	v^\flat\left( u^i\left( \eta(\sigma)^{\Theta_\tau + (p^{[k:\mathbb{F}_p]}-1)i}\right) \right) \begin{cases}
	= p/(p-1) + i & \text{if $p$ does not divide $t_\tau - i$} \\
	\geq p^2/(p-1) + i & \text{otherwise}
	\end{cases}
	$$
	Since $p^2(p-1) +i \geq p/(p-1)$ for $i \geq -p$ it follows that $(\sigma-1)(f) \in \operatorname{Hom}(\mathfrak{Z},\mathfrak{W}) \otimes_{k[[u]]} u^{p/p-1}\mathcal{O}_{C^\flat}$ if and only if $a_i = 0$ except possibly if $i = t_\tau$. In other words, if and only if $f \in \operatorname{Hom}(\mathfrak{Z},\mathfrak{W}) + \varphi(\operatorname{Hom}(\mathfrak{Z},\mathfrak{W}))$.
	
	Now we prove the inductive step. Let $0 \rightarrow W_{\mathbb{F}}^1\rightarrow W_{\mathbb{F}} \rightarrow W^2_{\mathbb{F}} \rightarrow 0$ be an exact sequence of $G_K$-representations. As in \ref{ext} and \ref{gal} we can write $\mathfrak{W} = \mathfrak{W}^1 \oplus \mathfrak{W}^2$ with $\mathfrak{W}^i \in \mathcal{L}^{\leq p}_{\operatorname{crys}}(W_{\mathbb{F}}^i)$, so that $\varphi_{\mathfrak{W}} = (\varphi_{\mathfrak{W}^1}+ g \circ \varphi_{\mathfrak{W}^2},\varphi_{\mathfrak{W}^2})$ for some $g \in \operatorname{Hom}(\mathfrak{W}^2,\mathfrak{W}^1)$, and so that $\sigma_{\mathfrak{W}} = (\sigma_{\mathfrak{W}^1} + g_\sigma \circ \sigma_{\mathfrak{W}^2},\sigma_{\mathfrak{W}^2})$ for some $g_\sigma \in \operatorname{Hom}(\mathfrak{W}^2,\mathfrak{W}^1) \otimes_{k[[u]]} u^{p/p-1}\mathcal{O}_{C^\flat}$. Applying $\operatorname{Hom}(\mathfrak{Z},-)$ this allows us to identify $\mathfrak{H} := \operatorname{Hom}(\mathfrak{Z},\mathfrak{W})$ with $\mathfrak{H}^1 \oplus \mathfrak{H}^2$ where $\mathfrak{H}^i = \operatorname{Hom}(\mathfrak{Z},\mathfrak{W}^i)$, so that 
	$$
	\varphi_{\mathfrak{H}} = (\varphi_{\mathfrak{H}^1} + \widetilde{g} \circ \varphi_{\mathfrak{H}^2},\varphi_{\mathfrak{H}^2}), \qquad \sigma_{\mathfrak{H}} = (\sigma_{\mathfrak{H}^1} + \widetilde{g}_\sigma \circ \sigma_{\mathfrak{H}^2}, \sigma_{\mathfrak{H}^2})
	$$
	where $\widetilde{g} \in \operatorname{Hom}(\mathfrak{H}^2,\mathfrak{H}^1)$ sends $h \mapsto g \circ h$, and where $\widetilde{g}_\sigma \in \operatorname{Hom}(\mathfrak{H}^2,\mathfrak{H}^1)\otimes_{k[[u]]} u^{p/p-1}\mathcal{O}_{C^\flat}$ sends $h \mapsto g_\sigma \circ h$. If we write $f =(f_1,f_2) \in \frac{1}{u^p}(\mathfrak{H}^1 \oplus \mathfrak{H}^2)$ then, as $(\sigma_{\mathfrak{H}}-1)(f) \in \mathfrak{H} \otimes_{k[[u]]} u^{p/p-1}\mathcal{O}_{C^\flat}$, we have
	$$
	\begin{aligned}
	(\sigma_{\mathfrak{H}^2}-1)(f_2) &\in \mathfrak{H}^2 \otimes_{k[[u]]} u^{p/p-1}\mathcal{O}_{C^\flat} \\
	(\sigma_{\mathfrak{H}^1} -1)(f_1) + \widetilde{g}_\sigma \circ \sigma_{\mathfrak{H}^2}(f_2) &\in \mathfrak{H}^1 \otimes_{k[[u]]} u^{p/p-1} \mathcal{O}_{C^\flat}
	\end{aligned}
	$$ 
	By our inductive hypothesis we deduce $f_2 = f_2' + f_2''$ with $f_2' \in \mathfrak{H}^2$ and $f_2'' \in \varphi(\mathfrak{H}^2)$. Thus, we can write 
	$$
	f = \underbrace{(f_1 - \widetilde{g}(f_2''),f_2')}_{:=y} + \underbrace{(\widetilde{g}(f_2''),f_2'')}_{:=z}
	$$ 
	with $z \in \varphi(\mathfrak{H})$. Since $ \varphi(\mathfrak{H}) \subset u^{-p}\mathfrak{H}$ we have $(\sigma-1)(z) \in \mathfrak{H} \otimes_{k[[u]]} u^{p/p-1}\mathcal{O}_{C^\flat}$. Thus $(\sigma-1)(y) \in \mathfrak{H} \otimes_{k[[u]]} u^{p/p-1}\mathcal{O}_{C^\flat}$. We also see $y \in u^{-p}\mathfrak{H}$. This means that to prove the result for $f$ it suffices to do so for $y$, i.e. we can assume in the above that $f_2''= 0$. Thus, $f_2 \in \mathfrak{H}^2$ and so $\sigma_{\mathfrak{H}^2}(f_2) \in \mathfrak{H}^2 \otimes_{k[[u]]} \mathcal{O}_{C^\flat}$. Since $\widetilde{g}_\sigma \in \operatorname{Hom}(\mathfrak{H}^2,\mathfrak{H}^1) \otimes_{k[[u]]} u^{p/p-1} \mathcal{O}_{C^\flat}$ it follows that $\widetilde{g}_\sigma \circ \sigma_{\mathfrak{H}^2}(f_2) \in \mathfrak{H}^1 \otimes_{k[[u]]} u^{p/p-1}\mathcal{O}_{C^\flat}$. Thus $(\sigma_{\mathfrak{H}^1} -1)(f_1) \in \mathfrak{H}^1 \otimes_{k[[u]]} u^{p/p-1}\mathcal{O}_{C^\flat}$ and so $f_1 \in \mathfrak{H}^1 + \varphi(\mathfrak{H}^1)$ by induction also. We conclude that $f \in \mathfrak{H} + \varphi(\mathfrak{H}^1) \subset \mathfrak{H} + \varphi(\mathfrak{H})$.
\end{proof}
\section{Local structure of $\mathcal{L}^{\leq p}_{\operatorname{crys}}$ in the unramified case}\label{localstructure}
\subsection{Commutative algebra}

\begin{lemma}\label{commalgebra}
	Let $A$ be a local Noetherian $\mathbb{Z}_p$-algebra  with $A[\frac{1}{p}] \neq 0$ and residue field of characteristic $p$.
	\begin{enumerate}
		\item If $\mathfrak{p} \subset A[\frac{1}{p}]$ is a maximal ideal and $\mathfrak{q}$ denotes its preimage in $A$ then $\operatorname{dim}A_\mathfrak{q} \leq  \operatorname{dim} A - 1$.
		\item If the residue field of $A$ is finite then $A/\mathfrak{q}$ is finite over $\mathbb{Z}_p$ and the residue field of $A_{\mathfrak{q}}$ is finite over $\mathbb{Q}_p$.
	\end{enumerate}
\end{lemma}
\begin{proof}
	The inclusion $A/\mathfrak{q} \rightarrow A[\frac{1}{p}]/\mathfrak{p}$ becomes an isomorphism after inverting $p$, and so $\operatorname{dim} A/\mathfrak{q} \leq 1$ by \cite[10.5.1]{EGAIV}. Since $A/\mathfrak{q}$ is a domain and not a field (its residue field has characteristic $p$) it must be that $\operatorname{dim} A/\mathfrak{q} = 1$. Thus $\operatorname{dim} A_{\mathfrak{q}} \leq \operatorname{dim} A - 1$. For (2), by the above $A/(\mathfrak{q},p)$ is zero-dimensional. Thus $A/(\mathfrak{q},p)$ is an Artin local ring with finite residue field; so it is finite over $\mathbb{F}_p$. As such $A/\mathfrak{q}$ is finite over $\mathbb{Z}_p$ (cf. \cite[Tag 031D]{stacks-project}) and so $A[\frac{1}{p}]/\mathfrak{p} = A_{\mathfrak{q}}/\mathfrak{q}A_{\mathfrak{q}}$ is $\mathbb{Q}_p$-finite.
\end{proof}

\begin{lemma}\label{Zpflat}
	Let $A$ be a  Noetherian local $\mathbb{Z}_p$-algebra with finite residue field. Suppose that $A$ is reduced, $\mathbb{Z}_p$-flat, and Nagata (cf. \cite[Tag 032E]{stacks-project}). If $\mathfrak{m}_A$ denotes the maximal ideal of $A$ and $j \geq 1$ then there exists a finite flat $\mathbb{Z}_p$-algebra $C$ such that $A \rightarrow A/\mathfrak{m}_A^j$ factors through a map $A \rightarrow C$.
\end{lemma}
\begin{proof}
	If $A$ is $\mathbb{Z}_p$-flat then so is its $\mathfrak{m}_A$-adic completion. If $A$ is Nagata and reduced then its $\mathfrak{m}_A$-adic completion is reduced, cf. \cite[Tag 07NZ]{stacks-project}. Thus we may assume $A$ is $\mathfrak{m}_A$-adically complete. 
	
	For every maximal ideal $\mathfrak{p} \subset A[\frac{1}{p}]$, Lemma~\ref{commalgebra} shows that $A/(\mathfrak{p} \cap A)$ is finite flat over $\mathbb{Z}_p$. As $A[\frac{1}{p}]$ is Jacobson (cf. \cite[Tag 02IM]{stacks-project}), the intersection of its maximal ideals equals its nilradical, and this is zero because $A$ is reduced and $\mathbb{Z}_p$-flat. Thus $\bigcap (\mathfrak{p} \cap A) =0$, the intersection running over all maximal ideals in $A[\frac{1}{p}]$. The same is true if the intersection runs over a suitably chosen countable subset $\lbrace \mathfrak{p}_1,\mathfrak{p}_2,\ldots \rbrace$ of all maximal ideals in $A[\frac{1}{p}]$.\footnote{The Artin--Rees lemma implies $\mathfrak{p} = \bigcap_{n \geq 1} (\mathfrak{p} + \mathfrak{m}_A^n)$ and so the intersection of the ideals in $\lbrace \mathfrak{p}+ \mathfrak{m}_A^n \rbrace_{\mathfrak{p},n}$, which consists of countably many distinct ideals, equals the intersection of the $\mathfrak{p}$. Thus there exists a countable subset $\lbrace \mathfrak{p}_i \rbrace_{i \geq 1}$ of the $\mathfrak{p}$'s, and integers $n_i \geq 1$, such that $\bigcap \mathfrak{p} = \bigcap_i( \mathfrak{p}_i+ \mathfrak{m}_A^{n_i})$. Since $\mathfrak{p}_i \subset \mathfrak{p}_i + \mathfrak{m}_A^{n_i}$ we have $\bigcap \mathfrak{p} = \bigcap \mathfrak{p}_i$.} The $\mathfrak{q}_i = \bigcap_{i=1}^i (\mathfrak{p}_i \cap A)$ then form a decreasing sequence of closed ideals in $A$ whose intersection is zero. It follows from \cite[III, \S2, Proposition 8]{BourbakiAlgComm3-4} that there exists an $n$ such that $\mathfrak{q}_n \subset \mathfrak{m}_A^j$. Setting $C = A/\mathfrak{q}_n$ proves the lemma. 
\end{proof}
\subsection{Hodge types and connected components}\label{no1/p}

\begin{sub}
	Let $B$ be an arbitrary $\mathbb{Z}_p$-algebra and $\mathfrak{M}_B$ a finite projective $\mathfrak{S}_B$-module equipped with a map $\varphi^*\mathfrak{M}_B \rightarrow \mathfrak{M}_B$ with cokernel killed by $E(u)^h$. For any $B$-algebra $B'$ set $\mathfrak{M}_{B'} = \mathfrak{M}_B \otimes_{B} B'$. For $i \geq 0$ define 
	$$
	K^i(\mathfrak{M}_B) = \operatorname{coker} \left( \mathfrak{M}_B^\varphi \rightarrow \mathfrak{M}_B/E(u)^i\mathfrak{M}_B \right)
	$$
	This $\mathfrak{S} \otimes_{\mathbb{Z}_p} B$-module is finite over $B$. On $\mathfrak{M}^\varphi_B$ we define a filtration 
	$$
	F^i(\mathfrak{M}_B) = \mathfrak{M}^\varphi_B \cap E(u)^i \mathfrak{M}_B
	$$
	with graded piece $\operatorname{gr}^i(\mathfrak{M}_B)$. Note that multiplication by $E(u)$ induces an injection $\operatorname{gr}^{i-1}(\mathfrak{M}_B) \rightarrow \operatorname{gr}^i(\mathfrak{M}_B)$. We let $\mathcal{G}^i(\mathfrak{M}_B)$ denote its cokernel. Thus $\mathcal{G}^i(\mathfrak{M}_B)$ is the $i$-th graded piece of the filtered $\mathcal{O}_K \otimes_{\mathbb{Z}_p} B$-module $\mathfrak{M}^\varphi_B / E(u)\mathfrak{M}^\varphi_B$ whose $i$-th filtered piece is the image of $F^i(\mathfrak{M}_B)$. It follows from (2) of Proposition~\ref{kis} that, for $B$ a finite $\mathbb{Q}_p$-algebra, 
	\begin{equation}\label{filteredz}
	\mathcal{G}^i(\mathfrak{M}_B) = \operatorname{gr}^i(D_{\operatorname{dR}}(V_B))
	\end{equation}
	whenever $\mathfrak{M}_B$ is the Breuil--Kisin module associated to a crystalline representation $V_B$.
\end{sub}

	\begin{lemma}\label{algebra}
		\begin{enumerate}
			\item $K^i(\mathfrak{M}_B)$ is $B$-flat if and only if $F^i(\mathfrak{M}_B) \otimes_{B} B' \rightarrow F^i(\mathfrak{M}_{B'})$ is surjective for every quotient $B' = B/I$.
			\item Each $\mathcal{G}^i(\mathfrak{M}_B)$ is $\mathcal{O}_K \otimes_{\mathbb{Z}_p} B$-finite and for every $B$-algebra $B'$ there are natural $\mathcal{O}_K \otimes_{\mathbb{Z}_p} B$-module homomorphisms $\mathcal{G}^i(\mathfrak{M}_B) \otimes_B B' \rightarrow \mathcal{G}^i(\mathfrak{M}_{B'})$.
			\item If (1) holds for all $i\geq 0$ then the maps in (2) are isomorphisms, and each $\mathcal{G}^i(\mathfrak{M}_B)$ is $B$-flat.
		\end{enumerate}
	\end{lemma}

This lemma does not require $K$ to be unramified over $\mathbb{Q}_p$.
\begin{proof}
	The kernel of $\mathfrak{M}_B/E(u)^i\mathfrak{M}_B \rightarrow K^i(\mathfrak{M}_B)$ is equal to $\mathfrak{M}^\varphi_B/F^i(\mathfrak{M}_B)$.  Thus $\mathfrak{M}^\varphi_B/F^i(\mathfrak{M}_B)$ is $B$-finite. As $K^i(\mathfrak{M}_B)$ is of formation compatible with base-change, being the cokernel of a map between modules compatible with base-change, there are surjective maps
	$$
	\left( \mathfrak{M}^\varphi_B/F^i(\mathfrak{M}_B) \right) \otimes_{B} B' \rightarrow \mathfrak{M}_B^\varphi/ F^i(\mathfrak{M}_{B'})
	$$
	whose kernel is $\operatorname{Tor}^B_1(K^i(\mathfrak{M}_B),B')$. Since this kernel can be identified with the cokernel of $F^i(\mathfrak{M}_B) \otimes_{B} B' \rightarrow F^i(\mathfrak{M}_{B'})$ we deduce (1). Since $\operatorname{gr}^i(\mathfrak{M}_B)$ is the kernel of the obvious surjection $\mathfrak{M}^\varphi_B/F^{i+1}(\mathfrak{M}_B) \rightarrow \mathfrak{M}^\varphi_B/F^i(\mathfrak{M}_B)$ we see each $\operatorname{gr}^i(\mathfrak{M}_B)$ is $B$-finite. We also obtain maps
	$$
	\operatorname{gr}^i(\mathfrak{M}_B) \otimes_{B} B' \rightarrow \operatorname{gr}^i(\mathfrak{M}_{B'})
	$$
	If both $\operatorname{Tor}^B_1(K^i(\mathfrak{M}_B),B')$ and $\operatorname{Tor}^B_1(K^{i+1}(\mathfrak{M}_B),B')$ are zero then we also have $\operatorname{Tor}^B_1(\mathfrak{M}^{\varphi}_B/F^{i}(\mathfrak{M}_B),B')= 0$, and so these maps are isomorphisms. As $\mathcal{G}^i(\mathfrak{M}_B)$ is the cokernel of $\operatorname{gr}^{i-1}(\mathfrak{M}_B) \rightarrow \operatorname{gr}^i(\mathfrak{M}_B)$ we deduce (2), and the first part of (3). For the last part of (3); consider the following diagram for $B'$ and $B$-algebra.
	$$
	\begin{tikzcd}
	& \operatorname{gr}^{i-1}(\mathfrak{M}_B) \otimes_B B' \arrow[d] \arrow[r,"\alpha"] & \operatorname{gr}^{i}(\mathfrak{M}_B) \otimes_B B' \arrow[d] \arrow[r] & \mathcal{G}^i(\mathfrak{M}_B) \arrow[r] \arrow[d] & 0 \\
	0 \arrow[r] & \operatorname{gr}^{i-1}(\mathfrak{M}_{B'}) \arrow[r] & \operatorname{gr}^i(\mathfrak{M}_{B'}) \arrow[r] & \mathcal{G}^i(\mathfrak{M}_{B'}) \arrow[r] & 0
	\end{tikzcd}
	$$
	The rows are exact and one easily checks the squares commute. If the $K^i(\mathfrak{M}_B)$ are $B$-flat for all $i \geq 0$ then the vertical arrows in the diagram are isomorphism. Hence $\alpha$ is injective and so, since the kernel of $\alpha$ identifies with $\operatorname{Tor}_1^B(\mathcal{G}^i(\mathfrak{M}_B),B')$, we deduce each $\mathcal{G}^i(\mathfrak{M}_B)$ is $B$-flat.
\end{proof}
\begin{sub}
	Let $A$ be a complete Noetherian ring with finite residue field $\mathbb{F}$ of characteristic $p$. Let $V_A$ be a finite free $A$-module equipped with a continuous $A$-linear action of $G_K$. To ease notation we write $\mathcal{L}$ for the $A$-scheme $\mathcal{L}^{\leq p}_{A,\operatorname{crys}}$ from Corollary~\ref{formal}. 
\end{sub}
\begin{lemma}
	For $i \geq 0$ there is a coherent sheaf $K^i$ on $\mathcal{L}$ with the following property: For any morphism $\operatorname{Spec}B \rightarrow \mathcal{L}$ of $A$-schemes, with $B$ either finite free over $\mathbb{Z}_p$ or such that $\mathfrak{m}_A^nB=0$ for some $n \geq 1$, let $\mathfrak{M}_B \in \mathcal{L}^{\leq p}_{\operatorname{crys}}(V_B)$ be the associated Breuil--Kisin module. Then the global sections of the pullback of $K^i$ to $\operatorname{Spec}B$ are computed by $K^i(\mathfrak{M}_B)$. 
\end{lemma}
\begin{proof}
	Since the formation of $K^i(-)$ is compatible with base-change (being the cokernel of a map between modules compatible with base-change) we obtain a compatible system of coherent sheaves on $\mathcal{L} \otimes_{A} A/\mathfrak{m}_A^n$ for $n \geq 1$. Grothendieck's existence theorem \cite[5.1.6]{EGAIII} (see also \cite[Tag 088C]{stacks-project}) produces the sheaf $K^i$ on $\mathcal{L}$ with the desired properties.
\end{proof}

The following is the key lemma and, unlike the previous results of this section, crucially uses the assumption that $K/\mathbb{Q}_p$ is unramified and that we restrict to weights $\leq p$.
\begin{lemma}\label{Lcirc}
	Let $\mathcal{L}^\circ \rightarrow \mathcal{L}$ denote the closed immersion defined by the ideal consisting of sections which are either nilpotent or $p$-power torsion. If $p =2$ assume that $K_\infty \cap K(\mu_{p^\infty}) = K$.\footnote{In \cite[Lemma 2.1]{Wang17} it is shown that the compatible system $\pi^{1/p^\infty}$ from \ref{notation} can always be chosen so that $K_\infty \cap K(\mu_{p^\infty}) = K$.} Then the pull-back of $K^i$ to $\mathcal{L}^\circ$ is flat.
\end{lemma}
\begin{proof}
	It suffices to prove flatness at the closed points of $\mathcal{L}^\circ$. Thus, for a closed point $x \in \mathcal{L}$, it suffices to show $K^i(\mathfrak{M}_B)$ is $B$-flat whenever $B = \mathcal{O}_{\mathcal{L},x}/\mathfrak{m}_{\mathcal{L},x}^n$ for some $n \geq 1$ and any $\mathfrak{M}_B \in \mathcal{L}^{\leq p}_{\operatorname{crys}}(V_B)$. By definition $\mathcal{O}_{\mathcal{L},x}$ is $\mathbb{Z}_p$-flat and reduced. It is also Nagata (since it is a localisation of a finite type algebra over a complete local ring, cf. \cite[Tag 032E]{stacks-project}). Therefore Lemma~\ref{Zpflat}, applied with $A = \mathcal{O}_{\mathcal{L},x}$, reduces the problem to that of showing $K^i(\mathfrak{M}_C)$ is $C$-flat whenever $C$ is a finite flat $\mathbb{Z}_p$-algebra. This is the case by Lemma~\ref{flat} below (this is where the assumption that $K_\infty \cap K(\mu_{p^\infty}) = K$ when $p=2$ is used).
\end{proof}

\begin{lemma}\label{flat}
	Let $C$ be a local finite flat $\mathbb{Z}_p$-algebra, and suppose $\mathfrak{M}_C \in \mathcal{L}^{\leq p}_{\operatorname{crys}}(V_C)$ for some continuous representation $V_C$ of $G_K$ on a finite free $C$-module. If $p =2$ assume that $K_\infty \cap K(\mu_{p^\infty}) = K$. Then $K^i(\mathfrak{M}_C)$ is $C$-flat.
\end{lemma}
\begin{proof}
	It suffices to show $K^i(\mathfrak{M}_C)$ is $\mathbb{Z}_p$-flat and that $K^i(\mathfrak{M}_{C} \otimes_C C/pC) = K^i(\mathfrak{M}_C) \otimes_C C/pC$ is $C/pC$-flat, cf. \cite[Tag 00ML]{stacks-project}. If $p>2$ then, since $V_C[\frac{1}{p}]$ is crystalline, \cite[Theorem 4.20]{GLS} ensures the existence of an $\mathfrak{S}$-basis $(e_j)$ of $\mathfrak{M}_C$ such that $\mathfrak{M}_C^\varphi$ is generated over $\mathfrak{S}$ by $E(u)^{r_i}e_i$ for certain integers $r_i$. If $p =2$ and $K_\infty \cap K(\mu_{p^\infty}) = K$ then the same is true, as explained in \cite[Section 4]{Wang17}. This implies $K^i(\mathfrak{M}_C)$ is $p$-torsion free.
	
	To show flatness modulo~$p$ set $B =C/pC$ and let $B \rightarrow B'$ be a surjective homomorphism. After Lemma~\ref{algebra} it suffices to show the natural map $F^i(\mathfrak{M}_B) \rightarrow F^i(\mathfrak{M}_{B'})$ is surjective. Proposition~\ref{inclusion} implies $\mathfrak{M}_B$ is a strongly divisible. It follows from \cite[5.4.6]{B18} and \cite[5.4.2]{B18} that if $\mathfrak{M}$ is any strongly divisible Breuil--Kisin module and $\mathfrak{M} \rightarrow \mathfrak{N}$ is a $\varphi$-equivariant surjection into a Breuil--Kisin module which is free as a $k[[u]]$-module then $F^i(\mathfrak{M}) \rightarrow F^i(\mathfrak{N})$ is surjective. Applying this with $\mathfrak{M} = \mathfrak{M}_B$ and $\mathfrak{N} = \mathfrak{M}_{B'}$ proves the lemma. 
\end{proof}
\begin{remark}
	Note the second paragraph in the proof of Lemma~\ref{flat} implies $K^i$ is flat on $\mathcal{L} \otimes_{A} A/\mathfrak{m}_A$. Thus it seems possible that $K^i$ is flat on the whole of $\mathcal{L}$, though we do not know how to prove this.
\end{remark}
\begin{corollary}
	In the situation of Lemma~\ref{Lcirc} there is, for each $i \geq 0$, a coherent sheaf $\mathcal{G}^i$ on $\mathcal{L}^\circ$ with the following properties. For any morphism $\operatorname{Spec}B \rightarrow \mathcal{L}^\circ$ of $A$-schemes, with $B$ either finite free over $\mathbb{Z}_p$ or such that $\mathfrak{m}_A^nB = 0$ for some $n \geq 1$, let $\mathfrak{M}_B \in \mathcal{L}^{\leq p}_{\operatorname{crys}}(V_B)$ be the associated Breuil--Kisin module. Then the global sections of the pullback of $\mathcal{G}^i$ to $\operatorname{Spec}B$ are computed by $\mathcal{G}^i(\mathfrak{M}_B)$. Furthermore, $\mathcal{G}^i$ is flat on $\mathcal{L}^\circ$. 
\end{corollary}
\begin{proof}
	Since $K^i$ is flat on $\mathcal{L}^\circ$, Lemma~\ref{algebra} implies that on each $\mathcal{L}^\circ \otimes_A A/\mathfrak{m}_A^n$ the formation of $\mathcal{G}^i(-)$ is compatible with base-change. Thus we obtain a compatible system of coherent sheaves $\mathcal{G}^i$ on the $\mathcal{L}^\circ \otimes_{A} A/\mathfrak{m}_A^n$. By (3) of Lemma~\ref{algebra} we also know these sheaves are flat on $\mathcal{L}^\circ \otimes_A A/\mathfrak{m}_A^n$. Grothendieck's existence theorem \cite[5.1.6]{EGAIII} produces a sheaf $\mathcal{G}^i$ as desired; that it is flat follows because each $\mathcal{G}^i \otimes_{A} A/\mathfrak{m}_A^n$ is flat.
\end{proof}
\begin{sub}
	Let $E$ be a finite extension of $\mathbb{Q}_p$ such that $A$ is an $\mathcal{O}_E$-algebra. Let us fix a $p$-adic Hodge type $\mathbf{v}$, i.e.  a finite free $K \otimes_{\mathbb{Q}_p} E$-module $D_{\mathbf{v}}$ equipped with a grading $\operatorname{gr}^i(D_{\mathbf{v}})$ by $K \otimes_{\mathbb{Q}_p} E$-submodules.\footnote{Below, when we speak of a $p$-adic Hodge type $\mathbf{v}$, the field $E$ will be implicit in the data of $\mathbf{v}$.} Assume this grading is concentrated in degree $[0,p]$. If $B$ is a finite local $E$-algebra then we say a crystalline representation $V_B$ has $p$-adic Hodge type $\mathbf{v}$ if there are isomorphisms
	$$
	\operatorname{gr}^i(D_{\operatorname{crys}}(V_B)) \xrightarrow{\sim} \operatorname{gr}^i(D_{\mathbf{v}}) \otimes_{E} B
	$$ 
	for all $i \in \mathbb{Z}$. Since we are assuming that $K$ is unramified over $\mathbb{Q}_p$, $p$-adic Hodge types can be described integrally: there exists a finite free $\mathcal{O}_K \otimes_{\mathbb{Z}_p} \mathcal{O}_E$-module $D^\circ_{\mathbf{v}}$ with a grading $\operatorname{gr}^i(D^\circ_{\mathbf{v}})$ by $\mathcal{O}_K \otimes_{\mathbb{Z}_p} \mathcal{O}_E$-submodules, so that $D_\mathbf{v} \cong D_{\mathbf{v}}^\circ \otimes_{\mathcal{O}_{K}} K$ as graded modules. This is because $\mathcal{O}_K$ is unramified over $\mathbb{Z}_p$ and so $\mathcal{O}_K \otimes_{\mathbb{Z}_p} \mathcal{O}_E$ is a product of unramified extensions of $\mathcal{O}_E$.
\end{sub}
\begin{sub}\label{vcomponent}
	Let $B$ be a $\mathbb{Z}_p$-algebra. Since $\mathcal{O}_K$ is unramified over $\mathbb{Z}_p$, a finite $\mathcal{O}_K \otimes_{\mathbb{Z}_p} B$-module which is flat over $B$ is flat over $\mathcal{O}_K \otimes_{\mathbb{Z}_p} B$ (cf. \cite[Tag 00MH]{stacks-project}). Provided $K_\infty \cap K(\mu_{p^\infty}) = K$ if $p = 2$, this implies that the flat coherent sheaf $\mathcal{G}^i$ on $\mathcal{L}^\circ$ is a flat sheaf of $\mathcal{O}_K \otimes_{\mathbb{Z}_p} \mathcal{O}_{\mathcal{L}^\circ}$-modules. As such, if for each $p$-adic Hodge type $\mathbf{v}$ we define $\mathcal{L}^{\mathbf{v}}$ to be the set of $x \in \mathcal{L}^\circ$ with 
	$$
	\mathcal{G}^i_x \cong \operatorname{gr}^i(D^\circ_{\mathbf{v}}) \otimes_{\mathcal{O}_E} \mathcal{O}_{\mathcal{L}^\circ,x}
	$$
	as $\mathcal{O}_K \otimes_{\mathbb{Z}_p} \mathcal{O}_{\mathcal{L},x}$-modules for each $i \geq 0$, then $\mathcal{L}^{\mathbf{v}}$ is a union of connected components of $\mathcal{L}^\circ$.
\end{sub}
\begin{proposition}\label{crysdef}
	If $p=2$ assume that $K_\infty \cap K(\mu_{p^\infty}) = K$. Let $A^{\mathbf{v}}_{\operatorname{crys}}$ denote the quotient of $A$ corresponding to the scheme-theoretic image of $\mathcal{L}^{\mathbf{v}} \rightarrow \operatorname{Spec} A$. Then
	\begin{enumerate}
		\item The morphism $\mathcal{L}^{\mathbf{v}}  \rightarrow \operatorname{Spec}A^{\mathbf{v}}_{\operatorname{crys}}$ becomes an isomorphism after inverting $p$.
		\item For any finite reduced $\mathbb{Q}_p$-algebra $B$, a map $A \rightarrow B$ factors through $A^{\mathbf{v}}_{\operatorname{crys}}$ if and only if $V_B = V_{A} \otimes_{A} B$ is crystalline with $p$-adic Hodge type $\mathbf{v}$.
	\end{enumerate}
\end{proposition}

By construction $A^{\mathbf{v}}_{\operatorname{crys}}$ is reduced and $\mathbb{Z}_p$-flat and so (2) uniquely determines this quotient.
\begin{proof}
	For (1) use that, after Lemma~\ref{closedimmersion}, $\mathcal{L}^\circ \rightarrow \operatorname{Spec}A$ becomes a closed immersion after inverting $p$. For (2) argue as in Proposition~\ref{quotient}, the point being that if $C$ is a reduced finite flat $\mathbb{Z}_p$-algebra then any $\mathfrak{M}_C \in \mathcal{L}^{\leq p}_{\operatorname{crys}}(V_C)$ induces a $C$-valued point of $\mathcal{L}^\circ$, and this point factors through $\mathcal{L}^{\mathbf{v}}$ if and only if $\mathcal{G}^i(\mathfrak{M}_C) \cong \operatorname{gr}^i(D^\circ_{\mathbf{v}})$ for each $i \geq 0$. 
\end{proof}
\subsection{Cyclotomic-freeness}

\begin{sub}\label{setup}
	For this subsection let $\mathbb{F}$ be a finite field of characteristic $p$ and consider $Z_{\mathbb{F}}$ and $W_{\mathbb{F}}$, both finite dimensional $\mathbb{F}$-vector spaces equipped with a continuous $\mathbb{F}$-linear action of $G_K$. Further let $\mathfrak{Z} \in \mathcal{L}^{\leq p}_{\operatorname{crys}}(Z_{\mathbb{F}})$ and $\mathfrak{W} \in \mathcal{L}^{\leq p}_{\operatorname{crys}}(W_{\mathbb{F}})$. It will be useful to consider the following hypothesis:
\end{sub}
\begin{hypothesis}\label{hyp}
	\begin{enumerate}
		\item 	Every continuous cocycle $G_K \rightarrow \operatorname{Hom}(\mathfrak{Z},\mathfrak{W}) \otimes_{k[[u]]} u^{p/p-1}\mathcal{O}_{C^\flat}$ given by $\sigma \mapsto F_\sigma$ from with (i) $(\varphi -1)(F_\sigma) = 0$ for all $\sigma \in G_K$, and (ii)
 $F_\sigma = 0$ for all $\sigma \in G_{K_\infty}$, is zero.
		\item If $V_{\mathbb{F}}$ is a continuous representation of $G_K$ on a finite dimensional $\mathbb{F}$-vector space, and $\mathfrak{M} \in \mathcal{L}^{\leq p}_{\operatorname{crys}}(V_{\mathbb{F}})$ then (1) is satisfied when $\mathfrak{Z} = \mathfrak{W} = \mathfrak{M}$.
	\end{enumerate}

\end{hypothesis}

\begin{lemma}\label{end}
	Suppose (1) of Hypothesis~\ref{hyp} is satisfied. Then 
	\begin{enumerate}
		\item Under the identification $\operatorname{Hom}(\mathfrak{Z},\mathfrak{W}) \otimes_{k[[u]]} C^\flat = \operatorname{Hom}(Z_{\mathbb{F}},W_{\mathbb{F}}) \otimes_{\mathbb{F}_p} C^\flat$ there are inclusions $\operatorname{Hom}(\mathfrak{Z},\mathfrak{W})^{\varphi =1} \subset \operatorname{Hom}(Z_{\mathbb{F}},W_{\mathbb{F}})^{G_K}$.
		\item Let $0 \rightarrow W_{\mathbb{F}} \rightarrow Y \rightarrow Z_{\mathbb{F}} \rightarrow 0$ be an exact sequence of $G_{K_\infty}$-representations and let $\mathfrak{Y} \in \mathcal{L}^{\leq p}_{\operatorname{SD}}(Y)$ be such that Lemma~\ref{exactO} induces an exact sequence $0 \rightarrow \mathfrak{W} \rightarrow \mathfrak{Y} \rightarrow \mathfrak{Z} \rightarrow 0$. Then there exists at most one way of extending the $G_{K_\infty}$-action on $Y$ to a $G_K$-action so that $\mathfrak{Y} \in \mathcal{L}^{\leq p}_{\operatorname{crys}}(Y)$ and $0 \rightarrow W_{\mathbb{F}} \rightarrow Y \rightarrow Z_{\mathbb{F}} \rightarrow 0$ is $G_K$-equivariant.
	\end{enumerate}
\end{lemma}
\begin{proof}
	For (1) note that in general $\operatorname{Hom}(\mathfrak{Z},\mathfrak{W})^{\varphi =1} \subset \operatorname{Hom}(Z_{\mathbb{F}},W_{\mathbb{F}})^{G_{K_\infty}}$. As a consequence, if $f \in \operatorname{Hom}(\mathfrak{Z},\mathfrak{W})^{\varphi =1}$ then the $1$-cocycle $\sigma \mapsto (\sigma-1)(f)$ satisfies the conditions of Hypothesis~\ref{hyp}. Thus $(\sigma -1)(f) = 0$ and so $f$ is $G_K$-equivariant. 
	
	For (2) recall from \ref{gal} that, after choosing an $\mathfrak{S}_{\mathbb{F}}$ splitting of $0 \rightarrow \mathfrak{W} \rightarrow \mathfrak{Y} \rightarrow \mathfrak{Z} \rightarrow 0$ so that $\varphi_{\mathfrak{Y}} = (\varphi_{\mathfrak{W}} + f \circ \varphi_{\mathfrak{Z}},\varphi_\mathfrak{Z})$, the possible ways of extending the $G_{K_\infty}$-action on $Y$ to a $G_K$-action as required by the lemma are in bijection with the set of $1$-cocycles $\sigma \mapsto f_\sigma$ taking values in $\operatorname{Hom}(\mathfrak{Z},\mathfrak{W}) \otimes_{k[[u]]} u^{p/p-1}\mathcal{O}_{C^\flat}$ and satisfying $(\varphi -1)(f_\sigma) = (\sigma-1)(f)$ and $f_\sigma = 0$ for $\sigma \in G_{K_\infty}$. As the difference of two such cocycles is a cocycle as in Hypothesis~\ref{hyp} we obtain (2).
\end{proof}

\begin{lemma}\label{inductionstep}
	Let $0 \rightarrow Z_{\mathbb{F},1} \rightarrow Z_{\mathbb{F}} \rightarrow Z_{\mathbb{F},2} \rightarrow 0$ be a $G_K$-equivariant exact sequence, and suppose $0 \rightarrow \mathfrak{Z}_1 \rightarrow \mathfrak{Z} \rightarrow \mathfrak{Z}_2 \rightarrow 0$ is the corresponding $\varphi$-equivariant exact sequence from Lemma~\ref{exactO}. If (1) of Hypothesis~\ref{hyp} is satisfied when $\mathfrak{Z}$ is replaced by $\mathfrak{Z}_1$ and $\mathfrak{Z}_2$ then (1) of Hypothesis~\ref{hyp} is satisfied itself.
\end{lemma}
\begin{proof}
	Applying $\operatorname{Hom}(-,\mathfrak{W})$ to $0 \rightarrow \mathfrak{Z}_1 \rightarrow \mathfrak{Z} \rightarrow \mathfrak{Z}_2 \rightarrow 0$ yields a $\varphi$-equivariant exact sequence 
	\begin{equation}\label{homexact}
	0 \rightarrow \operatorname{Hom}(\mathfrak{Z}_1,\mathfrak{W}) \rightarrow \operatorname{Hom}(\mathfrak{Z},\mathfrak{W}) \rightarrow \operatorname{Hom}(\mathfrak{Z}_2,\mathfrak{W}) \rightarrow 0
	\end{equation}
	which is $G_K$-equivariant after applying $\otimes_{k[[u]]} C^\flat$. Thus, if $\sigma \mapsto F_\sigma$ is a $1$-cocycle as in Hypothesis~\ref{hyp} then so is its image in $\operatorname{Hom}(\mathfrak{Z}_2,\mathfrak{W}) \otimes_{k[[u]]} u^{p/p-1}\mathcal{O}_{C^\flat}$. We conclude that if (1) of Hypothesis~\ref{hyp} is satisfied with $\mathfrak{Z}$ replaced with $\mathfrak{Z}_2$ then this image must be zero. The sequence \ref{homexact} remains exact after applying $\otimes_{k[[u]]} u^{p/p-1}\mathcal{O}_{C^\flat}$ since each of its terms is $k[[u]]$-free. Therefore $F_{\sigma} \in \operatorname{Hom}(\mathfrak{Z}_{1},\mathfrak{W}) \otimes_{k[[u]]} u^{p/p-1}\mathcal{O}_{C^\flat}$ for each $\sigma$. If (1) of Hypothesis~\ref{hyp} is satisfied with $\mathfrak{Z}$ replaced with $\mathfrak{Z}_1$ then we must have $F_\sigma = 0$. Hence (1) of Hypothesis~\ref{hyp} itself is satisfied.
\end{proof}

\begin{proposition}\label{what}
	(1) of Hypothesis~\ref{hyp} is satisfied when either of the following two conditions hold:
	\begin{enumerate}
		\item $\mathfrak{Z} \in \mathcal{L}^{\leq p-1}_{\operatorname{crys}}(Z_{\mathbb{F}})$.
		\item Every Jordan--Holder factor of $Z_{\mathbb{F}}$ is absolutely irreducible and if $Z$ is a Jordan--Holder factor of $Z_{\mathbb{F}}$ so that $Z \otimes_{\mathbb{F}} \overline{\mathbb{F}}(1)$ is unramified, then $Z \otimes_{\mathbb{F}} \overline{\mathbb{F}}(1)$ is not a Jordan--Holder factor of $W_{\mathbb{F}} \otimes_{\mathbb{F}} \overline{\mathbb{F}}$.\footnote{Here $\overline{\mathbb{F}}(1)$ denotes the one-dimensional representation of $G_K$ over $\overline{\mathbb{F}}$ on which $G_K$-acts by the cyclotomic character. Likewise for $\overline{\mathbb{F}}(-1)$, but for the inverse of the cyclotomic character.}
	\end{enumerate}
\end{proposition}
\begin{proof}
	By inducting on the length of $Z_{\mathbb{F}}$, and using Lemma~\ref{inductionstep}, we can reduce to the case that $Z_{\mathbb{F}}$ is irreducible. Let $\sigma \mapsto F_\sigma$ be as in Hypothesis~\ref{hyp} and suppose $\sigma \in G_K$ is such that $F_\sigma \neq 0$. Let $\mathfrak{J}$ be the kernel of the restriction of $F_\sigma$ to $\mathfrak{Z}$. Since $F_\sigma$ is $\varphi$-equivariant, $\mathfrak{J}$ is a $\varphi$-stable $\mathfrak{S}_{\mathbb{F}}$-submodule of $\mathfrak{Z}$. Since the image of $F_\sigma$ is $u$-torsionfree, $\mathfrak{J} \otimes_{k[[u]]} C^\flat = \mathfrak{Z} \otimes_{k[[u]]} C^\flat$ only if $\mathfrak{J} =\mathfrak{Z}$, and this does not happen since $F_\sigma \neq 0$. Since $Z_{\mathbb{F}}$ is irreducible as a $G_K$-representation it is irreducible as a $G_{K_\infty}$-representation, cf. \ref{unique}. Therefore $\mathfrak{J}= 0$ and $F_\sigma$ is injective, otherwise the $G_{K_\infty}$-representation $(\mathfrak{J} \otimes_{k[[u]]} C^\flat)^{\varphi =1} \subset Z_{\mathbb{F}}$ contradicts the $G_{K_\infty}$-irreducibility of $Z_{\mathbb{F}}$.
	
	For each $z \in \mathfrak{Z} \setminus u \mathfrak{Z}$ and each $n \geq 1$ there exists $\delta_n \in \mathbb{Z}$ and $z_n \in \mathfrak{Z} \setminus u\mathfrak{Z}$ such that 
	$$
	\varphi^n(z) = u^{p^n+\ldots+p -\delta_n} z_n
	$$   
	Using that $u^p \mathfrak{Z} \subset \mathfrak{Z}^{\varphi}$ we deduce that $\delta_n \geq 0$ (the point being that $\varphi(z) \not\in u^{p+1}\mathfrak{Z}$, if it was then $\varphi(z)/u \in \mathfrak{Z}^\varphi$ which implies $z \in u^{1/p} \mathfrak{Z} \otimes_{k[[u]]} \mathcal{O}_{C^\flat}$, a contradiction). In particular there is a $\delta' \geq 0$ such that $\varphi(z_n) = u^{p-\delta'}z_{n+1}$ and so, since
	$$
	\varphi^{n+1}(z) = u^{p^{n+1}+ \ldots + p^2 - p\delta_n} \varphi(z_n) = u^{p^{n+1}+\ldots+ p^2 + p -p\delta_n - \delta'} z_{n+1}
	$$
	we see that $\delta_{n+1} = p\delta_n + \delta' \geq p\delta_n$. In particular, if $\delta_N > 0$ for some $N$ then $\delta_n \rightarrow \infty$ as $n \rightarrow \infty$. As $F_\sigma$ is injective and $\mathfrak{Z}$ is finitely generated there exists $\gamma >0$ such that $F_\sigma(z) \not\in \mathfrak{W} \otimes_{k[[u]]} u^{\gamma+p/(p-1)} \mathcal{O}_{C^\flat}$ for any $z \in \mathfrak{Z} \setminus u \mathfrak{Z}$. This implies
	$$
	F_\sigma(\varphi^n(z)) = u^{p^n+\ldots + p -\delta_n} F_\sigma(z_n) \not\in \mathfrak{W} \otimes_{k[[u]]} u^{p^n+\ldots+p - \delta_n + \gamma + p/(p-1)} \mathcal{O}_{C^\flat}
	$$
	for any $z \in \mathfrak{Z}\setminus u\mathfrak{Z}$ and $n \geq 0$. As $F_\sigma$ is $\varphi$-equivariant and $F_\sigma \in \operatorname{Hom}(\mathfrak{Z},\mathfrak{W}) \otimes_{k[[u]]} u^{p/(p-1)}\mathcal{O}_{C^\flat}$ we also deduce
	$$
	F_\sigma(\varphi^n(z)) = \varphi^n(F_\sigma(z)) \in \mathfrak{W} \otimes_{k[[u]]} u^{p^{n+1}/(p-1)} \mathcal{O}_{C^\flat}
	$$
	Note that $p^n+\ldots+p -\delta_n + \gamma -p/(p-1) = p^{n+1}/(p-1) -\delta_n + \gamma$. If $\delta_N > 0$ for some $N$ then, by choosing $n$ large enough that $-\delta_n + \gamma < 0$, we obtain a contradiction. We conclude that $F_\sigma = 0$ unless, for all $z \in \mathfrak{Z} \setminus u\mathfrak{Z}$, $\varphi(z) = u^pz'$ for some $z' \in \mathfrak{Z} \setminus u\mathfrak{Z}$. Equivalently $F_\sigma=0$ unless $\mathfrak{Z}^\varphi = u^p \mathfrak{Z}$. 
	
	This completes the proof when $\mathfrak{Z} \in \mathcal{L}^{\leq p-1}_{\operatorname{crys}}(Z_{\mathbb{F}})$ since then $u^{p-1}\mathfrak{Z} \subset \mathfrak{Z}^\varphi$. Therefore assume we are as in (2). If $\mathfrak{Z}^\varphi = u^p\mathfrak{Z}$ then $\varphi$ induces a semilinear automorphism of $ u^{-p/p-1}\mathfrak{Z}$, and hence a $\overline{k}$-semilinear automorphism of $\widetilde{\mathfrak{Z}} := (u^{p/p-1}\mathfrak{Z}) \otimes_{k[[u]]} \overline{k}$. Via the $\varphi$-equivariant section of $\mathcal{O}_{C^\flat} \rightarrow \overline{k}$ given by Teichmuller representatives, we $\varphi$-equivariantly view $\widetilde{\mathfrak{Z}}$ as a subset of $\mathfrak{Z} \otimes_{k[[u]]} C^\flat$. Since $\overline{k}$ is algebraically closed $\widetilde{\mathfrak{Z}}$ is generated by $\varphi$-invariant elements, and so $\widetilde{\mathfrak{Z}}^{\varphi =1} = Z_{\mathbb{F}}$. It is a straightforward exercise so show that, as a $G_{K_\infty}$-representation, $\widetilde{\mathfrak{Z}}^{\varphi =1}$ is a twist of an unramified representation by the inverse of the cyclotomic character. Thus $\widetilde{\mathfrak{Z}}^{\varphi =1}$ has the same description as a $G_K$-representation, cf. \ref{unique}. As $Z_{\mathbb{F}}$ is absolutely irreducible it follows that $Z_{\mathbb{F}}$ is one-dimensional.
	
	We have shown that if a non-zero cocycle $\sigma \mapsto F_\sigma$ exists as in Hypothesis~\ref{hyp} then $Z_{\mathbb{F}} \otimes_{\mathbb{F}} \overline{\mathbb{F}}(1)$ is an unramified character. Since $F_\sigma|_{G_{K_\infty}} = 0$ this cocycle represents a class in $H^1(G_K,\operatorname{Hom}(Z_{\mathbb{F}} \otimes_{\mathbb{F}} \overline{\mathbb{F}},W_{\mathbb{F}}\otimes_{\mathbb{F}} \overline{\mathbb{F}}))$ which is killed by restriction to $G_{K_\infty}$. If $Z_{\mathbb{F}} \otimes_{\mathbb{F}} \overline{\mathbb{F}}(1)$ is not a Jordan--Holder factor of $W_{\mathbb{F}} \otimes_{\mathbb{F}} \overline{\mathbb{F}}$ then \cite[2.3.5]{B18b} implies this restriction map is injective, so $F_\sigma = (\sigma-1)(F)$ for some $F \in \operatorname{Hom}(Z_{\mathbb{F}} \otimes_{\mathbb{F}} \overline{\mathbb{F}},W_{\mathbb{F}}\otimes_{\mathbb{F}} \overline{\mathbb{F}})$ which is fixed by $G_{K_\infty}$. Applying \cite[2.3.5]{B18b} again then implies $F$ is fixed by $G_K$, so $F_\sigma = 0$. We conclude (1) of Hypothesis~\ref{hyp} is satisfied.
\end{proof}

This motivates the following definition.
\begin{definition}\label{cyclofree}
	\begin{enumerate}
		\item We say $V_{\mathbb{F}}$ is cyclotomic-free if every Jordan--Holder factor of $V_{\mathbb{F}}$ is absolutely irreducible, and if $Z$ is a Jordan--Holder factor of $V_{\mathbb{F}}$ such that $Z \otimes_{\mathbb{F}} \overline{\mathbb{F}}(1)$ is unramified then $Z \otimes_{\overline{\mathbb{F}}} \overline{\mathbb{F}}(1)$ is not a Jordan--Holder factor of $V_{\mathbb{F}} \otimes_{\mathbb{F}} \overline{\mathbb{F}}$.
		\item We say $V_{\mathbb{F}}$ is strongly cyclotomic-free is $V_{\mathbb{F}}|_{G_L}$ is cyclotomic-free for all finite unramified extensions $L/K$. Equivalently $V_{\mathbb{F}}$ is strongly cyclotomic-free if each Jordan--Holder factor is absolutely irreducible, and if an unramified twist of $\mathbb{F}(-1)$ is a Jordan--Holder factor of $V_{\mathbb{F}}$ then no Jordan--Holder factor of $V_{\mathbb{F}}$ is unramified.
	\end{enumerate}
\end{definition}

Most of our results will only require us to assume cyclotomic-freeness. However, to prove potential diagonalisability it will be necessary to replace a representation $V_{\mathbb{F}}$ by the restriction $V_{\mathbb{F}}|_{G_L}$ for some sufficiently large unramified extension $L/K$ so that $V_{\mathbb{F}}|_{G_L}$ has every Jordan--Holder factor one-dimensional. To apply our results we will need $V_{\mathbb{F}}|_{G_L}$ to be cyclotomic-free. The following example indicates why we therefore require $V_{\mathbb{F}}$ to be satisfy a stronger property than cyclotomic-freeness.

\begin{example} Assume $p=2$ and let $\psi$ be a non-trivial character of $G_K$ which becomes trivial when restricted to $G_L$ for $L/K$ a finite unramified extension. Let $V_{\mathbb{F}}$ be an irreducible representation of $G_K$ of dimension $[L:K]$. Assuming $\mathbb{F}$ to be sufficiently large, $L/K$ is then the smallest unramified extension such that $V_{\mathbb{F}}|_{G_L}$ has every Jordan--Holder factor one-dimensional. However if $V'_{\mathbb{F}} = V_{\mathbb{F}} \oplus \mathbb{F}(\psi)$ then $V'_{\mathbb{F}}|_{G_L}$ is not cyclotomic-free since $\mathbb{F}(\psi)|_{G_L}$ is trivial and, because $p=2$, the cyclotomic character is also trivial.
\end{example}
\begin{corollary}
	If $V_{\mathbb{F}}$ is cyclotomic-free then (2) of Hypothesis~\ref{hyp} holds for all $\mathfrak{M} \in \mathcal{L}^{\leq p}_{\operatorname{crys}}(V_{\mathbb{F}})$.
\end{corollary}
\begin{proof}
	This follows from Proposition~\ref{what} applied with $\mathfrak{Z} =\mathfrak{W} = \mathfrak{M}$.
\end{proof}
\subsection{Local analysis of $\mathcal{L}^{\leq p}_{\operatorname{crys}}$} 
\begin{sub}\label{def}
With notation as in \ref{setup}, a deformation of $V_{\mathbb{F}}$ to a complete local $W(\mathbb{F})$-algebra $A$, with residue field $\mathbb{F}$, is a finite free $A$-module $V_A$ equipped with a continuous $A$-linear action of $G_K$ together with a $G_K$-equivariant isomorphism $V_A \otimes_{A} \mathbb{F} \cong V_{\mathbb{F}}$.
	
	 Fix an $\mathbb{F}$-basis $\xi_{\mathbb{F}}$ of $V_{\mathbb{F}}$. Then a framed deformation of $V_{\mathbb{F}}$ is a deformation $V_A$ together with an $A$-basis $\xi_A$ which gets identified with $\xi_{\mathbb{F}}$ after applying $\otimes_{A} \mathbb{F}$. The functor
	 $$
	 D^\square_{V_{\mathbb{F}}}(A) = \lbrace \text{isomorphism classes of framed deformations of $V_{\mathbb{F}}$} \rbrace
	 $$
	 is representable by a complete local Noetherian $W(\mathbb{F})$-algebra $R = R^{\square}_{V_{\mathbb{F}}}$. Let $V_R$ denote the universal framed deformation. Applying Corollary~\ref{formal} to $R$ and $V_R$ gives a projective $R$-scheme $\mathcal{L} := \mathcal{L}^{\leq p}_{R,\operatorname{crys}}$.
\end{sub}

\begin{lemma}\label{dimm}
	Let $\mathbb{F}'$ be a finite extension of $\mathbb{F}$ and $x$ an $\mathbb{F}'$-valued point of $\mathcal{L}$. Suppose the corresponding $\mathfrak{M}_x \in \mathcal
	L^{\leq p}_{\operatorname{crys}}(V_{\mathbb{F}'})$ satisfies (2) of Hypothesis~\ref{hyp}. Then\footnote{It would be better to write $ \sum_{n+m =i} \operatorname{dim}_{\mathbb{F}} \mathcal{G}^n(\mathfrak{M}_{x}) - \operatorname{dim}_{\mathbb{F}}\mathcal{G}^m(\mathfrak{M}_x)$ as $\mathcal{G}^i(\operatorname{Hom}(\mathfrak{M}_x,\mathfrak{M}_x))$, but we have only defined $\mathcal{G}^i(-)$ for finite projective $\mathfrak{S}_{\mathbb{F}}$-modules equipped with maps $\varphi^*\mathfrak{M} \rightarrow \mathfrak{M}$. The image of the Frobenius on $\operatorname{Hom}(\mathfrak{M}_x,\mathfrak{M}_x)$ will not, in general, be contained in $\operatorname{Hom}(\mathfrak{M}_x,\mathfrak{M}_x)$.}
	$$
	\operatorname{dim}_{\mathbb{F}'} \mathcal{O}_{\mathcal{L},x}(\mathbb{F}'[\epsilon])  \leq d^2 + \sum_{i > 0} \sum_{n-m =i} \operatorname{dim}_{\mathbb{F}'} \mathcal{G}^n(\mathfrak{M}_{x}) - \operatorname{dim}_{\mathbb{F}'}\mathcal{G}^m(\mathfrak{M}_x)
	$$
	Here $\mathbb{F}'[\epsilon]$ is the ring of dual numbers over $\mathbb{F}'$.
\end{lemma}
\begin{proof}
	Replacing $V_{\mathbb{F}}$ by $V_{\mathbb{F}} \otimes_{\mathbb{F}} \mathbb{F}'$ we may suppose $\mathbb{F} = \mathbb{F}'$. An element of $\mathcal{O}_{\mathcal{L},x}(\mathbb{F}[\epsilon])$ give rise to a map $R \rightarrow \mathbb{F}[\epsilon]$, and so a framed deformation $V_{\mathbb{F}[\epsilon]}$ of $V_{\mathbb{F}}$ to $\mathbb{F}[\epsilon]$, and an element $\mathfrak{M}_{\mathbb{F}[\epsilon]} \in \mathcal{L}^{\leq p}_{\operatorname{crys}}(V_{\mathbb{F}[\epsilon]})$ satisfying $\mathfrak{M}_{\mathbb{F}[\epsilon]} \otimes_{\mathbb{F}[\epsilon]} \mathbb{F} = \mathfrak{M}_{x}$. When viewed as an $\mathfrak{S}_{\mathbb{F}}$-module $\mathfrak{M}_{\mathbb{F}[\epsilon]}$ fits into an exact sequence
	$$
	0 \rightarrow \mathfrak{M}_x \xrightarrow{\epsilon} \mathfrak{M}_{\mathbb{F}[\epsilon]} \rightarrow \mathfrak{M}_x \rightarrow 0
	$$
	As explained in Proposition~\ref{SDexact}, an $\mathfrak{S}_{\mathbb{F}}$-splitting of this sequence can be chosen so that $\varphi_{\mathfrak{M}_{\mathbb{F}[\epsilon]}} = (\varphi_{\mathfrak{M}_x} + f \circ \varphi_{\mathfrak{M}_x},\varphi_{\mathfrak{M}_x})$ for some $f \in  \operatorname{Hom}(\mathfrak{M}_x,\mathfrak{M}_x)$. Since choosing a different splitting replaces $f$ by $f + (\varphi-1)(g)$ for some $g \in \operatorname{Hom}(\mathfrak{M}_x,\mathfrak{M}_x)$ we obtain a well-defined map
	\begin{equation}\label{map}
	\mathcal{O}_{\mathcal{L},x}(\mathbb{F}[\epsilon]) \rightarrow \operatorname{Hom}(\mathfrak{M}_x,\mathfrak{M}_x)/ F^0 \operatorname{Hom}(\mathfrak{M}_x,\mathfrak{M}_x)
	\end{equation}
	where $F^0 \operatorname{Hom}(\mathfrak{M}_x,\mathfrak{M}_x)$ consists of those $g \in \operatorname{Hom}(\mathfrak{M}_x,\mathfrak{M}_x)$ for which $\varphi(g) \in \operatorname{Hom}(\mathfrak{M}_x,\mathfrak{M}_x)$. We remark that the target of \eqref{map} can be identified with $\operatorname{Ext}^1_{\operatorname{SD}}(\mathfrak{M}_x,\mathfrak{M}_x)$, the first Yoneda extension group in the exact category of strongly divisible Breuil--Kisin modules, cf. \cite[\S 4.1]{B18b}. It is easy to check this map is $\mathbb{F}$-linear. 
	
	If $\mathcal{W}$ is the multiset of integers containing $i$ with multiplicity equal to the $\mathbb{F}$-dimension of $\mathcal{G}^i(\mathfrak{M}_x)$ then \cite[4.2.5]{B18b} implies the right-hand side of \eqref{map} has $\mathbb{F}$-dimension equal to 
	$$
	\operatorname{dim}_{\mathbb{F}} \operatorname{Hom}(\mathfrak{M}_x,\mathfrak{M}_x)^{\varphi =1} + \operatorname{Card}\left\lbrace i-j>0 \mid i,j \in \mathcal{W} \right\rbrace
	$$
	Clearly the value of the double sum in the statement of the lemma equals the cardinality of $\left\lbrace i-j>0 \mid i,j \in \mathcal{W} \right\rbrace$, and so it remains to compute the dimension of the kernel of \eqref{map}. To do this we first claim this kernel is contained in the kernel of the composite
$$
\mathcal{O}_{\mathcal{L},x}(\mathbb{F}[\epsilon]) \rightarrow R(\mathbb{F}[\epsilon]) \rightarrow \operatorname{Ext}^1(V_{\mathbb{F}},V_{\mathbb{F}})
$$
Here the last maps sends $A \rightarrow \mathbb{F}[\epsilon]$ onto the exact sequence $0 \rightarrow V_{\mathbb{F}} \xrightarrow{\epsilon} V_R \otimes_R \mathbb{F}[\epsilon] \rightarrow V_{\mathbb{F}} \rightarrow 0$. If $\mathfrak{M}_{\mathbb{F}[\epsilon]}$ corresponds to an element in the kernel of \eqref{map} then the surjection $\mathfrak{M}_{\mathbb{F}[\epsilon]} \rightarrow \mathfrak{M}_x$ admits a $\varphi$-equivariant $\mathfrak{S}_{\mathbb{F}}$-linear splitting $s$. Since $\mathfrak{M}_{\mathbb{F}[\epsilon]} \rightarrow \mathfrak{M}$ becomes $G_K$-equivariant after applying $\otimes_{k[[u]]} C^\flat$ it follows that $(\sigma-1)(s) := \sigma \circ s \circ \sigma^{-1} - s$ is an element of $\operatorname{Hom}(\mathfrak{M}_x,\mathfrak{M}_x) \otimes_{k[[u]]} u^{\frac{p}{p-1}} \mathcal{O}_{C^\flat}$ for each $\sigma \in G_K$. Using Hypothesis~\ref{hyp} we deduce that $s$ is $G_K$-equivariant, and so $s$ induces a $G_K$-equivariant splitting of $V_{\mathbb{F}[\epsilon]} \rightarrow V_{\mathbb{F}}$. This proves the claim.

To finish the proof it therefore suffices to show that the kernel of $\mathcal{O}_{\mathcal{L},x}(\mathbb{F}[\epsilon]) \rightarrow R(\mathbb{F}[\epsilon]) \rightarrow \operatorname{Ext}^1(V_{\mathbb{F}},V_{\mathbb{F}})$ has dimension equal to 
$$
\leq d^2 - \operatorname{Hom}(\mathfrak{M}_x,\mathfrak{M}_x)^{\varphi=1}
$$
We claim that the kernel of the first map in the this composite is a torsor for 
$$
\operatorname{Hom}(V_{\mathbb{F}},V_{\mathbb{F}})^{G_K}/\operatorname{Hom}(\mathfrak{M}_x,\mathfrak{M}_x)^{\varphi=1}
$$
(note this makes sense since by Lemma~\ref{end} we do have $\operatorname{Hom}(\mathfrak{M}_x,\mathfrak{M}_x)^{\varphi=1} \subset \operatorname{Hom}(V_{\mathbb{F}},V_{\mathbb{F}})^{G_K}$). Since the kernel of the second map in this composite is clearly a torsor for $\operatorname{Hom}(V_{\mathbb{F}},V_{\mathbb{F}}) / \operatorname{Hom}(V_{\mathbb{F}},V_{\mathbb{F}})^{G_K}$, proving this claim will complete the argument.

To do this, note that any $h \in \operatorname{Hom}(V_{\mathbb{F}},V_{\mathbb{F}})^{G_K}$ produces an automorphism $a+b\epsilon \mapsto a + h(b)\epsilon$ of $V_{\mathbb{F}} \otimes_{\mathbb{F}} \mathbb{F}[\epsilon]$ which, when viewed as an automorphism of $V_{\mathbb{F}} \otimes_{\mathbb{F}} \mathbb{F}[\epsilon] \otimes_{\mathbb{F}_p} C^\flat$, acts on the set $X \subset \mathcal{L}_{\operatorname{crys}}^{\leq p}(V_{\mathbb{F}} \otimes_{\mathbb{F}} \mathbb{F}[\epsilon])$ containing those elements  corresponding to elements in the kernel of $\mathcal{O}_{\mathcal{L},x}(\mathbb{F}[\epsilon]) \rightarrow R(\mathbb{F}[\epsilon])$. This action is also transitive. To see this note that any two elements of $X$ are abstractly isomorphic as Breuil--Kisin modules by a $\varphi$-equivariant map inducing the identity modulo $\epsilon$. By Lemma~\ref{end} this isomorphism induces an automorphism of $V_{\mathbb{F}} \otimes_{\mathbb{F}} \mathbb{F}[\epsilon]$ which, being the identity modulo $\epsilon$, comes from some $h \in \operatorname{Hom}(V_\mathbb{F},V_{\mathbb{F}})^{G_K}$. Finally we note that $\operatorname{Hom}(\mathfrak{M}_{x},\mathfrak{M}_{x})^{\varphi =1}$ is the stabliser of any point of $X$ under this action.
\end{proof}
\begin{proposition}\label{smooooth}
	If $p=2$ assume that $K_\infty \cap K(\mu_{p^\infty}) = K$. Let $x \in \mathcal{L} $ be an $\mathbb{F}'$-valued closed point of $\mathcal{L}$ with $\mathcal{O}_{\mathcal{L},x}[\frac{1}{p}] \neq 0$, and assume that the corresponding element of $\mathcal{L}^{\leq p}_{\operatorname{crys}}(V_{\mathbb{F}'})$ satisfies (2) of Hypothesis~\ref{hyp}. Then $\mathcal{O}_{\mathcal{L},x}$ is $\mathbb{Z}_p$-flat and $\mathcal{O}_{\mathcal{L},x}/p$ is regular. The completion of $\mathcal{O}_{\mathcal{L},x}$ is a power series over $W(\mathbb{F}')$.
\end{proposition}
\begin{proof}
	Let $\mathfrak{p} \in \mathcal{O}_{\mathcal{L},x}[\frac{1}{p}]$ be a maximal ideal and let $\mathfrak{q}$ be its preimage in $\mathcal{O}_{\mathcal{L},x}$. Set $B$ equal to the residue field of $(\mathcal{O}_{\mathcal{L},x})_{\mathfrak{q}}$ and $C = \mathcal{O}_{\mathcal{L},x}/\mathfrak{q} \subset B$. By Lemma~\ref{commalgebra} we know $B$ is a finite extension of $\mathbb{Q}_p$ and $C$ is finite flat over $\mathbb{Z}_p$. 
	
	Let $y \in \mathcal{L}$ denote the image of $\operatorname{Spec} B \rightarrow \mathcal{L}$. Since $\mathcal{L}[\frac{1}{p}]$ is Jacobson and $B$ is finite over $\mathbb{Q}_p$, $y$ is closed in $\mathcal{L}[\frac{1}{p}]$, cf. \cite[Tag 01TB]{stacks-project}. The map $\operatorname{Spec} C \rightarrow \mathcal{L}$ corresponds to $\mathfrak{M}_C \in \mathcal{L}^{\leq p}_{\operatorname{crys}}(V_C)$ with $\mathfrak{M}_C \otimes_{C} \mathbb{F}'$ the Breuil--Kisin module corresponding to $x$. Lemma~\ref{flat} and Lemma~\ref{algebra} imply that $\mathcal{G}^i(\mathfrak{M}_C \otimes_{C}\mathbb{F}') \cong \mathcal{G}^i(\mathfrak{M}_C)  \otimes_{C} \mathbb{F}'$. If $V_B = V_C \otimes_C B$ is the representation of $G_K$ induced by $y$ then $V_B$ is crystalline and we also have $\mathcal{G}^i(\mathfrak{M}_C)  \otimes_{C} B \cong \operatorname{gr}^i(D_{\operatorname{dR}}(V_B))$, cf. \ref{filteredz}. If $\mathfrak{M}_x \in \mathcal{L}^{\leq p}_{\operatorname{crys}}(V_{\mathbb{F}})$ corresponds to $x$ then we deduce
	$$
	\operatorname{dim}_{\mathbb{F}'} \mathcal{G}^i(\mathfrak{M}_C \otimes_C \mathbb{F}') =  \operatorname{dim}_B \operatorname{gr}^i(D_{\operatorname{dR}}(V_B))
	$$
	From Proposition~\ref{quotient} we know $\mathcal{L}[\frac{1}{p}] = \operatorname{Spec} R^{\leq p}_{\operatorname{crys}}[\frac{1}{p}]$. By \cite[2.6.2]{Kis08} and \cite[3.3.8]{Kis08} the connected component of $R^{\leq p}_{\operatorname{crys}}[\frac{1}{p}]$ containing $y$ is equidimensional of dimension 
	$$
	\begin{aligned}
	d^2 + \sum_{i > 0} \operatorname{dim}_{\mathbb{F}}\operatorname{gr}^i&(D_{\operatorname{dR}}(\operatorname{Hom}(V_B,V_B)) = \\ &d^2 + \sum_{i > 0 } \sum_{n-m= i} \operatorname{dim}_{\mathbb{F}}\operatorname{gr}^n(D_{\operatorname{dR}}(V_B)) - \operatorname{dim}_{\mathbb{F}}\operatorname{gr}^m(D_{\operatorname{dR}}(V_B))
	\end{aligned}
	$$
	 Since $y$ is a closed point of $\mathcal{L}[\frac{1}{p}]$ this is the dimension of $\mathcal{O}_{\mathcal{L},y} = (\mathcal{O}_{\mathcal{L},x})_{\mathfrak{q}}$. From Lemma~\ref{dimm} we deduce
	\begin{equation*}\label{leqbutactuallly=}
	 \operatorname{dim}_{\mathbb{F}} \mathcal{O}_{\mathcal{L},x}(\mathbb{F}[\epsilon]) \leq \operatorname{dim} \mathcal{O}_{\mathcal{L},y}
	\end{equation*}
	On the other hand
	$$
	\operatorname{dim} \mathcal{O}_{\mathcal{L},y} \leq  \operatorname{dim} \mathcal{O}_{\mathcal{L},x} - 1 \leq \operatorname{dim} \mathcal{O}_{\mathcal{L},x}/p
	$$
	We've used Lemma~\ref{commalgebra} for the first inequality and \cite[Tag 00OM]{stacks-project} for the second. Hence $\mathcal{O}_{\mathcal{L},x}/p$ is regular and these two displayed inequalities are equalities.
	
	To show $\mathcal{O}_{\mathcal{L},x}$ is $\mathbb{Z}_p$-flat, note that $p$ is in the maximal ideal of $\mathcal{O}_{\mathcal{L},x}$, and so $\operatorname{dim}\mathcal{O}_{\mathcal{L},x}[\frac{1}{p}] \leq \operatorname{dim} \mathcal{O}_{\mathcal{L},x} - 1 = \operatorname{dim} \mathcal{O}_{\mathcal{L},y}$. As $\mathcal{O}_{\mathcal{L},y}$ is obtained from $\mathcal{O}_{\mathcal{L},x}[\frac{1}{p}]$ by localisation $\operatorname{dim}\mathcal{O}_{\mathcal{L},y} \leq \operatorname{dim} \mathcal{O}_{\mathcal{L},x}[\frac{1}{p}]$ and we have equality. Let $I \subset \mathcal{O}_{\mathcal{L},x}$ be the ideal of elements killed by a power of $p$. Then $\mathcal{O}_{\mathcal{L},x}[\frac{1}{p}] = (\mathcal{O}_{\mathcal{L},x}/I)[\frac{1}{p}]$ and so
	$$
	\operatorname{dim} \mathcal{O}_{\mathcal{L},x}/p = \operatorname{dim}\mathcal{O}_{\mathcal{L},x} - 1 = \operatorname{dim}\mathcal{O}_{\mathcal{L},x}[\tfrac{1}{p}] \leq \operatorname{dim} \mathcal{O}_{\mathcal{L},x}/I-1 \leq  \operatorname{dim}\mathcal{O}_{\mathcal{L},x}/(I,p)
	$$
	We conclude $\mathcal{O}_{\mathcal{L},x}/p$ and $\mathcal{O}_{\mathcal{L},x}/(I,p)$ have the same dimension. Since $\mathcal{O}_{\mathcal{L},x}/p$ is regular the image of $I$ in $\mathcal{O}_{\mathcal{L},x}/p$ must be zero, and so $I \subset p \mathcal{O}_{\mathcal{L},x}$. As such any $x \in I$ can be written as $x= py$; by the definition of $I$ we see $y \in I$ and so $I \subset \cap p^n \mathcal{O}_{\mathcal{L},x}  = 0$. We conclude $\mathcal{O}_{\mathcal{L},x}$ is $\mathbb{Z}_p$-flat. That the completion of $\mathcal{O}_{\mathcal{L},x}$ at its maximal ideal is a power series ring is then a standard consequence of the fact that $\mathcal{O}_{\mathcal{L},x}$ is $\mathbb{Z}_p$-flat and $\mathcal{O}_{\mathcal{L},x}/p$ is regular.
\end{proof}
\begin{corollary}\label{regular}
Assume that $V_{\mathbb{F}}$ is cyclotomic-free and that $K_\infty \cap K(\mu_{p^\infty}) = K$ if $p=2$. The closed subscheme of $\mathcal{L}$ defined by the ideal of $p$-power torsion sections is equal to a union of connected components of $\mathcal{L}$, and is regular. This closed subscheme therefore coincides with $\mathcal{L}^\circ$ defined in Lemma~\ref{Lcirc}.
\end{corollary}
\begin{proof}
Since $V_{\mathbb{F}}$ is cyclotomic-free, Proposition~\ref{smooooth} implies that a closed point $x \in \mathcal{L}$ is contained in $\mathcal{L}^\circ$ if and only if $\mathcal{O}_{\mathcal{L},x}$ is $\mathbb{Z}_p$-flat. Further, if this is the case then $\mathcal{O}_{\mathcal{L},x}$ is regular. Flatness implies $\mathcal{L}^\circ \subset \mathcal{L}$ is open, cf. \cite[Tag 00RC]{stacks-project}. We see $\mathcal{L}^\circ$ is regular as it is regular at closed points.
\end{proof}

\section{Applications to deformation rings}\label{applicationstodefrings}
\subsection{Potential diagonalisability}

\begin{sub}
	Now consider a finite extension $E$ of $\mathbb{Q}_p$, with residue field $\mathbb{F}$, and a finite free $\mathcal{O}_E$-module $V$ equipped with a continuous $\mathcal{O}_E$-linear action of $G_K$. Assume that $V[\frac{1}{p}]$ is crystalline of $p$-adic Hodge type $\mathbf{v}$. Let $V_{\mathbb{F}} = V \otimes_{\mathcal{O}_{E}} \mathbb{F}$ and let $R^{\mathbf{v}}_{\operatorname{crys}}$ denote the quotient of $R^{\square}_{V_{\mathbb{F}}}$ from Proposition~\ref{crysdef} (or more generally the quotient defined in \cite{Kis08} if $\mathbf{v}$ is not concentrated in degrees $[0,p]$). If $\xi$ is an $\mathcal{O}_E$-basis of $V$, we say $(V,\xi)$ is diagonalisable if the corresponding point of $R^{\mathbf{v}}_{\operatorname{crys}}$ lies on the same irreducible component\footnote{Since $R^{\mathbf{v}}_{\operatorname{crys}}$ is $\mathbb{Z}_p$-flat, this is equivalent to asking that the image of their generic fibres lie on the same irreducible component of $R^{\mathbf{v}}_{\operatorname{crys}}[\frac{1}{p}]$.} of $R^{\mathbf{v}}_{\operatorname{crys}}$ as an $\mathcal{O}_{\overline{E}}$-valued point (here $\mathcal{O}_{\overline{E}}$ denotes the ring of integers in an algebraic closure of $E$) whose corresponding representation is a direct sum of crystalline characters. Say $V$ is potentially diagonalisable if $V|_{G_L}$ is diagonalisable for some finite extension $L$ of $K$. These notions were introduced in \cite[1.4]{BLGGT}.
\end{sub}

\begin{lemma}\label{potdiagfacts}
	\begin{enumerate}
		\item Whether or not $V$ is potentially diagonalisable is independent of the choice of $\xi$.
		\item If  $V'$ is a $G_K$-stable $\mathcal{O}_E$-lattice inside $V[\frac{1}{p}]$ then $V$ is potentially diagonalisable if and only if $V'$ is.
		\item If $V[\frac{1}{p}]$ admits a $G_K$-stable filtration then $V$ is potentially diagonalisable if and only if each graded piece is potentially diagonalisable. In particular this is the case if each graded piece is one-dimensional.
	\end{enumerate} 
\end{lemma}
\begin{proof}
	Both (1) and (2) follow from \cite[1.4.1]{BLGGT}. For (3) we refer to \cite[2.1.2]{GL14}.
\end{proof}

We now prove the theorem from the introduction. Recall the definition of strongly cyclotomic-free is given in Definition~\ref{cyclofree}.

\begin{theorem}\label{potdiag}
	Assume $\mathbb{F}$ is sufficiently large and that the $p$-adic Hodge-type $\mathbf{v}$ is concentrated in degree $[0,p]$. If $V_\mathbb{F}$ is strongly cyclotomic-free then $V$ is potentially diagonalisable.
\end{theorem}
\begin{proof}
	First, if $p=2$ then we choose our compatible system of $p$-th-power roots of a uniformiser of $K$ so that $K_\infty \cap K(\mu_{p^\infty}) = K$. This can always be done, cf. \cite[Lemma 2]{Wang17}. As is $V_{\mathbb{F}}$ strongly cyclotomic-free, $V_{\mathbb{F}}|_{G_L}$ is strongly cyclotomic-free for any finite unramified extension $L/K$. As $\mathbb{F}$ is sufficiently large we may therefore assume each Jordan--Holder factor of $V_{\mathbb{F}}$ is one-dimensional. Under these assumptions we claim there exists another deformation $V'$ of $V_{\mathbb{F}}$ such that:
	\begin{itemize}
		\item $V'[\frac{1}{p}]$ is crystalline with Hodge--Tate weights in $[0,p]$.
		\item Every Jordan--Holder factor of $V'[\frac{1}{p}]$ is one-dimensional.
		\item If $\mathfrak{M}$ and $\mathfrak{M}'$ are the Breuil--Kisin modules associated to $V$ and $V'$ respectively, then $\mathfrak{M} \otimes_{\mathcal{O}_{E}} \mathbb{F} = \mathfrak{M}' \otimes_{\mathcal{O}_{E}} \mathbb{F}$ in $\mathcal{L}^{\leq p}_{\operatorname{crys}}(V_{\mathbb{F}})$.
	\end{itemize}  
	Such a $V'$ is constructed below, cf. Corollary~\ref{lift} below.
	
	Assuming for now that $V'$ can be constructed we now explain how this implies potential diagonalisability of $V$. Let $\mathcal{L} = \mathcal{L}^{\leq p}_{R,\operatorname{crys}}$ and $\mathcal{L}^\circ \subset \mathcal{L}$ be the $\mathbb{Z}_p$-flat locus from Corollary~\ref{regular}. Then it follows that $\mathfrak{M}$ and $\mathfrak{M}'$ induce $\mathcal{O}_E$-valued points of $\mathcal{L}^{\circ}$ (cf. (2) of Remark~\ref{points}). The image of the closed point in $\operatorname{Spec}\mathcal{O}_E$ under these two maps coincide, and so both $\mathcal{O}_E$-points lie on the same connected component of $\mathcal{L}^{\circ}$. By Corollary~\ref{regular}, $\mathcal{L}^\circ$ is normal, and so these points lie on the same irreducible component. Hence their images in $R^{\mathbf{v}}_{\operatorname{crys}}$ lie in the same irreducible component. By (3) of Lemma~\ref{potdiagfacts} we know $V'$ is potentially diagonalisable, and so $V$ is potentially diagonalisable also.
\end{proof}

To complete proof of Theorem~\ref{potdiag} we must construct a $V'$ as above. Thus we make the following definition:
\begin{definition}\label{crystlift}
	Let $\mathfrak{M}_{\mathbb{F}} \in \mathcal{L}^{\leq p}(V_{\mathbb{F}})$. We say $\mathfrak{M}_{\mathbb{F}}$ admits a crystalline lift if there exists a finite extension $E/\mathbb{Q}_p$ with residue field $\mathbb{F}'$ containing $\mathbb{F}$ and a finite free $\mathcal{O}_E$-module $V$ equipped with a continuous $\mathcal{O}_E$-linear action of $G_K$ so that (i) $V[\frac{1}{p}]$ is crystalline with Hodge--Tate weights in $[0,p]$, (ii) $V \otimes_{\mathcal{O}_{E}} \mathbb{F}' = V_{\mathbb{F}} \otimes_{\mathbb{F}} \mathbb{F}'$, and (iii) if $\mathfrak{M}$ denotes the Breuil--Kisin module associated to $V$ then $\mathfrak{M} \otimes_{\mathcal{O}_{E}} \mathbb{F}' = \mathfrak{M}_{\mathbb{F}} \otimes_{\mathbb{F}} \mathbb{F}'$ in $\mathcal{L}^{\leq p}(V_{\mathbb{F}'})$
\end{definition}
 
 For the next lemma consider a $G_{K_\infty}$-equivariant exact sequence $0 \rightarrow W_{\mathbb{F}} \rightarrow V_{\mathbb{F}} \rightarrow Z_{\mathbb{F}} \rightarrow 0$ of finite dimensional $\mathbb{F}$-vector spaces. If $\mathfrak{M}_{\mathbb{F}} \in \mathcal{L}^{\leq p}(V_{\mathbb{F}})$ then Lemma~\ref{exactO} produces an exact sequence $0 \rightarrow \mathfrak{W}_{\mathbb{F}} \rightarrow \mathfrak{M}_{\mathbb{F}} \rightarrow \mathfrak{Z}_{\mathbb{F}} \rightarrow 0$. 
 \begin{lemma}\label{liftingext}
 	Assume that $\mathfrak{W} \in \mathcal{L}^{\leq p}_{\operatorname{crys}}(W)$ and $\mathfrak{Z} \in \mathcal{L}^{\leq p}_{\operatorname{crys}}(Z)$ are crystalline lifts of $W_{\mathbb{F}}$ and $Z_{\mathbb{F}}$ respectively. Assume that $\mathfrak{W}_{\mathbb{F}} \oplus \mathfrak{Z}_{\mathbb{F}}$ satisfies (2) of Hypothesis~\ref{hyp}, and that $\mathfrak{M}_{\mathbb{F}}$ from the previous paragraph is strongly divisible. Then there exists a crystalline lift $\mathfrak{M} \in \mathcal{L}^{\leq p}_{\operatorname{crys}}(V)$ of $\mathfrak{M}_{\mathbb{F}}$ such that
 	\begin{enumerate}
 		\item $\mathfrak{M}$ fits into a $\varphi$-equivariant exact sequence $0 \rightarrow \mathfrak{W} \rightarrow \mathfrak{M} \rightarrow \mathfrak{Z} \rightarrow 0$ of $\mathfrak{S}_{\mathcal{O}_E}$-modules which recovers $0 \rightarrow \mathfrak{W}_{\mathbb{F}} \rightarrow \mathfrak{M}_{\mathbb{F}} \rightarrow \mathfrak{Z}_{\mathbb{F}} \rightarrow 0$ after applying basechanging to $\mathbb{F}'$.
 		\item $V$ fits into a $G_{K}$-equivariant exact sequence $0 \rightarrow W \rightarrow V \rightarrow Z \rightarrow 0$ which recovers $0 \rightarrow \mathfrak{W} \rightarrow \mathfrak{M} \rightarrow \mathfrak{Z} \rightarrow 0$ after basechanging to $W(C^\flat)$.
 	\end{enumerate}
 \end{lemma}
In particular if both $Z_{\mathbb{F}}$ and $W_{\mathbb{F}}$ are cyclotomic-free then $Z_{\mathbb{F}} \oplus W_{\mathbb{F}}$ is cyclotomic-free also, and so in this case (1) of Hypothesis~\ref{hyp} is satisfied for $\mathfrak{W}_{\mathbb{F}} \oplus \mathfrak{Z}_{\mathbb{F}}$. We remark also that a similar result is proven \cite[5.3.1]{B18b} but with a different notion of cyclotomic-free. 
\begin{proof}
	The element $\mathfrak{W}_{\mathbb{F}} \oplus \mathfrak{Z}_{\mathbb{F}} \in \mathcal{L}^{\leq p}_{\operatorname{crys}}(W_{\mathbb{F}} \oplus Z_{\mathbb{F}})$ defines a closed point $x$ of the $R = R^\square_{W_{\mathbb{F}} \oplus Z_{\mathbb{F}}}$-scheme  $\mathcal{L} = \mathcal{L}^{\leq p}_{R,\operatorname{crys}}$ from Corollary~\ref{formal}. Also $\mathfrak{W} \oplus \mathfrak{Z}$ defines an $\mathcal{O}_E$-valued point $y$ of $\mathcal{L}$ through which $x$ factors.
	
	View the extension $0 \rightarrow \mathfrak{W}_{\mathbb{F}} \rightarrow \mathfrak{M}_{\mathbb{F}} \rightarrow \mathfrak{Z}_{\mathbb{F}} \rightarrow 0$ as an extension of $\mathfrak{W}_{\mathbb{F}} \oplus \mathfrak{Z}_{\mathbb{F}}$ by itself. By assumption $\mathfrak{M}_{\mathbb{F}}$ is strongly divisible so we obtain an element of the set $\operatorname{Ext}^1_{\operatorname{SD}}(\mathfrak{W}_{\mathbb{F}} \oplus \mathfrak{Z}_{\mathbb{F}},\mathfrak{W}_{\mathbb{F}} \oplus \mathfrak{Z}_{\mathbb{F}})$ described in the proof of Lemma~\ref{dimm}. It follows from the proof of Proposition~\ref{smooooth} that the map $\mathcal{O}_{\mathcal{L},x}(\mathbb{F}[\epsilon]) \rightarrow \operatorname{Ext}^1_{\operatorname{SD}}(\mathfrak{W}_{\mathbb{F}}\oplus \mathfrak{Z}_{\mathbb{F}},\mathfrak{W}_{\mathbb{F}} \oplus \mathfrak{Z}_{\mathbb{F}})$ in \eqref{map} is surjective. Thus there is a tangent vector $x'\colon \operatorname{Spec}\mathbb{F}[\epsilon] \rightarrow \mathcal{L}$ mapping onto this extension class. Since the completion of $\mathcal{O}_{\mathcal{L},x}$ is a power series ring over $W(\mathbb{F})$ the point $y$ factors through a morphism $\operatorname{Spec}\mathcal{O}_E[\epsilon] \rightarrow \mathcal{L}$ lifting $x'$ (here $\mathcal{O}_E[\epsilon]$ denotes the ring of dual numbers over $\mathcal{O}_E$). This morphism induces an extension $0 \rightarrow W \oplus Z \rightarrow V' \rightarrow W \oplus Z \rightarrow 0$ as well as an $\mathfrak{M}' \in \mathcal{L}^{\leq p}_{\operatorname{crys}}(V')$ fitting into an exact sequence
	$$
	0 \rightarrow \mathfrak{W} \oplus \mathfrak{Z} \rightarrow \mathfrak{M}' \rightarrow \mathfrak{W} \oplus \mathfrak{Z} \rightarrow 0
	$$
	From this we obtain the representation $V$ and $\mathfrak{M} \in \mathcal{L}^{\leq p}_{\operatorname{crys}}(V)$ as desired.
\end{proof}

\begin{corollary}\label{lift}
	Suppose $V_{\mathbb{F}}$ is cyclotomic-free and every Jordan--Holder factor is one-dimensional. Let $\mathfrak{M}_{\mathbb{F}} \in \mathcal{L}^{\leq p}_{\operatorname{SD}}(V_{\mathbb{F}})$. Then $\mathfrak{M}_{\mathbb{F}}$ admits a crystalline lift $V$ so that every Jordan--Holder factor of $V[\frac{1}{p}]$ is one-dimensional.
\end{corollary}
\begin{proof}
	If $V_{\mathbb{F}}$ is one dimensional then the result is easy (see for example part (1) of \cite[Lemma 6.3]{GLS}). For the general case induct on the length of $V_{\mathbb{F}}$ using Lemma~\ref{liftingext}.
\end{proof}

In particular we see $\mathfrak{M} \in \mathcal{L}^{\leq p}_{\operatorname{crys}}(V_{\mathbb{F}})$.
\subsection{A possible improvement} We would now like to explain how Theorem~\ref{potdiag} can be strengthened, assuming a conjectural statement regarding the fibre of $\mathcal{L}$ over the closed point of $\operatorname{Spec}R$.
\begin{sub}
 As usual let $\mathbb{F}$ denote a finite field of characteristic $p$ and let $V_{\mathbb{F}}$ denote a finite-dimensional $\mathbb{F}$-vector space equipped with a continuous $\mathbb{F}$-linear action of $G_K$. Let $R = R^\square_{V_{\mathbb{F}}}$ and $\mathcal{L} = \mathcal{L}^{\leq p}_{R,\operatorname{crys}}$. We also let $\mathcal{L}_{\mathbb{F}} = \mathcal{L} \otimes_R \mathbb{F}$ be the fibre of $\mathcal{L}$ over the closed point of $\operatorname{Spec}R$.
\end{sub}
\begin{sub}\label{notationally}
	Let us first assume $V_{\mathbb{F}}$ is absolutely irreducible and $\mathbb{F}$ is sufficiently large, so that $V_{\mathbb{F}} = \operatorname{Ind}_{G_L}^{G_K} W_{\mathbb{F}}$ with $L/K$ unramified and $W_{\mathbb{F}}$ one-dimensional. We also assume that the residue field $l$ of $L$ embeds into $\mathbb{F}$. Recall from Lemma~\ref{inducerestrict} that there is a map $f_*\colon \mathcal{L}^{\leq p}(W_{\mathbb{F}}) \rightarrow \mathcal{L}^{\leq p}(V_{\mathbb{F}})$. It follows from \ref{sdcondition} and Lemma~\ref{valu} that the image of this map lies in $\mathcal{L}^{\leq p}_{\operatorname{crys}}(V_{\mathbb{F}})$. 
\end{sub}
\begin{conjecture}\label{conjecture}
	With $V_{\mathbb{F}}$ as in \ref{notationally}, every closed point of $\mathcal{L}_{\mathbb{F}}$ lies in the same connected component as a closed point arising from $f_*\mathfrak{N}$ for some $\mathfrak{N} \in \mathcal{L}^{\leq p}(W_{\mathbb{F}})$.
\end{conjecture}

We are going to prove this when $V_{\mathbb{F}}$ is $2$-dimensional. Before doing so we record some consequences of this conjecture.
\begin{lemma}\label{irredlift}
	Suppose Conjecture~\ref{conjecture} holds, and if $p=2$ that $K_\infty \cap K(\mu_{p^\infty}) = K$. Then any $\mathfrak{M} \in \mathcal{L}^{\leq p}_{\operatorname{crys}}(V_{\mathbb{F}})$ admits a crystalline lift.
\end{lemma}

In particular we deduce $\mathcal{L}^\circ  = \mathcal{L}$ in this situation.
\begin{proof}
We have to show the local ring of $\mathcal{L}$ at the closed point corresponding to $\mathfrak{M}$ is non-zero after inverting $p$. After Corollary~\ref{regular} it suffices to show every connected component of $\mathcal{L}$ contains at least one closed point admitting a crystalline lift. Using the conjecture we are reduced to proving that, if $\mathfrak{N} \in \mathcal{L}^{\leq p}(W_{\mathbb{F}})$, $f_*\mathfrak{N}$ admits a crystalline lift, and this is easy. Choose a crystalline character lifting $\mathfrak{N}$ and consider the induction of that character from $G_L$ to $G_K$.
\end{proof}
\begin{lemma}\label{generallift}
	Suppose Conjecture~\ref{conjecture} for all absolutely irreducible $V_{\mathbb{F}}$, and if $p=2$ that $K_\infty \cap K(\mu_{p^\infty}) =K$. Then any $\mathfrak{M} \in \mathcal{L}^{\leq p}_{\operatorname{SD}}(V_{\mathbb{F}}')$ with $V_{\mathbb{F}}'$ cyclotomic-free (but not necessarily irreducible) admits a crystalline lift $V$ such that every Jordan--Holder factor of $V[\frac{1}{p}]$ has irreducible reduction modulo~$p$.
\end{lemma}
\begin{proof}
	Using Lemma~\ref{irredlift} this follows as in Corollary~\ref{lift}, by inductively applying Lemma~\ref{liftingext}.
\end{proof}
\begin{corollary}
	Suppose conjecture~\ref{conjecture} holds for all absolutely irreducible $V_{\mathbb{F}}$. Then Theorem~\ref{potdiag} holds with strong cyclotomic-freeness replaced by cyclotomic-freeness.
\end{corollary}
\begin{proof}
	First suppose the $V \otimes_{\mathcal{O}_{E}} \mathbb{F} = V_{\mathbb{F}}$ is irreducible. Conjecture~\ref{conjecture} then implies $V$ lies in the same irreducible component of $R^{\mathbf{v}}_{\operatorname{crys}}$ as a point obtained by inducing a crystalline character over an unramified extension. Since such points are potentially diagonalisable (after a finite extension they become a sum of crystalline characters) this proves the result in the irreducible case.
	
	 For the general case, we know after Lemma~\ref{generallift} that $V$ lies in the same component as a point whose Jordan--Holder factors are all irreducible modulo~$p$. The previous paragraph implies each of these Jordan--Holder factors is potentially diagonalisable, and so Lemma~\ref{potdiagfacts} implies the point itself is potentially diagonalisable. We conclude $V$ is also.
\end{proof}
\begin{proposition}
	Conjecture~\ref{conjecture} holds if $V_{\mathbb{F}}$ is two-dimensional.
\end{proposition}
\begin{proof}
	Replacing $V_\mathbb{F}$ by an unramified twist (which is allowable by Lemma~\ref{unramtwist!} and the comment made in \ref{functorialremark}) we can assume the situation is as in the proof of Proposition~\ref{irredequal}, cf. in particular \ref{contain} and \ref{explicitrankone}. Thus there is an $\mathfrak{N} \in \mathcal{L}^{\leq p}(W_{\mathbb{F}})$ with generators $(e_\theta)_{\theta \in \operatorname{Hom}_{\mathbb{F}_p}(l,\mathbb{F})}$ satisfying	
	$$
	\varphi(e_{\theta \circ \varphi}) = u^{r_\theta}e_\theta, \qquad 0 \leq r_\theta \leq p
	$$
	together with an inclusion $\mathfrak{M} \subset f_*\mathfrak{N}$ (where $f_*\mathfrak{N}$ denotes $\mathfrak{N}$ viewed as an $\mathfrak{S}_{\mathbb{F}}$-module). For $\mathfrak{M}$ to be contained in $\mathcal{L}^{\leq p}_{\operatorname{crys}}(V_{\mathbb{F}})$ it is necessary and sufficient that:
	\begin{itemize}
		\item If $m \in \mathfrak{M}$ then $\varphi(m) \in \mathfrak{M}$. If $\varphi(m) \in u^{p+1}\mathfrak{M}$ then $m \in u\mathfrak{M}$. 
		\item If $\sum \alpha_\theta e_\theta \in \mathfrak{M}$ with $\alpha_\theta \in \mathbb{F}$ then $\sum_{r_\theta \equiv r \operatorname{mod}p} \alpha_\theta e_\theta \in \mathfrak{M}$.
	\end{itemize}
	(cf. \ref{equivalent}). Recall that the first condition is implied by $u^p \mathfrak{M} \subset \mathfrak{M}^\varphi \subset \mathfrak{M}$, and it implies $ue_\theta \in \mathfrak{M}$ for every $\theta$. 
	
	For $\tau \in \operatorname{Hom}_{\mathbb{F}_p}(k,\mathbb{F})$ we write $\mathfrak{M}_\tau$ for the summand of $\mathfrak{M}$ on which $k$ acts through $\tau$ (with $\mathfrak{M}$ viewed as an $\mathbb{F}[[u]]$-module). If $\theta \in \operatorname{Hom}_{\mathbb{F}_p}(l,\mathbb{F})$ is such that $\theta|_k =\tau$ then elements of $\mathfrak{M}_\tau$ have the form $\alpha e_\theta + \beta e_{\theta \circ \varphi^h}$ where $h = [K:\mathbb{Q}_p]$. In particular the possible shapes of the $\mathfrak{M}_\tau$ can be divided into two:
	\begin{enumerate}[label=(\roman*)]
		\item Either there is an $\alpha \in \mathbb{F}^\times$ so that $e_\theta + \alpha e_{\theta \circ \varphi^h}$ and $ue_\theta$ generate $\mathfrak{M}_\tau$ over $\mathbb{F}[[u]]$ (for some $\theta$ with $\theta|_k = \tau$). 
		\item Or no such $\alpha$ exists. Thus $\mathfrak{M}_\tau$ is generated by $u^{x_\theta}e_\theta$ and $u^{x_{\theta \circ \varphi^h}}e_{\theta \circ \varphi^h}$ for some $x_{\theta},x_{\theta \circ \varphi^h} \in [0,1]$ (again $\theta$ is some embedding with $\theta|_k =\tau$).
	\end{enumerate}
	Set $d(\mathfrak{M})$ equal to the number of $\tau$ as in case (i). Note that if $d(\mathfrak{M}) = 0$ then $\mathfrak{M} = f_*\mathfrak{N}'$ where $\mathfrak{N}' \in \mathcal{L}^{\leq p}(W_{\mathbb{F}})$ is the $\mathfrak{S}_{\mathbb{F}}$-submodule of $\mathfrak{N}$ generated by $u^{x_\theta}e_\theta$. Arguing by induction it therefore suffices to show $\mathfrak{M}$ lies in the same connected component of $\mathcal{L}_{\mathbb{F}}$ as an $\mathfrak{M}' \subset f_*\mathfrak{N}$ with $d(\mathfrak{M}') < d(\mathfrak{M})$. For this we will need a lemma.
	\begin{lemma}\label{moving}
		Suppose $\mathfrak{M}_\tau$ and $\mathfrak{M}_{\tau \circ \varphi}$ are as in (i) and $\theta \in \operatorname{Hom}_{\mathbb{F}_p}(l,\mathbb{F})$ with $\theta|_k = \tau$. Then $\mathfrak{M}_\tau$ is generated by $e_\theta + \alpha e_{\theta \circ \varphi^h}$ and $ue_\theta$, for some $\alpha \in \mathbb{F}^\times$, if and only if $\mathfrak{M}_{\tau \circ \varphi}$ is generated by $e_{\theta \circ \varphi} + \alpha e_{\theta \circ \varphi^{h+1}}$ and $ue_{\theta \circ \varphi}$. Furthermore, $r_\theta = r_{\theta \circ \varphi^h}$.
	\end{lemma}
	\begin{proof}
		The second bullet point above implies $r_\theta \equiv r_{\theta \circ \varphi^h}$ modulo~$p$, so we have equality except possibly if $r_\theta = 0$ or $p$. Suppose $r_\theta  =0$ so that $r_{\theta \circ \varphi^h}$ equals $0$ or $p$. As such, if $e_{\theta \circ \varphi} + \alpha e_{\theta \circ \varphi^{h+1}} \in \mathfrak{M}_{\tau \circ \varphi}$, then applying $\varphi$ shows that $e_{\theta} + \alpha u^{r_{\theta \circ \varphi^h}} e_{\theta \circ \varphi^h} \in \mathfrak{M}_\tau$. If $r_{\theta \circ \varphi^h} = p$ then $e_\theta \in \mathfrak{M}_\tau$ which contradicts the fact that $\mathfrak{M}_\tau$ is as in (i). Thus $r_{\theta \circ \varphi^h} = 0$ and $e_\theta + \alpha e_{\theta \circ \varphi^h} \in \mathfrak{M}_\tau$. This proves the lemma when $r_\theta  = 0$. 
		
		Now suppose $r_\theta >0$. Then $r_{\theta \circ \varphi^h} = r_\theta$ except possibly if $r_\theta = p$ and $r_{\theta \circ \varphi^h} = 0$. Applying the previous paragraph with $\theta$ replaced by $\theta \circ \varphi^h$ shows the exceptional case is impossible. Suppose $\alpha \in \mathbb{F}$ is such that $e_\theta + \alpha e_{\theta \circ \varphi^h} \in \mathfrak{M}_\tau$. The first bullet point above implies that $e_{\theta \circ \varphi} + \alpha e_{\theta \circ \varphi^{h+1}} \in \mathfrak{M}_{\tau \circ \varphi}$ (since $\varphi$ maps $u(e_{\theta \circ \varphi} + \alpha e_{\theta \circ \varphi^{h+1}})$ onto $u^{p+r_\theta} (e_\theta + \alpha e_{\theta \circ \varphi^h})$). This proves the lemma when $r_\theta >0$. 
	\end{proof}
	
	If $M_\mathbb{F}$ is the etale $\varphi$-module associated to $V_{\mathbb{F}}$ set $M_{\mathbb{F}[T]} = M_{\mathbb{F}} \otimes_{\mathbb{F}} \mathbb{F}[T]$, for a formal variable $T$. We are going to construct $\mathfrak{M}_{\mathbb{F}[T]} \in \mathcal{L}^{\leq p}_{\operatorname{crys}}(V_{\mathbb{F}} \otimes_{\mathbb{F}} \mathbb{F}[T])$ with $\mathfrak{M}_{\mathbb{F}[T]} \subset (f_*\mathfrak{N}) \otimes_{\mathbb{F}} \mathbb{F}[T]$, which at $T = 1$ recovers $\mathfrak{M}$ and which at $T= 0$ produces $\mathfrak{M}'$ with $d(\mathfrak{M}') < d(\mathfrak{M})$. As $\mathfrak{M}_{\mathbb{F}[T]}$ induces a morphism $\mathbb{A}^1 \rightarrow \mathcal{L}_{\mathbb{F}}$ connecting $\mathfrak{M}$ and $\mathfrak{M}'$ this will complete the proof.
	
	To produce $\mathfrak{M}_{\mathbb{F}[T]}$ choose $\tau \in \operatorname{Hom}_{\mathbb{F}_p}(k,\mathbb{F})$ so that $J = \lbrace \tau \circ \varphi^n, \tau \circ \varphi^{n-1},\ldots, \tau \rbrace$ is such that $\mathfrak{M}_{\tau \circ \varphi^j}$ is as in (i) for $0 \leq j \leq n$ and as in (ii) for $j = -1$ and $n+1$. We can assume such a $\tau$ exists for the following reasons. We can always assume there is a $\tau$ with $\mathfrak{M}_\tau$ as in (i) as otherwise $d(\mathfrak{M}) =0$. If $\mathfrak{M}_\tau$ is as in (i) for every $\tau$ then Lemma~\ref{moving} would imply $r_\theta = r_{\theta \circ \varphi^h}$ for every $\theta \in \operatorname{Hom}_{\mathbb{F}_p}(l,\mathbb{F})$. In this case the submodule of $f_*\mathfrak{N}$ generated by the $e_\theta + e_{\theta \circ \varphi^h}$ for all $\theta \in \operatorname{Hom}_{\mathbb{F}_p}(l,\mathbb{F})$ is $\varphi$-stable and so corresponds to a $G_{K_\infty}$-subrepresentation of $V_{\mathbb{F}}$, contradicting irreducibility.
	
	Choose $\theta \in \operatorname{Hom}_{\mathbb{F}_p}(l,\mathbb{F})$ so that $\theta|_k = \tau$; there is an $\alpha \in \mathbb{F}^\times$ such that $\mathfrak{M}_\tau$ is generated over $\mathbb{F}[[u]]$ by $e_{\theta} + \alpha e_{\theta \circ \varphi^h}$ and $ue_{\theta}$. Using Lemma~\ref{moving} we see $\mathfrak{M}_{\tau \circ \varphi^j}$ is generated by $e_{\theta \circ \varphi^j} + \alpha e_{\theta \circ \varphi^{h+j}}$ and $ue_{\theta \circ \varphi^{h+j}}$ for $0 \leq j \leq n$. For $0 \leq j \leq n$ we define 
	$$
	\mathfrak{M}_{\mathbb{F}[T],\tau \circ \varphi^j} = \text{ the $\mathbb{F}[[u]][T]$-linear span of $e_{\theta \circ \varphi^j} + T\alpha e_{\theta \circ \varphi^{h+j}}$ and $ue_{\theta \circ \varphi^j}$}
	$$
	(note this is well-defined; for $0 \leq j,j' \leq n$ if $\tau \circ \varphi^j = \tau \circ \varphi^{j'}$ then $j = j'$, otherwise $J = \operatorname{Hom}_{\mathbb{F}_p}(k,\mathbb{F})$ which we've shown in th paragraph above is impossible). For $\tau' \not\in J$ we set $\mathfrak{M}_{\mathbb{F}[T],\tau'} = \mathfrak{M}_{\tau'} \otimes_{\mathbb{F}} \mathbb{F}[T]$. Define $\mathfrak{M}_{\mathbb{F}[T]} = \bigoplus_{\tau' \in \operatorname{Hom}_{\mathbb{F}_p}(k,\mathbb{F})} \mathfrak{M}_{\mathbb{F}[T],\tau'}$. This is a projective $\mathfrak{S}_{\mathbb{F}[T]}$-module inside $M_{\mathbb{F}[T]}$, which by construction equals $\mathfrak{M}$ at $T =1$. The scheme from Proposition~\ref{kisprop} is described as a closed subscheme of the affine Grassmannian, and so $\mathcal{L}_{\mathbb{F}}$ is also closed in the affine Grassmannian. Thus to show $\mathfrak{M}_{\mathbb{F}[T]} \in \mathcal{L}^{\leq p}_{\operatorname{crys}}(V_{\mathbb{F}} \otimes_{\mathbb{F}} \mathbb{F}[T])$ it suffices to show that $\mathfrak{M}_\lambda \in \mathcal{L}^{\leq p}_{\operatorname{crys}}(V_{\mathbb{F}} \otimes_{\mathbb{F}} \mathbb{F}')$ whenever $\mathbb{F}'$ is a finite extension of $\mathbb{F}$ and $\mathfrak{M}_\lambda$ is obtained from $\mathfrak{M}_{\mathbb{F}[T]}$ by evaluating $T$ at $\lambda \in \mathbb{F}'$. This means verifying the two bullet points above for $\mathfrak{M}_\lambda$. The second bullet point is clear from the construction of $\mathfrak{M}_{\mathbb{F}[T]}$. To show the first we only have to check $u^p \mathfrak{M}_{\lambda} \subset \mathfrak{M}^\varphi_{\lambda} \subset \mathfrak{M}_{\lambda}$. If $\tau'$ is such that both $\tau'$ and $\tau' \circ \varphi$ are not in $J$ then
	$$
	u^p\mathfrak{M}_{\lambda,\tau'} \subset \mathfrak{M}_{\lambda,\tau'}^\varphi \subset \mathfrak{M}_{\lambda,\tau'}
	$$
	since $\mathfrak{M}_{\lambda, \tau' \circ \varphi} = \mathfrak{M}_{\tau' \circ \varphi}$ (so that $\mathfrak{M}^\varphi_{\lambda,\tau'} = \mathfrak{M}^\varphi_{\tau'}$) and $\mathfrak{M}_{\lambda,\tau'} = \mathfrak{M}_{\tau'}$. For the remaining $\tau'$ we choose $\mathbb{F}[[u]]$-bases $\xi_{\tau'}$ of $\mathfrak{M}_{\lambda,\tau'}$ as follows:
	\begin{enumerate}
		\item If $\tau' = \tau \circ \varphi^j \in J$, so that $0 \leq j \leq n$, then $ \xi_{\tau'} = (e_{\theta \circ \varphi^j} + \lambda \alpha e_{\theta \circ \varphi^{h+j}},ue_{\theta \circ \varphi^{h+j}})$.
		\item If $\tau' = \tau \circ \varphi^{-1}$ or $\tau \circ \varphi^{n+1}$ not in $J$, then $\xi_{\tau'} = (u^{\delta_j} e_{\theta \circ \varphi^j}, u^{\delta_{j+h}} e_{\theta \circ \varphi^{j+h}})$ where $\delta_j$ and $\delta_{j+h}$ respectively equal $x_{\theta \circ \varphi^j}$ and $x_{\theta \circ \varphi^{j+h}}$ as defined in (ii) above. In particular, both $\delta_j$ and $\delta_{j+h}$ are integers in $[0,1]$.
	\end{enumerate}
	There are matrices $A_{\tau'}$ valued in $\mathbb{F}[[u]]$ so that $\varphi(\xi_{\tau' \circ \varphi}) = (\xi_{\tau'}) A_{\tau'}$. If $\tau' = \theta \circ \varphi^j$ with $\tau' \circ \varphi$ and $\tau' \in J$ (i.e. $0 \leq j \leq n-1$) then we compute
	$$
	A_{\tau'} = \begin{pmatrix}
	u^{r_{\theta \circ \varphi^j}} & 0 \\ 0 & u^{p-1 + r_{\theta \circ \varphi^{h+j}}}
	\end{pmatrix}
	$$
	This doesn't depend upon $\lambda$ so we deduce that $u^p \mathfrak{M}_{\lambda,\tau'} \subset \mathfrak{M}_{\lambda,\tau'}^\varphi \subset \mathfrak{M}_{\lambda,\tau'}$ from the fact that it holds when $\lambda  =1$. If $\tau'= \tau$, so  $\tau' \in J$ but $\tau' \circ \varphi^{-1} \not\in J$, then we compute
	$$
	A_{\tau'} = \begin{pmatrix}
	u^{r_{\theta \circ \varphi^{-1}} - \delta_{-1}} & 0 \\ \lambda \alpha u^{r_{\theta \circ \varphi^{h-1}} - \delta_{h-1}} & u^{p + r_{\theta \circ \varphi^{h-1}} - \delta_{h-1}}
	\end{pmatrix}
	$$
	As such, in order that $u^p\mathfrak{M}_{\lambda,\tau'} \subset \mathfrak{M}^\varphi_{\lambda,\tau'} \subset \mathfrak{M}_{\lambda,\tau'}$, we need both $r_{\theta \circ \varphi^{-1}} - \delta_{-1} =  r_{\theta \circ \varphi^{h-1}} - \delta_{h-1} = 0$. Setting $\lambda =1$ we see this is the case, so it is the case for any $\lambda$. Finally we consider the case $\tau' = \tau \circ \varphi^n$ so that $\tau' \circ \varphi \not\in J$ and $\tau' \in J$. Then 
	$$
	A_{\tau'} = \begin{pmatrix}
	u^{p\delta_{n+1} + r_{\theta \circ \varphi^n}} & 0 \\ -\lambda \alpha u^{p\delta_{n+1} + r_{\theta \circ \varphi^n} -1} & u^{p\delta_{n+h+1} + r_{\theta \circ \varphi^{n+h}} -1}
	\end{pmatrix}
	$$
	Again whether or not $u^p\mathfrak{M}_{\lambda,\tau'} \subset \mathfrak{M}^\varphi_{\lambda,\tau'} \subset \mathfrak{M}_{\lambda,\tau'}$ is a condition on the powers of $u$ appearing in this matrix (we must have $p\delta_{n+1} + r_{\theta \circ \varphi^n} \in [1,p]$ and $p\delta_{n+h + 1} + r_{\theta \circ \varphi^{n+h}} - 1 \in [0,p-1]$). As these conditions holds with $\lambda =1$ they hold for general $\lambda$. 
\end{proof}
\begin{example}\label{example}
	While it seems plausible that the same kind of strategy could be used to give a full proof of Conjecture~\ref{conjecture}, in higher dimensions the situation is more complicated. Below we give an example which illustrates this difficulty. Suppose $K= \mathbb{Q}_p$ and that $L/K$ is the unramified extension of degree $7$. For an appropriate one-dimensional representation $W_{\mathbb{F}}$ of $G_L$ there is an $\mathfrak{N} \in \mathcal{L}^{\leq p}(W_{\mathbb{F}})$ so that $\mathfrak{N}$ is generated by $(e_6,e_5,\ldots,e_1,e_0)$ with
	$$
	\varphi(e_6,\ldots, e_0) = (e_5,ue_4,e_3,ue_2,e_1,u^2e_0,ue_6)
	$$
	If $\mathfrak{M} \subset f_*\mathfrak{N}$ is the $\mathbb{F}[[u]]$-submodule generated by
	$$
	e_6 + e_4 + e_2, e_5+ e_3, ue_4, ue_3, ue_2, e_1,e_0
	$$
	then one easily checks that $\mathfrak{M} \in \mathcal{L}^{\leq p}_{\operatorname{crys}}(V_{\mathbb{F}})$ where $V_{\mathbb{F}}  = \operatorname{Ind}_{G_L}^{G_K} W_{\mathbb{F}}$. If we define $\mathfrak{M}_{\mathbb{F}[T]}$ to be generated over $\mathbb{F}[[u]][T]$ by
	$$
	e_6 + T^2e_4 + Te_2, e_5 + Te_3, ue_4,ue_3,ue_2, e_1,e_0
	$$
	then one also checks by hand that this defines an element of $\mathcal{L}^{\leq p}_{\operatorname{crys}}(V_{\mathbb{F}} \otimes_{\mathbb{F}} \mathbb{F}[T])$. Thus we obtain a morphism $\mathbb{A}^1 \rightarrow \mathcal{L}_{\mathbb{F}}$ connecting $\mathfrak{M}$ with $f_*\mathfrak{N}'$ where $\mathfrak{N}' \subset \mathfrak{N}$ is generated by $e_6,e_5,ue_4,ue_3,ue_2,e_1,e_0$.
\end{example}
\subsection{Final remarks}

\begin{sub}
	So far we've seen that the scheme $\mathcal{L}^\circ$ describes the irreducible components of $R^{\mathbf{v}}_{\operatorname{crys}}$ (they correspond to the connected components of $\mathcal{L}^\circ$). In some specific situations we can do better. As usual let $\mathbb{F}$ be a finite field of characteristic $p$ and $V_{\mathbb{F}}$ a representation of $G_K$ on an $\mathbb{F}$-vector space. Assume that $V_{\mathbb{F}}$ is cyclotomic-free, and if $p=2$ that $K_\infty \cap K(\mu_{p^\infty}) = K$. We also fix a $p$-adic Hodge type $\mathbf{v}$ concentrated in degree $[0,p]$.
\end{sub}

The following notation is taken from \cite[5.1.3]{BM02}. For complete local Noetherian $\mathbb{Z}_p$-algebras $R,R_1,\ldots,R_r$ write $R \sim \prod_{i=1}^r R_i$ if there exists a $\mathbb{Z}_p$-algebra homomorphism $R \rightarrow \prod R_i$ which becomes an isomorphism after inverting $p$ and is such that each projection $R \rightarrow \prod R_i \rightarrow R_j$ is surjective.
\begin{proposition}\label{goodcase}
	Suppose that the fibre of $\mathcal{L}^{\mathbf{v}}$ over the closed point of $\operatorname{Spec}R$ is reduced and zero-dimensional and assume $\mathbb{F}$ is sufficiently large. Then 
	$$
	R^{\mathbf{v}}_{\operatorname{crys}} \sim \prod \mathcal{O}_{\mathcal{L},x}
	$$
	where the product runs over the closed points of $\mathcal{L}^{\mathbf{v}}$.\footnote{In particular we are asserting that each $\mathcal{O}_{\mathcal{L},x}$ is complete.}
\end{proposition}
\begin{proof}
	The fibre of $\mathcal{L}^{\mathbf{v}} \rightarrow \operatorname{Spec}R$ over the closed point being zero-dimensional implies this map is quasi-finite and therefore finite since it is projective. In particular $\mathcal{L}^{\mathbf{v}} = \operatorname{Spec}S$ is affine and the induced finite map $R \rightarrow S$ becomes surjective after inverting $p$. By definition $R^{\mathbf{v}}_{\operatorname{crys}}$ is the quotient of $R$ by the kernel of this map. Clearly $S$ is a product of the local rings of $\mathcal{L}^{\mathbf{v}}$ at its closed points; all that remains is to show the maps $R^{\mathbf{v}}_{\operatorname{crys}} \rightarrow \mathcal{O}_{\mathcal{L}^\mathbf{v},x}$ are surjective. Since $\mathbb{F}$ is assumed to be sufficiently large we may assume this map is an isomorphism on residue fields, so we only need to show that $\mathfrak{m}_{R^\mathbf{v}_{\operatorname{crys}}} \mathcal{O}_{\mathcal{L}^\mathbf{v},x} = \mathfrak{m}_{\mathcal{L}^\mathbf{v},x}$ (i.e. that $R^{\mathbf{v}}_{\operatorname{crys}} \rightarrow \mathcal{O}_{\mathcal{L}^{\mathbf{v}},x}$ is unramified). This follows from the assumption that the fibre of $\mathcal{L}^{\mathbf{v}}$ at the closed point of $\operatorname{Spec}R$ is reduced.
\end{proof}

\begin{sub}\label{matrices}
	We now show how to verify the hypotheses of Proposition~\ref{goodcase} in explicit cases. Let $\mathfrak{M} \in \mathcal{L}^{\leq p}_{\operatorname{crys}}(V_{\mathbb{F}})$ and suppose $\mathfrak{M}_{\mathbb{F}[\epsilon]} \in \mathcal{L}^{\leq p}_{\operatorname{crys}}(V_{\mathbb{F}} \otimes_{\mathbb{F}} \mathbb{F}[\epsilon])$ is such that $\mathfrak{M}_{\mathbb{F}[\epsilon]} \otimes_{\mathbb{F}[\epsilon]} \mathbb{F} = \mathfrak{M}$. As in the proof of Lemma~\ref{dimm} we can view $\mathfrak{M}_{\mathbb{F}[\epsilon]}$ as fitting into an exact sequence $0 \rightarrow \mathfrak{M} \rightarrow \mathfrak{M}_{\mathbb{F}[\epsilon]} \rightarrow \mathfrak{M} \rightarrow 0$. If $\xi$ is an $\mathbb{F}[[u]]$-basis of $\mathfrak{M}$ (viewed as a row vector) then any $\mathfrak{S}_{\mathbb{F}}$-splitting corresponds to an $X \in \operatorname{Mat}(\mathbb{F}((u)))$; the splitting sends $\xi$ onto $\xi(1+\epsilon X)$. As explained in \ref{ext}, any such splitting gives rise to an $f \in \frac{1}{u^p}\operatorname{Hom}(\mathfrak{M},\mathfrak{M})$ such that 
	$$
	\varphi_{\mathfrak{M}_{\mathbb{F}[\epsilon]}} = (\varphi_{\mathfrak{M}} + f \circ \varphi_{\mathfrak{M}}, \varphi_{\mathfrak{M}})
	$$
	Writing $\varphi_{\mathfrak{M}}(\xi) = \xi A$ for some matrix $A$ we compute that the matrix of $f$ with respect to $\xi$ is given by $A \varphi(X) A^{-1} - X$. As $\mathfrak{M}_{\mathbb{F}[\epsilon]}$ is strongly divisible, Proposition~\ref{SDexact} implies that $X$ can be chosen so that $A \varphi(X) A^{-1} - X \in \operatorname{Mat}(\mathbb{F}[[u]])$. Any such $\mathfrak{M}_{\mathbb{F}[\epsilon]}$ corresponds to a tangent vector, around the closed point corresponding to $\mathfrak{M}$, mapping onto the zero tangent vector at the closed point of $\operatorname{Spec}R$. Such tangent vectors fit into the diagram
	$$
	\begin{tikzcd}
	\operatorname{Spec}\mathbb{F}[\epsilon] \arrow[d] \arrow[r] & \mathcal{L} \arrow[d]\\
	\operatorname{Spec}\mathbb{F} \arrow[r] & \operatorname{Spec}R
	\end{tikzcd}
	$$
	and so describe the tangent vectors of  $\mathcal{L} \otimes_R \mathbb{F}$ at closed points. If every such tangent vector is zero then it will follow that $\mathcal{L} \otimes_R \mathbb{F}$ is zero-dimensional and smooth, and so verifies the hypothesis of Proposition~\ref{goodcase}. In other words, to verify the hypothesis of Proposition~\ref{goodcase}, it suffices to show that $\mathfrak{M}_{\mathbb{F}[\epsilon]} = \mathfrak{M} \otimes_{\mathbb{F}} \mathbb{F}[\epsilon]$ if $X \in \operatorname{Mat}(\mathbb{F}((u)))$ is such that $A\varphi(X)A^{-1} - X \in \operatorname{Mat}(\mathbb{F}[[u]])$ then $X \in \operatorname{Mat}(\mathbb{F}[[u]])$.
\end{sub}
We conclude with some in examples in which Proposition~\ref{goodcase} can be applied.
\begin{example}
	Suppose $K = \mathbb{Q}_p$ and suppose $\mathbb{F}$ is sufficiently large. A $p$-adic Hodge type $\mathbf{v}$ concentrated in degree $[0,p]$ then corresponds to a pair of integers $0 \leq a \leq b \leq p$. Suppose $V_{\mathbb{F}} = \operatorname{Ind}_{G_L}^{G_K} W_{\mathbb{F}}$ is two-dimensional and irreducible, with $W_{\mathbb{F}}$ one-dimensional. In this case we shall show $R^{\mathbf{v}}_{\operatorname{crys}}$ is either zero, or is formally smooth over $\mathbb{Z}_p$.
	\begin{itemize}
		\item First, using the discussion from \ref{equivalent} it is easy to verify that any $\mathfrak{M} \in \mathcal{L}^{\leq p}_{\operatorname{crys}}(V_{\mathbb{F}})$ is of the form $f_*\mathfrak{N}$ for an $\mathfrak{N} \in \mathcal{L}^{\leq p}(W_{\mathbb{F}})$. Of course this is only true because $K= \mathbb{Q}_p$. Thus there exists an $\mathbb{F}[[u]]$-basis of $\mathfrak{M}$ with respect to which the matrix of $\varphi$ is given by $A = \left(\begin{smallmatrix}
		0 & xu^s \\ xu^r & 0
		\end{smallmatrix}\right)$ with $r,s \in [0,p]$ not equal and $x \in \mathbb{F}^\times$.
		\item Next take $X =\left( \begin{smallmatrix}
	a & b \\ c & d
	\end{smallmatrix} \right)$ and compute
	$$
	A\varphi(X)A^{-1} - X = \begin{pmatrix}
	\varphi(d) - a & u^{s-r}  \varphi(c) - b \\ u^{r-s} \varphi(b) - c & \varphi(a) - d
	\end{pmatrix}
	$$
	We want to verify the condition in \ref{matrices}, so assume $A \varphi(X) A^{-1} - X \in \operatorname{Mat}(\mathbb{F}[[u]])$. It is easy to see this implies $a,d \in \mathbb{F}[[u]]$. For $c$ and $b$ we may assume that $r>s$, otherwise interchange $c$ and $b$ in our argument. Let $x_c$ and $x_b$ denote the $u$-adic valuations of $c$ and $b$ respectively and assume $x_c < 0$. Then $r-s + p x_b = x_c$, and so $x_b < 0$ also. This implies $s-r + px_c = x_b$ and so $p(x_b+x_c) = (x_c+ x_b)$, a contradiction. Now assume $x_c \geq 0$ and $x_b < 0$. We still have $s-r + px_c = x_b$ and so, as $s-r \geq -p$, we must have $x_c = 0$. Thus $s-r = x_b$. On the other hand we see that $r-s + px_b \geq 0$ which is another contradiction. We conclude that $X \in \operatorname{Mat(\mathbb{F}[[u]])}$, and so in this case $\mathcal{L} \otimes_{R} \mathbb{F}$ is zero-dimensional and reduced.
	\item Finally we argue that $\mathcal{L}^\mathbf{v}$ contains at most one closed point. One computes that for $\mathfrak{M}$ as above $\mathcal{G}^i(\mathfrak{M})$ is zero unless $i = r$ or $s$, in which case it is one-dimensional. This implies that the Breuil--Kisin modules associated to two closed point of $\mathcal{L}^{\mathbf{v}}$ must be abstractly isomorphic as Breuil--Kisin modules; they must therefore be equal since an abstract $\varphi$-equivariant isomorphism induces a $G_{K_\infty}$-equivariant automorphism of $V_{\mathbb{F}}$, and these are all given by scalar multiplication, since $V_{\mathbb{F}}$ is irreducible.
	\end{itemize} 
\end{example}
\begin{example}
	Continue to assume that $K= \mathbb{Q}_p$ and that $\mathbb{F}$ is sufficiently large. We can then also treat the two dimensional reducible case $V_{\mathbb{F}} \sim \left(  \begin{smallmatrix}
	\chi_1 & c \\ 0 & \chi_2 
	\end{smallmatrix} \right)$, at least as long as $\chi_1 \chi_2^{-1} \neq \chi_{\operatorname{cyc}}$ so that $V_{\mathbb{F}}$ is cyclotomic-free. We can compute that $R^{\mathbf{v}}_{\operatorname{crys}}$ is either zero or is formally smooth over $\mathbb{Z}_p$, except in the following exceptional cases:
	\begin{itemize}
		\item When $\mathbf{v} = (a,a+p-1)$ and $V_{\mathbb{F}}$ is a split extension with $\chi_1\chi_2^{-1}$ equal to a non-trivial unramified character, $R^{\mathbf{v}}_{\operatorname{crys}}$ is either zero or has two irreducible components, each of which is formally smooth.
		\item When $\mathbf{v} = (a,a+p-1)$, $\chi_1\chi_2^{-1} =1$ and the cocycle $c(\sigma)$ is ramified, $R^{\mathbf{v}}_{\operatorname{crys}}$ is either zero or formally smooth. If $c(\sigma)$ is unramified then we are only able to deduce that $R^{\mathbf{v}}$ has a single irreducible component (\cite{Sand} computes these rings directly in this case; they are not formally smooth, in fact they are not even Cohen--Macaulay). 
	\end{itemize} 
	Let us only explain the claims in the last bullet point. We leave the rest as an exercise for the interested reader (the arguments are more straightforward). After twisting we may suppose $V_{\mathbb{F}}$ admits an $\mathbb{F}$-basis $\xi_{\mathbb{F}}$ so that $\sigma(\xi_{\mathbb{F}}) = \xi_{\mathbb{F}} \left( \begin{smallmatrix}
	1 & c(\sigma) \\ 0 & 1
	\end{smallmatrix}\right)$. Then the etale $\varphi$-module associated to $V_{\mathbb{F}}$, viewed as a sub-module of $V_{\mathbb{F}} \otimes_{\mathbb{F}_p} C^\flat$,  is generated by $\xi :=\xi_{\mathbb{F}}\left( \begin{smallmatrix}
	1 & \alpha \\ 0 & 1
	\end{smallmatrix} \right)$ where $\alpha \in C^\flat$ is such that $\sigma(\alpha) - \alpha = c(\sigma)$ for $\sigma \in G_{K_\infty}$.
	\begin{itemize}
		\item  Note that $\varphi(\alpha) - \alpha \in \mathbb{F}((u))$. Note also that $\alpha$ is only well-defined up to translation by elements of $\mathbb{F}((u))$. This allows us to assume that $\varphi(\alpha) - \alpha = \alpha_0 + \alpha_{-1}u^{-1} + \ldots + \alpha_{-n}u^{-n}$ for some $\alpha_i \in \mathbb{F}$ and some $p > n \geq 0$. Let us choose $\alpha$ so that $n$ is minimal. If $\sigma \mapsto c(\sigma)$ is unramified then we can clearly take $\alpha \in \overline{k} \otimes_{\mathbb{F}_p} \mathbb{F}$ and in this case $n = 0$. Conversely if $n = 0$ then $\alpha \in \overline{k} \otimes_{\mathbb{F}_p} \mathbb{F}$ and so $\sigma \mapsto c(\sigma)$ is unramified.
		\item We first compute the set of $\mathfrak{M} \in \mathcal{L}^{\leq p}(V_{\mathbb{F}})$. Any such $\mathfrak{M}$ is generated by $\xi B$ for some $B \in \operatorname{GL}_2(\mathbb{F}((u)))$. Using the Iwasawa decomposition for $\operatorname{GL}_2(\mathbb{F}((u)))$ we may assume $B= \left( \begin{smallmatrix}
		u^r & b \\ 0 & u^s
		\end{smallmatrix}\right)$ for some $b \in \mathbb{F}((u))$. With respect to $\xi B$ the matrix of $\varphi$ is given by 
		$$
		A = \begin{pmatrix}
		u^{(p-1)r} & u^{-r} ( \varphi(b) - u^{(p-1)s}b + u^{ps}(\varphi(\alpha) - \alpha)) \\ 0 & u^{(p-1)s}
		\end{pmatrix}
		$$
		Thus $r,s \in [0,1]$. When $r = s = 0$ we must have $\varphi(b) - b + \varphi(\alpha) - \alpha \in \mathbb{F}[[u]]$, which is only possible if $\varphi(\alpha) - \alpha \in \mathbb{F}[[u]]$ and $b \in \mathbb{F}[[u]]$. Thus when $c(\sigma)$ is unramified we obtain a single element of $\mathcal{L}^{\leq p}(V_{\mathbb{F}})$, and otherwise this case contributes no elements. For $r = 0, s = 1$, we must have $\varphi(b) - u^{p-1}b + u^p(\varphi(\alpha) - \alpha) \in \mathbb{F}[[u]]$. This occurs if and only if $b \in \mathbb{F}[[u]]$ so in this case we obtain a single element of $\mathcal{L}^{\leq p}(V_{\mathbb{F}})$. If $ r =1, s = 0$ then we must have $\varphi(b) - b + \varphi(\alpha) - \alpha \in u\mathbb{F}[[u]]$. We see this is only possible if $\varphi(\alpha) - \alpha = 0$ (i.e. $c(\sigma) = 0$) and $b \in \mathbb{F}[[u]]$. In this case we obtain multiple elements of $\mathcal{L}^{\leq p}(V_{\mathbb{F}})$, one for every $b \in \mathbb{F}$. Finally, if $r = s = 1$ then we must have $\varphi(b) - u^{p-1} b + u^p(\varphi(\alpha) - \alpha) \in u \mathbb{F}[[u]]$, and one sees that this case contributes one element to $\mathcal{L}^{\leq p}(V_{\mathbb{F}})$.
		\item We assert that each of the $\mathfrak{M} \in \mathcal{L}^{\leq p}(V_{\mathbb{F}})$ from the previous bullet point lies in $\mathcal{L}^{\leq p}_{\operatorname{crys}}(V_{\mathbb{F}})$. This can be done by first checking each is strongly divisible (which is easy to do by hand). Then use Lemma~\ref{lift} to deduce each is contained in $\mathcal{L}^{\leq p}_{\operatorname{crys}}(V_{\mathbb{F}})$ where $V_{\mathbb{F}}$ is equipped with \emph{some} $G_K$-action extending the $G_{K_\infty}$-action. Finally use that there is at most one way to extend the $G_{K_\infty}$ action on $V_{\mathbb{F}}$ to a $G_K$-action, because $V_{\mathbb{F}}$ is cyclotomic-free.
		\item Let us now focus on the case $\mathbf{v} = (0,p-1)$. We first suppose $c(\sigma)$ is non-zero. From the above $\mathcal{L}^{\mathbf{v}}$ consists of one closed point $\mathfrak{M}$ admitting a basis on which $\varphi$ acts by $\left( \begin{smallmatrix}
		1 & u^p\beta \\ 0 & u^{p-1}
		\end{smallmatrix} \right)$ where $\beta  = \varphi(\alpha) - \alpha$. To compute the tangent vectors of $\mathcal{L} \otimes_{R} \mathbb{F}$ around this point take $X = \left( \begin{smallmatrix}
		 a & b \\ c & d 
		 \end{smallmatrix} \right)$. Then $A\varphi(X)A^{-1} - X$ equals
		 $$
		 \begin{pmatrix}
		 \varphi(a) - a +  u^p\beta \varphi(c) & u^{1-p} \left( u^p\beta ( \varphi(d) - \varphi(a) - u^p\beta \varphi(c)) + \varphi(b) - u^{p-1} b \right) \\
		  u^{p-1}\varphi(c) - c & -u^p\beta \varphi(c) + \varphi(d) - d
		 \end{pmatrix}
		 $$
		 Assume that $A\varphi(X)A^{-1} - X \in \operatorname{Mat}(\mathbb{F}[[u]])$. From $u^{p-1}\varphi(c) - c \in \mathbb{F}[[u]]$ we deduce $c$ has $u$-adic valuation $\geq -1$. If this valuation is $-1$ then $\varphi(a) - a + u^p\beta \varphi(c) \in \mathbb{F}[[u]]$ implies $v(\varphi(a)- a) = v(\beta)$. By construction $p<v(\beta)\leq 0$ while $v(\varphi(a)-a) = pv(a)$ unless $a \in \mathbb{F}[[u]]$, so we must have $a,\beta \in \mathbb{F}[[u]]$. Thus $c \not\in \mathbb{F}[[u]]$ implies $c(\sigma)$ is unramified. Regardless of the valuation of $c$ we see $\alpha \in \mathbb{F}[[u]]$. Similarly $d \in \mathbb{F}[[u]]$, and also $b \in \mathbb{F}[[u]]$. We conclude that if $c(\sigma) \neq 0$ then $\mathcal{L}^{\mathbf{v}} \otimes_{R} \mathbb{F}$ is a reduced point when $c(\sigma)$ is ramified and a non-reduced point if $c(\sigma)$ is unramified.
		 \item Finally we consider when $c(\sigma) = 0$. In this case we may construct a morphism $\mathbb{A}^1_{\mathbb{F}} \rightarrow \mathcal{L}^{\mathbf{v}}$ given by the element $\mathfrak{M}_{\mathbb{F}[T]} \in \mathcal{L}^{\leq p}_{\operatorname{crys}}(V_{\mathbb{F}} \otimes_{\mathbb{F}} \mathbb{F}[T])$ with basis 
		 $$
		 \xi\begin{pmatrix}
		 u & T \\ 0 & 1
		 \end{pmatrix}
		 $$
		 for $\xi$ as above. Since $\mathcal{L}^{\mathbf{v}}$ is projective this morphism extends to a morphism $\mathbb{P}^1_{\mathbb{F}} \rightarrow \mathcal{L}^{\mathbf{v}}$ which is an isomorphism. 
 	\end{itemize}
\end{example}

These calculations recover those of \cite[1.7.14]{KisFM}, see also \cite{Sand}.
\bibliography{/home/user/Dropbox/Maths/biblio}

\end{document}